\documentclass{amsart}	
\usepackage{mathrsfs}
\usepackage{mathtools} 
\usepackage{amsthm}
\usepackage{amscd}
\usepackage{amssymb}
\usepackage{latexsym}
\usepackage{graphicx}
\usepackage{stmaryrd}
\usepackage{xypic}
\usepackage{cancel}
\xyoption{curve}
\usepackage{nicefrac}
\usepackage{wasysym}
\usepackage{enumitem}
\usepackage{setspace}
\usepackage{ifthen}
\usepackage{verbatim}
\usepackage{eufrak}
\usepackage{relsize} 
\usepackage{tikz-cd}
\usepackage{stackengine} 
\usepackage{relsize} 
\usepackage{adjustbox}
\usepackage{textgreek} 
\usepackage[margin=1.5in]{geometry}
\usepackage[colorinlistoftodos,prependcaption,textsize=tiny]{todonotes}

\input{notation.tex} 

\newtheorem{thm}{Theorem}[section]
\newtheorem{cor}[thm]{Corollary}
\newtheorem{lem}[thm]{Lemma}

\newtheorem{prop}[thm]{Proposition}

\newtheorem*{lem*}{Lemma}
\newtheorem*{thm*}{Theorem}

\theoremstyle{definition}
\newtheorem{defn}[thm]{Definition}

\newtheorem{eg}[thm]{Example}

\newcommand{\cohomology}[1]{\catfont{H}^1}
\newcommand{\cubcatreplacement}{\catfont{R}}
\newcommand{\fundamentalcat}{\Tau_1}
\newcommand{\fundamentalgroupoid}{\Pi_1}
\newcommand{\STARS}{\catfont{Stars}}
\newcommand{\ATOMICS}{\catfont{Atoms}}

\newcommand{\DITOPOLOGICALBOX}{\blacksquare}

\newcommand{\groupoid}[1]{\mathcal{G}}

\newcommand{\dhomotopic}[1]{\leftrightsquigarrow_{#1}}
\newcommand{\topologicaldhomotopic}{\leftrightsquigarrow_{\direalize{\;\mathfrak{d}\;}}}
\newcommand{\intervalobject}[1]{\mathfrak{i}}
\newcommand{\futurehomotopic}[1]{\leadsto_{#1}}
\newcommand{\FIBRANTCUBICALSETS}{\infty\GROUPOIDS}
\newcommand{\INTERVALCAT}{\BOX_1}
\newcommand{\POTOP}{\catfont{Pos}}
\newcommand{\DISLATS}{\catfont{Dis}}
\newcommand{\CUBICALSETS}{\hat\BOX}
\newcommand{\SIMPLICIALSETS}{\hat\DEL}
\newcommand{\pls}{\raisebox{.4\height}{\scalebox{.6}{+}}}
\newcommand{\mins}{\raisebox{.4\height}{\scalebox{.8}{-}}}

\newcommand{\dihomeo}{\varphi} 
\newcommand{\csing}{\sing\,}
\newcommand{\cnerve}{\nerve\,}
\newcommand{\snerve}{\catfont{ner}_\DEL}

\newcommand{\cset}[1]{\ifthenelse{1=#1}{C}{\ifthenelse{2=#1}{A}{B}}}

\newcommand{\REGULARCUBICALSETS}{\CUBICALSETS_{\catfont{REG}}}

\newcommand{\THMPartOfTestFunctors}{Part of Theorem 4.1.26, \cite{cisinskiprefaisceaux}}
\newcommand{\THMStrictCubicalCategories}{Theorem 8.8, \cite{agl2002multiple}}
\newcommand{\THMModularForbiddenMinors}{Theorem, \cite{dedekind1900drei}}
\newcommand{\PROPUnderlyingMonoidalPreorderedSets}{Proposition 5.11, \cite{krishnan2009convenient}}
\newcommand{\ThmDiEmbed}{Theorem 2.5, \cite{fernandes2007classification}}

\newcommand{\PROPTop}{Proposition 5.8, \cite{krishnan2009convenient}}

\newcommand{\PROPInclusions}{Proposition 5.6, \cite{krishnan2015cubical}}
\newcommand{\PropSD}{Proposition 7.4, 7.5, \cite{krishnan2015cubical}}
\newcommand{\THMPospaces}{Theorem 4.7, \cite{krishnan2009convenient}}
\newcommand{\PROPLocCompactStreams}{Theorem 5.4, \cite{krishnan2009convenient}}
\newcommand{\ThmDHomotopy}{Theorem 7.1, \cite{krishnan2015cubical}}
\newcommand{\THMVanKampen}{Theorem 1, \cite{krishnan2015cubical}}

\newcommand{\THMXClosed}{Theorem 5.12, \cite{krishnan2009convenient}}
\newcommand{\THMQCubical}{Theorem 4.18, \cite{jardine2006categorical}}
\newcommand{\THMTriEquivalence}{Theorem, \cite{jardine2002cubical}}
\newcommand{\CORTriEquivalence}{Theorem, \cite{jardine2002cubical}}
\newcommand{\THMRealizeEquivalence}{Theorem, \cite{jardine2002cubical}}
\newcommand{\THMAMS}{Theorem, \cite{jardine2002cubical}}
\newcommand{\THMTransfer}{Theorem 3.10, \cite{riehl2011algebraic}}
\newcommand{\THMPointwise}{Theorem 4.3, \cite{riehl2011algebraic}}
\newcommand{\THMClassicalEquivalence}{Theorem 4.3, \cite{riehl2011algebraic}}

\newcommand{\LEMQTCocontinuous}{Lemma 7.2, \cite{krishnan2015cubical}}
\newcommand{\LEMSdTri}{Lemma 7.2, \cite{krishnan2015cubical}}
\newcommand{\LEMMincubicalMonics}{Lemma 6.2, \cite{krishnan2015cubical}}
\newcommand{\EGMSpaces}{Examples 2.2 and 3.8, \cite{cole2006mixing}}
\newcommand{\SimplicialGraphs}{Lemma 5.12, \cite{krishnan2015cubical}}
\newcommand{\THMBirkhoff}{Birkhoff's Representation Theorem}
\newcommand{\THMCat}{Proposition in \S4.2C, {\cite{gromov1987hyperbolic}}}
\newcommand{\LEMBoxInclusions}{Lemma 6.2, \cite{krishnan2015cubical}}
\newcommand{\COREZCellularModel}{Corollary 4.15, \cite{isaacson2011symmetric}}
\newcommand{\PROPTransferredRegularity}{Proposition 6.12, \cite{isaacson2011symmetric}}
\newcommand{\PROPTestFunctor}{Proposition 1.2.9, \cite{maltsiniotis2005theorie}}
\newcommand{\THMOldQuillenEquivalence}{Theorem 8.4.38, \cite{cisinskiprefaisceaux}}
\newcommand{\THMFroibenius}{Theorem 4.8, \cite{gambino2017frobenius}}
\newcommand{\PROPGrothendieckTestCriterion}{Proposition, p.86 44(d), \cite{grothendieck1983pursuing}}
\newcommand{\THMTestModelStructure}{Theorem 1.4.3, \cite{cisinskiprefaisceaux}}
\newcommand{\THMIntervalicDistributivityCharacterization}{Theorem, \cite{lazaraz2015note}}
\newcommand{\THMCATZEROCharacterization}{Theorem 8, \cite{lazaraz2015note}}
\newcommand{\PROPNaturalApproximations}{Proposition 8.11, \cite{krishnan2015cubical}}
\newcommand{\THMHomotopyIsDihomotopy}{Theorem 8, \cite{goubault2020directed}}
\newcommand{\THMFiniteModularity}{Theorem, \cite{dedekind1900drei}}
\newcommand{\THMFiniteDistributivity}{Theorem, \cite{dedekind1900drei}}
\newcommand{\THMDiamondIsomorphism}{Diamond Isomorphism Theorem}
\newcommand{\THMPrototypicalNonDistributiveLattice}{Forbidden Distributive Minor Theorem}
\newcommand{\THMPrototypicalNonModularLattice}{Forbidden Modular Minor Theorem}
\newcommand{\THMGeodesicCriteria}{Theorem, \cite{goubault2020directed}}
\newcommand{\THMCATZeroPosetCharacterization}{Theorem 2.5, \cite{ardila2012geodesics}}
\newcommand{\THMCATZeroGeodesicCharacterization}{Theorem 5.8, \cite{ardila2012geodesics}}
\newcommand{\PROPDistributiveSemilatticeRepresentations}{Propositions 5.7,5.8, \cite{gonzalez2021finite}}
\newcommand{\THMCATZeroUniqueGeodesics}{Theorem 5.5, \cite{bridson99metric}}
\newcommand{\PROPGeodesicsAreMinimizing}{\S II.1, Proposition 1.4, \cite{bridson99metric}}
\newcommand{\PROPSemilatticeBirkhoffRepresentation}{Propositions 5.7, 5.8, \cite{{gonzalez2021finite}}}

\newcommand{\THMProDiagrams}{Theorem, \S 3, \cite{meyer1980approximation}}

\theoremstyle{plain}

\newtheorem*{cor:non.equivariant.derived.formula}{Theorem {\ref{thm:derived.formula}} [Case $\indexcat{1}=\star$]}
\newtheorem*{prop:non.equivariant.mapping.cubcats}{Proposition {\ref{prop:mapping.cubcats}} [Case $\indexcat{1}=\star$]}
\newtheorem*{prop:non.equivariant.singular.cubcats}{Proposition {\ref{prop:singular.cubcats}} [Case $\indexcat{1}=\star$]}
\newtheorem*{prop:non.equivariant.nerves}{Proposition {\ref{prop:nerves}} [Case $\indexcat{1}=\star$]}
\newtheorem*{prop:singular.cubcats}{Proposition {\ref{prop:singular.cubcats}}}
\newtheorem*{prop:nerves}{Proposition {\ref{prop:nerves}}}
\newtheorem*{prop:fibrant.cubcats}{Proposition {\ref{prop:fibrant.cubcats}}}
\newtheorem*{thm:simple.cubcat.replacement}{Special Case of Theorem {\ref{thm:approx}}}
\newtheorem*{thm:non.equivariant.equivalence}{Theorem {\ref{thm:equivalence}} [Case $\indexcat{1}=\star$]}
\newtheorem*{thm:non.equivariant.formula}{Corollary {\ref{cor:formula}} [Case $\indexcat{1}=\star$]}
\newtheorem*{thm:fibrant.cubcats}{Theorem {\ref{thm:fibrant.cubcats}}}
\newtheorem*{thm:approx}{Theorem {\ref{thm:approx}}}
\newtheorem*{cor:cohomological.calculations}{Corollary \ref{cor:cohomological.calculations}}

\newtheorem*{thm:part.of.test-functors}{\THMPartOfTestFunctors{}}
\newtheorem*{thm:pro-diagrams}{\THMProDiagrams{}}
\newtheorem*{thm:strict.cubical.categories}{\THMStrictCubicalCategories{}}
\newtheorem*{thm:modular.forbidden.minors}{\THMModularForbiddenMinors{}}
\newtheorem*{prop:semilattice.birkhoff.representation}{\PROPSemilatticeBirkhoffRepresentation}
\newtheorem*{prop:geodesics.are.minimizing}{\PROPGeodesicsAreMinimizing}
\newtheorem*{thm:cat0.unique.geodesics}{\THMCATZeroUniqueGeodesics}
\newtheorem*{prop:distributive.semilattice.representations}{\PROPDistributiveSemilatticeRepresentations}
\newtheorem*{thm:partial.cat0.geodesic.characterization}{Part of \THMCATZeroGeodesicCharacterization}
\newtheorem*{thm:partial.cat0.poset.characterization}{Part of \THMCATZeroPosetCharacterization}
\newtheorem*{thm:geodesic.criteria}{\THMGeodesicCriteria}
\newtheorem*{thm:prototypical.non-distributive.lattice}{\THMPrototypicalNonDistributiveLattice}
\newtheorem*{thm:prototypical.non-modular.lattice}{\THMPrototypicalNonModularLattice}
\newtheorem*{thm:diamond.isomorphism}{\THMDiamondIsomorphism}
\newtheorem*{thm:intervalic.distributivity.characterization}{\THMIntervalicDistributivityCharacterization}
\newtheorem*{thm:cat0.characterization}{\THMCATZEROCharacterization}
\newtheorem*{thm:test.model.structure}{\THMTestModelStructure}
\newtheorem*{prop:Grothendieck.test.criterion}{\PROPGrothendieckTestCriterion}
\newtheorem*{prop:hemi-metrics}{Proposition 3.7, \cite{goubault2020directed}}
\newtheorem*{thm:nachbin}{Theorem 5, \cite{nachbin1976topology}}
\newtheorem*{cor:sd}{Corollary {\ref{cor:sd}}}
\newtheorem*{cor:gammas}{Corollary {\ref{cor:gammas}}}
\newtheorem*{thm:birkhoff}{\THMBirkhoff}
\newtheorem*{cor:ez.cellular.model}{\COREZCellularModel}
\newtheorem*{prop:transferred.regularity}{\PROPTransferredRegularity}
\newtheorem*{prop:test.functor}{\PROPTestFunctor}
\newtheorem*{thm:old.quillen.equivalence}{\THMOldQuillenEquivalence}
\newtheorem*{thm:froibenius}{\THMFroibenius}
\newtheorem*{thm:homotopy.is.dihomotopy}{\THMHomotopyIsDihomotopy}
\newtheorem*{prop:q-model-structure}{Proposition \S2.3, \cite{quillen1967homotopical}}
\newtheorem*{prop:normal.cube.paths}{Proposition 3.21, \cite{quillen1967homotopical}}

\newcommand{\THMSimplicialClassicalEquivalence}{Theorem \S2.3, \cite{quillen1967homotopical}}
\newtheorem*{thm:simplicial.classical.equivalence}{\THMSimplicialClassicalEquivalence}

\newcommand{\EGSSets}{Examples 3.6 and 4.9, \cite{gambino2017frobenius}}
\newtheorem*{eg:ssets}{\EGSSets}

\newtheorem*{eg:m-spaces}{\EGMSpaces}
\newtheorem*{lem:mincubical.monics}{\LEMMincubicalMonics}
\newtheorem*{thm:x-closed}{\THMXClosed}
\newtheorem*{thm:pointwise}{\THMPointwise}
\newtheorem*{thm:transfer}{\THMTransfer}
\newtheorem*{thm:small-object-argument}{\THMTransfer}

\newtheorem*{prop:natural.approximations}{\PROPNaturalApproximations}
\newtheorem*{thm:cat}{\THMCat}
\newtheorem*{lem:sd.tri}{\LEMSdTri}
\newtheorem*{lem:qt-cocontinuous}{\LEMQTCocontinuous}
\newtheorem*{prop:underlying.monoidal.preordered.sets}{\PROPUnderlyingMonoidalPreorderedSets}
\newtheorem*{prop:locally.compact.streams}{\PROPLocCompactStreams}
\newtheorem*{thm:pospaces}{\THMPospaces}
\newtheorem*{thm:diembed}{\ThmDiEmbed}
\newtheorem*{thm:van-kampen}{\THMVanKampen}
\newtheorem*{thm:q-cubical}{\THMQCubical}
\newtheorem*{thm:tri.equivalence}{\THMTriEquivalence}
\newtheorem*{cor:tri.equivalence}{\CORTriEquivalence}
\newtheorem*{thm:realize.equivalence}{\THMRealizeEquivalence}
\newtheorem*{thm:ams}{\THMAMS}
\newtheorem*{eg:mixing}{Example 2.2 from \cite{cole2006mixing}}
\newtheorem*{thm:classical-equivalence}{\THMClassicalEquivalence}
\newtheorem*{prop:hp}{Proposition \ref{prop:hp}}
\newtheorem*{cor:types}{Corollary \ref{cor:types}}
\newtheorem*{prop:extensions}{Proposition \ref{prop:extensions}}
\newtheorem*{lem:geodesic.metric}{Lemma 3.70, \cite{goubault2020directed}}
\newtheorem*{lem:simplicial.graphs}{\SimplicialGraphs}

\newtheorem*{cor:algebraic-q-cubical}{Corollary \ref{cor:algebraic-q-cubical}}
\newtheorem*{cor:algebraic-m-spaces}{Corollary \ref{cor:algebraic-m-spaces}}

\newtheorem*{thm:m-streams}{Theorem {\ref{thm:m-streams}}}
\newtheorem*{thm:cubical.dihomotopy}{Theorem {\ref{thm:cubical.dihomotopy}}}
\newtheorem*{cor:cubical.dihomotopy}{Corollary {\ref{cor:cubical.dihomotopy}}}

\newtheorem*{prop:cubical.homotopies}{Proposition {\ref{prop:cubical.homotopies}}}
\newtheorem*{prop:nerve.cubcats}{Proposition {\ref{prop:nerve.cubcats}}}
\newtheorem*{prop:kan.cubcat}{Proposition {\ref{prop:kan.cubcat}}}
\newtheorem*{cor:kan.bicubcat}{Corollary {\ref{cor:kan.bicubcat}}}
\newtheorem*{thm:fibrant}{Theorem {\ref{thm:fibrant}}}
\newtheorem*{prop:semifibrant}{Proposition {\ref{prop:semifibrant}}}
\newtheorem*{thm:equivalence}{Corollary {\ref{thm:equivalence}}}
\newtheorem*{cor:type-theoretic}{Corollary {\ref{cor:type-theoretic}}}
\newtheorem*{thm:q-spaces}{Theorem from \cite[\S II.3]{quillen1967homotopical}}
\newtheorem*{cor:q-equivalences}{Corollary {\ref{cor:q-equivalences}}}
\newtheorem*{cor:m-equivalences}{Corollary {\ref{cor:m-equivalences}}}
\newtheorem*{cor:excision}{Corollary {\ref{cor:excision}}}
\newtheorem*{cor:cubical.diequivalence}{Corollary {\ref{cor:cubical.diequivalence}}}
\newtheorem*{thm:cartesian.closed.streams}{\THMXClosed}
\newtheorem*{thm:whitehead}{Theorem \ref{thm:whitehead}}
\newtheorem*{prop:topological}{\PROPTop}
\newtheorem*{prop:sd}{\PropSD}
\newtheorem*{prop:inclusions}{\PROPInclusions}
\newtheorem*{thm:d-homotopy}{\ThmDHomotopy}
\newtheorem*{thm:finite.modularity}{\THMFiniteModularity}
\newtheorem*{thm:finite.distributivity}{\THMFiniteDistributivity}
\newtheorem*{lem:box-inclusions}{\LEMBoxInclusions{}}
\newtheorem*{prop:distributivity.boolean.generated}{Proposition \ref{prop:distributivity.boolean.generated}}

\title{Cubical Approximation for Directed Topology II}
\author{Sanjeevi Krishnan}

\begin{document}
\begin{abstract}
  The paper establishes an equivalence between directed homotopy categories of (diagrams of) cubical sets and (diagrams of) directed topological spaces.
  This equivalence both lifts and extends an equivalence between classical homotopy categories of cubical sets and topological spaces.  
  Some simple applications include combinatorial descriptions and subsequent calculations of directed homotopy monoids and directed singular $1$-cohomology monoids.
  Another application is a characterization of isomorphisms between small categories up to zig-zags of natural transformations as directed homotopy equivalences between directed classifying spaces.
  Cubical sets throughout the paper are taken to mean presheaves over the minimal symmetric monoidal variant of the cube category. 
  Along the way, the paper characterizes morphisms in this variant as the interval-preserving lattice homomorphisms between finite Boolean lattices.
\end{abstract}

\maketitle
\tableofcontents
\addtocontents{toc}{\protect\setcounter{tocdepth}{1}}

\section{Introduction}\label{sec:introduction}

The qualitative behavior of a complex system often corresponds to features of a \textit{directed topological} state space invariant under \textit{directed homotopy}, continuous deformations respecting the directionality.
For a simple example, the existence of non-constant causal curves is a directed homotopy invariant on a spacetime.
Traditionally, directed homotopy theory has been applied to deduce subtle but critical behavior in simple concurrent computer programs from state spaces modelled as unions of cubes each equipped with their standard product orders.
More recently, classical homotopy colimits of dynamical systems (e.g. \cite{jardine2013homotopy}), when equipped with extra directed structure, have been used to give semantics for a formal logic of hybrid systems \cite{winkowski2007towards}.  

The main challenge in the classification of directed topological spaces up to directed homotopy is that cell-by-cell constructions, ubiquitous in classical homotopy theory, are not possible.    
For example, proofs of cellular approximation, simplicial approximation, and the Whitehead Theorem involve constructing some kind of map on a cell complex one cell at a time; such constructions are possible because the inclusion of a boundary of a cube into a cube satisfies the homotopy extension property.  
In contrast, inclusions of directed topological spaces in nature almost never satisfy directed homotopy extension properties [Figure \ref{fig:hep.failure}].  
In intuitive terms, a seemingly minor local deformation in a directed topological state space $X$ can sometimes drastically affect the global behavior of the system that $X$ represents.
In homotopical parlance, directed topological state spaces almost never decompose into homotopy colimits, with respect to directed homotopy, of simpler directed state spaces. 
This indecomposability is seen to various degrees in other refinements \cite{chorny2012class,isaksen2001model,krishnan2022uniform,raptis2021bounded} of classical homotopy theory.  

\begin{figure}
		\includegraphics[width=3in,height=1.5in]{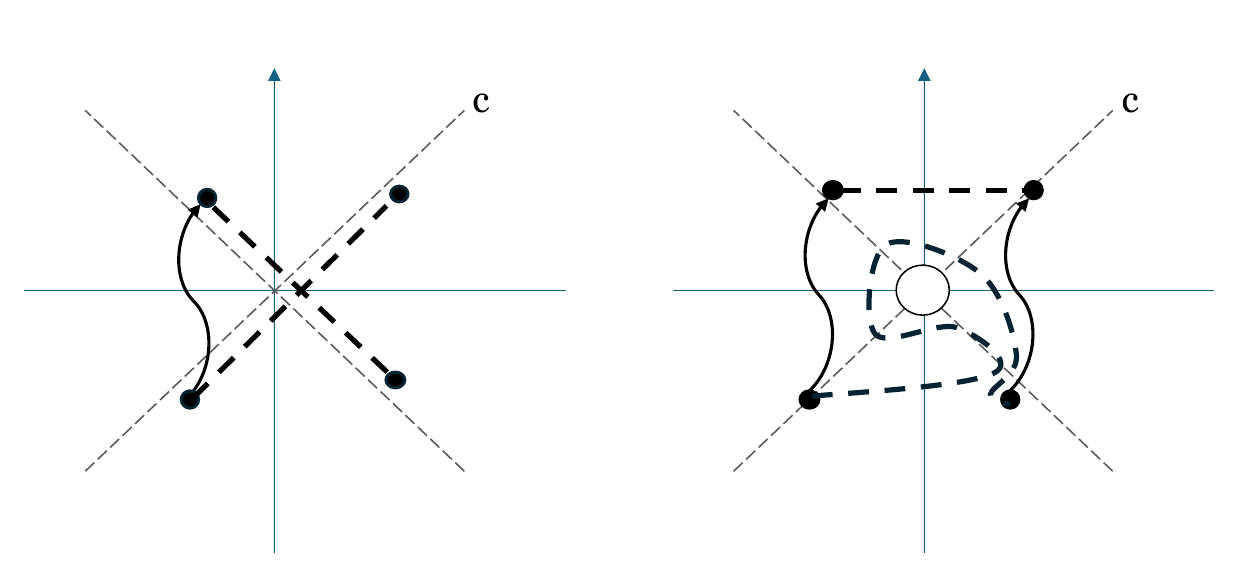}
	  \caption{
	    {\bf Directed Homotopy Inextendability}
	    Depicted above are Minkowski $(1+1)$-spacetimes.
	    Thin dotted diagonal lines represent the speed of light.    
	    Curves represent causal paths, smooth paths whose derivative always lies in the future-facing part of the closed light cone.  
	    In the left picture, a dotted homotopy of the restriction of the causal path to its endpoints extends to a classical homotopy from the causal path, but not to another causal path.  
	    In the right picture, a dotted homotopy of the restriction of the left causal path to its endpoints extends to a classical homotopy from one causal path to another causal path, but not \textit{through} such causal paths; a homotopy between such causal paths must go through a path that travels faster than the speed of light.    
	    A directed homotopy is, in particular, a homotopy \textit{through} directed maps, such as causal paths.  
	    }
	    \label{fig:hep.failure}
\end{figure}

Nonetheless, directed topological spaces in nature can often be encoded by combinatorial data. 
Closed $(1+1)$-spacetimes and the state spaces of concurrent computer programs arise as \textit{directed realizations} $\direalize{C}$ of \textit{cubical sets} $C$, abstract recipes for gluing cubes together.  
Some salient causal and conformal structure of higher dimensional spacetimes $X$ can be captured in their \textit{directed singular cubical sets} $\sing\,X$, directed and cubical analogues of singular complexes.   
What we want to understand is the \textit{directed homotopy category}, the category whose objects are cubical sets but whose morphisms are directed homotopy classes of maps between their directed realizations.
Existing work in that direction are cubical approximation theorems, combinatorial formulas for hom-sets of morphisms $A\ra B$ in the directed homotopy category, when either $A$ is a directed graph \cite[Theorem 4.1]{fajstrup2005dipaths} or $A$ is finite and $B$ satisfies a simplicial-like condition \cite[Corollary 8.2]{krishnan2015cubical}.
And what we want to eventually understand are higher derived variants of that directed homotopy category.     
Existing work in that direction is a prod-simplicial approximation of spaces of directed paths relative endpoints in directed realizations of interest in the applications \cite[Theorem 3.5]{raussen2010simplicial}.

\subsection{Cubical approximation}
The main results give equivalent combinatorial and directed topological descriptions of the directed homotopy category. 
In fact, this equivalence is stated and proven more generally as an \textit{equivariant} equivalence [Theorem \ref{thm:equivalence}], where all directed topological spaces and cubical sets are replaced by diagrams indexed by some small category $\indexcat{1}$; notably, a classical counterpart to this equivariant equivalence does not follow from existing, non-algebraic (cf. \cite{riehl2011algebraic}) Quillen equivalences.  
For simplicity we state the non-equivariant case, that directed realization $\direalize{-}$ from the category $\CUBICALSETS$ of cubical sets to the category $\DITOP$ of directed topological spaces passes to an equivalence $d(\CUBICALSETS)\simeq d(\DITOP)$ of directed homotopy categories.

\begin{thm:non.equivariant.equivalence}
  There exist dotted vertical localizations in the diagram
  \begin{equation*}
	  \begin{tikzcd}
		  \CUBICALSETS\ar{r}[above]{\direalize{\;-\;}}\ar[d,dotted] & \DITOP\ar[d,dotted]\\
		  d(\CUBICALSETS)\ar[dotted]{r}[below]{\simeq} & d(\DITOP)
	  \end{tikzcd}
  \end{equation*}
  by those cubical functions whose directed realizations are directed homotopy equivalences and those directed maps inducing cubical homotopy equivalences between directed singular cubical sets.  
  There exists a dotted horizontal adjoint equivalence making the entire diagram commute.  
\end{thm:non.equivariant.equivalence}

This equivalence has several simple consequences.
One is that an existing kind of directed equivalence between small categories \cite{minian2002cat}, intermediate in generality between Thomason weak equivalences and categorical equivalences, can be characterized as a functor inducing a directed homotopy equivalence between \textit{directed classifying spaces} [Theorem \ref{thm:1-ditypes}].
Another is a combinatorial description and subsequent calculation of \textit{directed homotopy monoids}, directed analogues of homotopy groups, on based directed classifying spaces [Corollaries \ref{prop:homotopy.calculations} and \ref{cor:directed.milnor}].
Yet another is a combinatorial description and subsequent calculation of \textit{singular directed $1$-cohomology monoids} [Proposition \ref{prop:cohomological.calculation} and Examples \ref{eg:torus.calculation} and \ref{eg:klein.calculation}], directed generalizations of singular $1$-cohomology groups beyond the setting of Abelian group coefficients, that give a functorial, causal and conformal invariant on spacetimes (cf. \cite{benini2016optimal,khavkine2016cohomology}).

What makes such calculations practical is a recognition of when a cubical set is something we term a \textit{cubcat}.  
Cubcats are cubical sets admitting extra unary operations on cubes parametrized by monotone maps between topological cubes and compatible composition operations on successive pairs of cubes [Definition \ref{defn:cubcats}].
Cubcats are directed analogues of \textit{fibrant} cubical sets, fibrant objects in a model structure on cubical sets Quillen equivalent to the usual model structure on topological spaces along realization and hence cubical models of classical weak types.  
Singular cubical sets of topological spaces and \textit{cubical nerves} of small groupoids are examples of fibrant cubical sets. 
Cubcats at once generalize small categories, directed analogues of singular complexes, and higher groupoids. 

\begin{prop:non.equivariant.nerves}
  Cubical nerves of small categories are cubcats.
\end{prop:non.equivariant.nerves}

\begin{prop:non.equivariant.singular.cubcats}
  Directed singular cubical sets are cubcats.
\end{prop:non.equivariant.singular.cubcats}

\begin{prop:fibrant.cubcats}
  Fibrant cubical sets are cubcats.
\end{prop:fibrant.cubcats}

Every cubical set can be approximated by a cubcat without changing the directed homotopy type of its directed realization [Theorem \ref{thm:equivalence} and Proposition \ref{prop:singular.cubcats}].  
Cubcats are convenient because of the following combinatorial description of the directed hom-space $\direalize{C}^{\direalize{\;B\;}}$ up to directed homotopy equivalence, a higher derived variant of the hom-set $d(\CUBICALSETS)(B,C)$, when $C$ is a cubcat (c.f. \cite[Theorem 3.5]{raussen2010simplicial}.)  

\begin{cor:non.equivariant.derived.formula}
  For cubcats $C$, natural directed maps of the form
  $$\direalize{C^B}\ra\direalize{C}^{\direalize{\;B\;}}$$
  represent $d(\DITOP)$-isomorphisms.  
\end{cor:non.equivariant.derived.formula}

A simple consequence is a formula for the set $d(\CUBICALSETS)(B,C)$ of directed homotopy classes of directed maps between directed realizations of cubical sets $B$ and $C$, when $C$ is a cubcat.

\begin{thm:non.equivariant.formula}
  For all cubical sets $B$ and cubcats $C$,
  \begin{equation}
    \label{eqn:non.equivariant.formula}
    d(\CUBICALSETS)(B,C)=\pi_0C^B,
  \end{equation}
  the set of cubical homotopy classes of cubical functions $B\ra C$.  
\end{thm:non.equivariant.formula}

Cubical homotopy types of cubcats correspond to directed homotopy types of directed realizations, just as cubical homotopy types of fibrant cubical sets correspond to classical homotopy types of topological realizations.  

\subsection{Intuition}
We can encode the operations of a cubcat by an endofunctor $\cubcatreplacement$.
Concretely, we can take $\cubcatreplacement C$ to be the cubical set of all directed singular cubes in the directed realization $\direalize{C}$ of a cubical set $C$ satisfying a certain local lifting property.   
Intuitively, we can think of each cube in $\cubcatreplacement C$ as a formal application of composition and unary operations to the cubes in $C$.
A cubcat can then be characterized precisely as a cubical set $C$ onto which $\cubcatreplacement\,C$ retracts [Proposition \ref{prop:free.cubcat.monad}].

The main results follow from two parts.
The first shows that inclusions $\cubcatreplacement C\ira\sing\direalize{C}$ admit suitably natural cubical homotopy inverses [Lemma \ref{lem:cubcats.are.sing.algebras}].
The second shows that the comonad of the adjunction $\direalize{-}\dashv\sing$ is directed homotopy idempotent [Lemma \ref{lem:homotopy.idempotent.comonad}].
Throughout the proofs, cell-by-cell constructions have to give way to constructions that are \textit{natural} in each cell and therefore applicable on all cells at once (e.g. Lemmas \ref{lem:hypercube.convexity}, \ref{lem:natural.approximations}, \ref{lem:close.maps}). 

The first part uses techniques developed previously \cite{krishnan2015cubical} and superficially streamlined in this paper [Lemmas \ref{lem:collapse.star}, \ref{lem:star.flower}, and \ref{lem:natural.retractions}].
The key observation is that a natural cubical function $\sd_9C\ra C$ from a $9$-fold edgewise subdivision $\sd_9C$ of a cubical set $C$ locally factors through representables [Lemma \ref{lem:local.lifts}].
Composition of directed singular cubes in $\direalize{C}$ with a directed homotopy equivalence $\direalize{C}\cong\direalize{\sd_9C}\simeq\direalize{C}$ gives the desired cubical homotopy inverse [Lemma \ref{lem:cubcats.are.sing.algebras}].

The second part requires more delicacy.
When $C$ is the cubical analogue of a regular CW complex, the closed cells of $\direalize{C}$ have topological lattice structure that can be used to construct the desired cubical homotopies [Lemma \ref{lem:restricted.homotopy.idempotent.comonad}].  
For general $C$, we want to lift each directed singular cube $\theta$ on $\direalize{C}$ along a directed map $\direalize{C_\theta\ra C}:\direalize{C_\theta}\ra\direalize{C}$, with $C_\theta$ a cubical analogue of a regular CW complex.  
Each directed singular cube $\theta$ on $\direalize{C}$ determines a diagram of cubes that constitutes the support of $\theta$ in $C$.
The strategy is to replace that diagram with one whose colimit satisfies the regularity condition and admits a natural lifting property (cf. \cite{bourke2016algebraic,grandis2006natural,riehl2011algebraic}) that holds not necessarily on the nose but at least up to natural directed homotopy. 

This strategy requires a new technique.  
The desired replacement is reminiscent of cofibrant replacement in a projective diagram model structure satisfying additional naturality constraints (e.g. \cite[Theorem 4.5]{riehl2011algebraic}.)
But the category $\BOX$ of cubes is too small to permit a suitable small object argument \cite{bourke2016algebraic}, implicit in the construction of such cofibrant replacements.
The key idea is to encode all lifts of interest as a single, natural lift [Lemma \ref{lem:star.rlp}] in the \textit{pro-completion} [Section \S\ref{sec:pro-completions}] of $\BOX$.
Coherence results for diagrams in pro-completions [\THMProDiagrams] give the replacement [Lemma \ref{lem:bamfl}].  

\subsection{Setting}
The dependence of the results on the exact setting is summarized here for the interested specialist.  
\textit{Directed topological spaces}, taken to mean \textit{streams} \cite{krishnan2009convenient}, can just as easily be taken to mean \textit{d-spaces} \cite{grandis2003directed}.      
\textit{Directed homotopy} is taken to mean a \textit{d-homotopy} \cite{grandis2003directed} as opposed to an \textit{h-homotopy} (e.g. \cite{krishnan2019hurewicz}), where an \textit{h-homotopy} is a homotopy through directed maps and a \textit{d-homotopy} is an h-homotopy that is also piecewise monotone and anti-monotone in its homotopy coordinate. 
An h-homotopy can always be replaced by a d-homotopy when the domain is compact and the codomain is a directed realization of a cubical set (c.f. \cite[Theorem 8.22]{krishnan2015cubical}).
The category $\BOX$ of cubes is enlarged, from the usual minimal variant used in the predecessor \cite[Theorem 8.22]{krishnan2015cubical} to this paper, to the minimal symmetric monoidal variant.
One convenience is a simple and explicit characterization of the $\BOX$-morphisms [Theorem \ref{thm:box.characterization}].
A subsequent convenience is an order-theoretic construction of cubical edgewise subdivision [Propositions \ref{prop:lattice.subdivision} and \ref{prop:cubical.subdivision}].
However, the most important application is to diagram replacement; its construction relies on the specific fact, unique to our variant $\BOX$ of the cube category, that parallel $\BOX$-epis (parallel projections) are isomorphic (by coordinate permutations) in the arrow category.  
Other technical results along the way [Lemmas \ref{lem:collapse.star}, \ref{lem:star.flower}, \ref{lem:natural.retractions}, \ref{lem:local.lifts}, \ref{lem:hypercube.convexity}, \ref{lem:natural.retractions}, \ref{lem:close.maps}] work for most variants of $\BOX$ in the literature that exclude the reversals and diagonals.  

\begin{figure}
  \begin{equation*}
   \xymatrix{
      \bullet\ar[r] & \bullet\\
      \bullet\ar[r]\ar[u] & \bullet\ar[u]
    }\quad
    \xymatrix{
      & \bullet\\
        \bullet\ar@/^5ex/[ur]\ar@/_5ex/[ur]
    }\quad
   \xymatrix{
      & \bullet\ar@/^5ex/[dl]\\
      \bullet\ar@/^5ex/[ur]
    }\quad
   \xymatrix{
      \bullet\ar[r] & \bullet\ar[d]\\
      \bullet\ar[u] & \bullet\ar[l]
    }
  \end{equation*}
  \caption{
  {\bf Equivalence as different categorical structures}.
  The directed graphs above freely generate equivalent groupoids but freely generate mutually inequivalent categories, some of which are nonetheless directed homotopy equivalent to one another.  
  After passage to free categories, the left two directed graphs are directed homotopy equivalent to one another, the right two directed graphs are directed homotopy equivalent to one another, but the left two and the right two are not directed homotopy equivalent to one another.
  Intuitively, classical equivalences ignore the structure of time in state spaces while categorical equivalences are sensitive to arbitrary subdivisions of time.  
  Directed homotopy sidesteps some of the combinatorial explosion that bedevils geometric models of state spaces sensitive to arbitrary subdivisions in time.
  Section \S\ref{sec:categorical.homotopy} formalizes the different notions of equivalence between small categories.  
    }
  \label{fig:spectrum}
\end{figure}

\subsection{Organization}
Some conventions are fixed in \S\ref{sec:conventions}. 
Point-set theories of directed topological spaces and cubical sets are recalled, further developed, and compared in \S\ref{sec:spaces}.
Homotopy theories, classical, directed, and categorical, are formalized and compared, albeit without some proofs, in \S\ref{sec:homotopy}.
Cubcats are introduced and used to give proofs of the main results in \S\ref{sec:cubcats}. 
The potential role of cubcats in a directed type theory is sketched in \S\ref{sec:conclusion}.  
Some relevant facts about modular lattices, triangulations, and pro-completions are recalled and further developed in Appendices \S\ref{sec:modular.lattices}, \S\ref{sec:triangulations}, and \S\ref{sec:pro-completions}. 
A method of diagram replacement, needed in the proof of cubical approximation, is developed in Appendix \S\ref{sec:diagram.replacement}.

\addtocontents{toc}{\protect\setcounter{tocdepth}{2}}

\section{Conventions}\label{sec:conventions}
This section first fixes some conventions.
Let $k,m,n,n_1,n_2,p,q$ denote natural numbers.
Let $i$ denote an integer.  
Let $\I$ denote the unit interval.
Let $\im\,f$ denote the image of a function $f$.  
Let $\ira$ denote an inclusion of some sort, such as an inclusion of a subset into a set, a subspace into a space, or a subcategory into a category.
Write $\graph{(\leqslant_X)}$ for the \textit{graph} of a preorder $\leqslant_X$ on a set $X$, the subset of $X^2$ consisting of all pairs $(x,y)$ such that $x\leqslant_Xy$.  

\subsubsection{Categories}
Let $\cat{1},\cat{2}$ denote arbitrary categories.
Let $\shape{1},\smallcat{1},\smallcat{2},\indexcat{1}$ denote small categories.
Let $\star$ denote a terminal object in a given category.
For a given monoidal category, let $\otimes$ denote its tensor product.
For each object $o$ in a given closed monoidal category, let $o^{(-)}$ denote the right adjoint to the endofunctor $o\otimes-$.
More generally for each object $o$ in a given closed monoidal category $\modelcat{1}$ and a given category $\cat{1}$ enriched, tensored, and cotensored over $\modelcat{1}$, we more generally write $-\otimes o$ and $o^{(-)}$ for the left and right endofunctors on $\cat{1}$ in an adjunction $-\otimes o\dashv o^{(-)}$, natural in $\modelcat{1}$-objects $o$, defined by tensoring and cotensoring with $o$.  
Notate special categories as follows.
\vspace{.1in}\\
\begin{tabular}{rll}
  $\SETS$ & sets (and functions)\\
  $\TOP$ & (weak Hausdorff k-)spaces (and continuous functions) \\
  $\CATS$ & small categories (and functors) \\
  $\GROUPOIDS$ & small groupoids (and functors) \\
  $\POTOP$ & locally order-convex pospaces with connected intervals & \S\ref{sec:pospaces}\\
  $\DITOP$ & (weak Hausdorff k-)streams & \S\ref{sec:streams}\\
  $\DISLATS$ & finite distributive lattices & \S\ref{sec:cube.configurations}\\
  $\INTERVALCAT$ & domain of abstract interval objects & \S\ref{sec:cubes} \\
  $\BOX$ & cube category  & \S\ref{sec:cubes}\\
  $\REGULARCUBICALSETS$ & cubical analogues of regular CW complexes & \S\ref{sec:cubical.sets}\\
  $\STARS_{k+1}$ & category of closed stars & \S\ref{sec:cubical.sets}
\end{tabular}

\vspace{.1in}

Let $[k]$ denote the set $\{0,1,\ldots,k\}$ equipped with the standard total order, regarded as a small category with objects $0,1,\ldots,k$ and exactly one arrow $m\ra n$ when $m\leqslant n$ and no arrows $m\ra n$ when $n<m$.  
The notation $[k]^n$ will refer to the $n$-fold $\CATS$-product of $[k]$.  
The notation $(-)^{[k]}$ will always refer to the right adjoint in the adjunction
$$-\times_{\CATS}[k]:\CATS\lras\CATS:(-)^{[k]}.$$

\begin{eg}
  \label{eg:arrow.category}
  The arrow category of a small category $\smallcat{1}$ is $\smallcat{1}^{[1]}$.
\end{eg}

Write $\hat{\shape{1}}$ for the category of $\SETS$-valued presheaves on $\shape{1}$, the functor category
$$\hat{\shape{1}}=\SETS^{\OP{\shape{1}}}.$$

Write $\shape{1}[-]$ for the Yoneda embedding $\shape{1}\ira\hat{\shape{1}}$. 
Let $F/G$ denote the comma category for diagrams $F,G$ in the same category.
For a diagram $F$ in $\hat{\shape{1}}$, let $\shape{1}/F=\shape{1}[-]/F$.
Let $\id_o$ denote the identity morphism for an object $o$ in a given category.
A functor $F:\cat{1}\ra\cat{2}$ is \textit{topological} if, for each diagram $D:\smallcat{1}\ra\cat{1}$, every cone $x\ra FD$ in $\cat{2}$ admits an initial lift to a cone in $\cat{1}$ along $F$; topological functors create limits and colimits \cite{borceux1994handbook}.

\subsubsection{Diagrams}
We will sometimes regard diagrams in a category $\cat{1}$ as equivariant versions of $\cat{1}$-objects.  
When we do, we adopt the following terminology.  
We take \textit{$\indexcat{1}$-streams}, \textit{$\indexcat{1}$-cubical sets}, and \textit{$\indexcat{1}$-categories} to mean $\indexcat{1}$-shaped diagrams in the respective categories $\DITOP$, $\CUBICALSETS$, and $\CATS$.  
We take \textit{$\indexcat{1}$-stream maps}, \textit{$\indexcat{1}$-cubical functions}, and \textit{$\indexcat{1}$-functors} to mean natural transformations between $\indexcat{1}$-streams, $\indexcat{1}$-cubical sets, and $\indexcat{1}$-categories.
We define enrichments on the categories of $\indexcat{1}$-streams and $\indexcat{1}$-cubical functions at the ends of sections \S\ref{sec:streams} and \S\ref{sec:cubical.sets}.  

\subsubsection{Preordered Sets}
For each preorder $\leqslant_{X}$ on a set or even more general class $X$, we write $\graph{(\leqslant_X)}$ for the set or more general class of all pairs $(x,y)$ for which $x\leqslant_Xy$.    
Preordered sets $P$, sets $P$ equipped with preorders which we denote by $\leqslant_P$, will be regarded as small categories with object set given by the underlying set of $P$ and with one morphism $x\ra y$ precisely when $x\leqslant_Py$.
In that sense a \textit{poset} is a skeletal preordered set.
A subposet $P$ of a poset $Q$ is \ldots
\begin{enumerate}
  \item \ldots \textit{order-convex in $Q$} if $y\in P$ whenever $x\leqslant_Qy\leqslant_Qz$ and $x,z\in P$
  \item \ldots an \textit{interval in $Q$} if it is order-convex and has both a minimum and maximum
  \item \ldots a \textit{chain in $Q$} if $\leqslant_P$ is total.
\end{enumerate}

In a poset $P$, write $x\vee_Py$ for the unique supremum of $x,y$ if it exists and $x\wedge_Py$ for the unique infimum of $x,y$ if it exists.  
A \textit{lattice} is always taken in the order-theoretic sense to mean a poset having all binary infima and binary suprema.
A lattice is \textit{complete} if it has all arbitrary infima and suprema.  
A lattice is \textit{distributive} if $x\wedge_L(y\vee_Lz)=(x\wedge_Ly)\vee_L(x\wedge_Lz)$ for all $x,y,z\in L$ or equivalently if $x\vee_L(y\wedge_Lz)=(x\vee_Ly)\wedge_L(x\vee_Lz)$ for all $x,y,z\in L$.  

\begin{eg}
  \label{eg:arrow.lattice}
  For a finite distributive lattice $L$, $L^{[k]}$ is a finite distributive lattice with
  $$(\alpha\vee_{L^{[k]}}\beta)(x)=\alpha(x)\vee_L\beta(x)\quad(\alpha\wedge_{L^{[k]}}\beta)(x)=\alpha(x)\wedge_L\beta(x)$$
  for all $x\in[k]$.  
\end{eg}

\begin{eg}
  For all $k$ and $n$, $\multiboxobj{k}{n}$ is a distributive lattice.
\end{eg}

A poset is \textit{Boolean} if it is $\CATS$-isomorphic to a power set, regarded as a poset under inclusion.
Every interval in a Boolean lattice is Boolean. 

\begin{eg}
  The finite Boolean lattices are, up to $\CATS$-isomorphism,
  $$[0],[1],\boxobj{2},\boxobj{3},\ldots$$
\end{eg}

A \textit{monotone function} will be taken to mean a functor $\phi:P\ra Q$ between preordered sets, a function $\phi:P\ra Q$ with $\phi(x)\leqslant_Q\phi(y)$ whenever $x\leqslant_Py$.
Define monotone functions
	\begin{align*}
	  \delta_\pm &:[0]\ra[1] &&  \delta_{\pm}(0)=\half\pm\half\\
	  \sigma&:[1]\ra[0]
	\end{align*}
A monotone function $\phi:L\ra M$ of finite lattices \textit{preserves (Boolean) intervals} if images of (Boolean) intervals in $L$ under $\phi$ are (Boolean) intervals in $M$.
A \textit{lattice homomorphism} is a function $\phi:L\ra M$ between lattices preserving binary suprema and binary infima.  

\subsubsection{Supports}
We employ some common notation for \textit{supports} and \textit{carriers}, like the support of a point in a topological realization or the carrier of a cube in a cubical subdivision. 
Consider a functor $F:\cat{1}\ra\cat{2}$ and $\cat{1}$-object $o$ admitting a complete lattice of subobjects.
Let $\support_F(\zeta,o)$ denote the minimal subobject of $o$ to which $\zeta$ corestricts, for each $\cat{2}$-morphism $\zeta$ to $Fo$. 
After identifying a point $x$ in a topological space $X$ with the map $\star\ra X$ from the singleton space whose image contains $x$, $\support_{|-|}(x,B)$ is the usual support of a point $x$ in the topological realization $|B|$ of a simplicial set $B$, the minimal subpresheaf $A\subset B$ with $x\in|A|$. 

\subsubsection{Constructions}
For reference, we list certain constructions defined throughout.
\vspace{.1in}\\
\begin{tabular}{rll}
  $\sd_{k+1}$ & subdivisions & \S\ref{sec:cube.configurations}, \S\ref{sec:cubical.sets}\\
  $\epsilon$ & natural transformation $\sd_3\ra\id_{\CUBICALSETS}$ & \S\ref{sec:cubical.sets}\\
  $\ex_{k+1}$ & right adjoint to $\sd_{k+1}$ & \S\ref{sec:cubical.sets}\\
  $\nerve$ & cubical nerves & \S\ref{sec:cubical.sets}\\
  $\fundamentalcat$ & fundamental category & \S\ref{sec:cubical.sets}\\
  $\fundamentalgroupoid$ & fundamental groupoid & \S\ref{sec:cubical.sets}\\
  $\Omega^n$ & $n$-fold directed loop space & \S\ref{sec:cubical.sets}\\
  $|\!-\!|$ & topological realizations & \S\ref{sec:space.comparisons}\\
  $\direalize{-}$ & directed realizations & \S\ref{sec:space.comparisons}\\
  $\csing$ & directed cubical singular functor & \S\ref{sec:space.comparisons}\\
	$\catfont{S}$ & monad of the adjunction $\direalize{-}\dashv\sing$ & \S\ref{sec:space.comparisons}\\
  $[-,-]_{\mathfrak{i}}$ & homotopy classes with respect to interval object $\mathfrak{i}$ & \S\ref{sec:abstract.homotopy}\\
  $\mathfrak{d}$ & cannonical interval object in $\CUBICALSETS$ & \S\ref{sec:abstract.homotopy}\\
  $\pi_0$ & path-components & \S\ref{sec:cubical.homotopy}, \S\ref{sec:continuous.homotopy}\\
		$\mathfrak{h}$ & interval object in $\DITOP$ that defines h-homotopy & \S\ref{sec:continuous.directed.homotopy}\\
		$\tau_n$ & $n$th directed homotopy monoid & \S\ref{sec:cubical.homotopy},\S\ref{sec:continuous.directed.homotopy}\\
  $\cohomology{1}$ & cubical $1$-cohomology & \S\ref{sec:cubical.homotopy}\\
  $\cohomology{1}_{\sing}$ & directed singular $1$-cohomology & \S\ref{sec:directed.homotopical.comparisons}\\
  $\cubcatreplacement$ & pointed endofunctor defining cubcats & \S\ref{sec:cubcats}\\
	\end{tabular}

\section{Directed Spaces}\label{sec:spaces}
Directed spaces can be modelled topologically and combinatorially.
This section recalls topological models, presheaf models, and comparison functors between them.  
\textit{Streams} provide topological models of directed spaces.  
\textit{Cubical sets}, presheaves over a cube category $\BOX$, provide combinatorial models of directed spaces. 
Streams can be constructed from cubical sets as \textit{directed realizations}.
Novel material includes a characterization of morphisms in the cube category [Theorem \ref{thm:box.characterization}] and a subsequent order-theoretic construction of cubical subdivision [\S\ref{sec:cube.configurations} and Proposition \ref{prop:cubical.subdivision}].

\subsection{Continuous}\label{subsec:directed.topological.spaces}
Directed spaces are modelled topologically in this paper as \textit{streams} \cite{krishnan2009convenient}.
An alternative topological model for directed spaces, common in the literature and essentially interchangeable with streams as foundations for directed homotopy, are \textit{d-spaces} \cite{grandis2003directed}.
An advantage of a stream-theoretic foundation for directed topology is that it naturally subsumes some of the theory of pospaces, whose point-set theory is well-developed in the literature \cite{lawson1973intrinsic}.   

\subsubsection{Pospaces}\label{sec:pospaces}
A \textit{pospace} is a poset $P$ topologized so that $\graph{(\leqslant_P)}$ is closed in the standard product topology on $P^2$.
A pospace $P$ is \textit{locally order-convex} if the underlying topology of $P$ has an open basis consisting of subsets which are order-convex in the underlying poset of $P$.  

\begin{thm:nachbin}
  Every compact pospace is locally order-convex.  
\end{thm:nachbin}

A \textit{subpospace} of a pospace $Q$ is a pospace $P$ that is at once a subposet and subspace of $Q$.  
A \textit{topological lattice} is a lattice $L$ topologized so that $\vee_L,\wedge_L$ are jointly continuous functions $L^2\ra L$. 
The underlying topological spaces of pospaces are necessarily Hausdorff.  
Conversely, topological lattices with Hausdorff underlying topological spaces are pospaces. 
Write $\vec{\I}$ and $\vec{\R}$ for the respective unit interval $\I$ and real number line $\R$ each equipped with the standard total order.
Write $\vec{\I}^n$ and $\vec{\R}^n$ for the pospaces whose underlying posets are $n$-fold products of underlying posets in the category of posets and monotone functions and whose underlying topological spaces are $n$-fold products in $\TOP$.  

\begin{eg}
  Fix $n$.  
  The pospace $\vec{\R}^n$, $\R^n$ equipped with the partial order defined by
  $$(x_1,x_2,\ldots,x_n)\leqslant_{\vec{\I}^n}(y_1,y_2,\ldots,y_n)\iff y_1-x_1,y_2-x_2,\ldots,y_n-x_n\geqslant 0,$$
  is a locally order-convex Hausdorff topological lattice.
\end{eg}

A \textit{monotone map} of pospaces is a function between pospaces that is at once monotone as a function between underlying posets and continuous as a function between underlying topological spaces.  
Let $\POTOP$ be the concrete category whose objects are the locally order-convex pospaces $P$ such that the intervals in the underlying poset of $P$ are connected in the underlying space of $P$ and whose morphisms are all monotone maps between such pospaces.
Every topological lattice whose underlying topological space is compact Hausdorff and connected is a $\POTOP$-object \cite[Proposition VI-5.15]{gierz2003continuous}.  

\begin{eg}
  The pospaces $\vec{\I}^n$ and $\vec{\R}^n$ are $\POTOP$-objects.  
\end{eg}

  \subsubsection{Streams}\label{sec:streams}
	A \textit{circulation} on a topological space $X$ is a function
	$$\leqslant:U\mapsto\;\leqslant_U$$
	assigning to each open subset $U\subset X$ a preorder $\leqslant_U$ on $U$ such that $\leqslant$ sends the union of a collection $\mathcal{O}$ of open subsets of $X$ to the preorder with smallest graph containing $\graph{(\leqslant_U)}$ for each $U\in\mathcal{O}$ \cite{krishnan2009convenient}.
	A \textit{stream} is a space equipped with a circulation on it \cite{krishnan2009convenient}.
	Intuitively, $x\leqslant_Uy$ in a state stream whenever a system restricted to the subset $U$ of states can evolve from $x$ to $y$.

	\begin{eg}
	  \label{eg:initial.circulations}
	  Every topological space admits an \textit{initial circulation} $\leqslant$ defined by
			\begin{equation*}
				x\leqslant_Uy\iff x=y\in U
			\end{equation*}
	\end{eg}

	A continuous function $f:X\ra Y$ of streams is a \textit{stream map} if $f(x)\leqslant_Uf(y)$ whenever $x\leqslant_{f^{-1}U}y$ for each open subset $U$ of $Y$ \cite{krishnan2009convenient}.
	A \textit{k-space} $X$ is a colimit of compact Hausdorff spaces in the category of topological spaces and continuous functions.
	Similarly, a \textit{k-stream} is a colimit of compact Hausdorff streams in the category of streams and stream maps \cite{krishnan2009convenient}.
	The underlying space of a k-stream is a k-space \cite[Proposition 5.8]{krishnan2009convenient}.
	A topological space $X$ is \textit{weak Hausdorff} if images of compact Hausdorff spaces in $X$ are Hausdorff.
	Call a stream \textit{weak Hausdorff} if its underlying topological space is weak Hausdorff.

	\begin{prop:locally.compact.streams}
	  Locally compact Hausdorff streams are weak Hausdorff k-streams.
	\end{prop:locally.compact.streams}

	Let $\TOP$ denote the complete, cocomplete, and Cartesian closed \cite{mccord1969classifying} category of weak Hausdorff k-spaces and continuous functions between them.
	Let $\DITOP$ denote the category of weak Hausdorff k-streams and stream maps.
	Redefine \textit{topological space} and \textit{stream}, like elsewhere (e.g. \cite{krishnan2009convenient, may1999concise}), to means objects in the respective categories $\TOP$ and $\DITOP$.
	The concrete \textit{forgetful functor} $\DITOP\ra\TOP$ naturally sending a stream to its underlying topological space lifts topological constructions in the following sense.

	\begin{prop:topological}
	  The forgetful functor $\DITOP\ra\TOP$ is topological.
	\end{prop:topological}

	In other words, each class of continuous functions $f_i:X\ra Y_i$ from a topological space $X$ to streams $Y_i$ induces a terminal circulation on $X$ making the $f_i$'s stream maps $X\ra Y_i$.
	Equivalently and dually, each class of continuous functions from streams to a fixed topological space induces a suitably initial circulation on that topological space.
	In this manner, the \textit{quotient} of a stream $Y$ by a subset $X$ can be defined as the quotient space $Y/X$ equipped with the initial circulation making the quotient map $Y\ra Y/X$ a stream map. 
	In particular, the forgetful functor $\DITOP\ra\TOP$ creates limits and colimits.
	A \textit{stream embedding} is a stream map $e:Y\ra Z$ such that a stream map $f:X\ra Z$ corestricts to a stream map $X\ra Y$ whenever $\im\,f\subset\,\im\,e$.
	A \textit{substream} of a stream $Y$ is a stream $X$ such that inclusion defines a stream embedding $X\ra Y$.  

	\begin{eg}
	  An open substream is an open subspace with a restricted circulation.
	\end{eg}

	\begin{thm:x-closed}
	  The category $\DITOP$ is Cartesian closed.
	\end{thm:x-closed}

	The categories $\DITOP,\TOP$ will sometimes be regarded as Cartesian monoidal.
	Explicit constructions of circulations are often cumbersome.
	Instead, circulations can be implicitly constructed from certain global partial orders in the sense of the following result, a special case of a more general observation \cite[Lemmas 4.2, 4.4 and Example 4.5]{krishnan2009convenient}.
	The following theorem allows us to henceforth regard $\POTOP$-objects as streams and monotone maps between them as stream maps.

	\begin{thm:pospaces}
	  There exists a fully faithful and concrete embedding
	  $$\POTOP\ira\DITOP,$$
	  sending each $\POTOP$-object $P$ to a unique stream having the same underlying topological space as $P$ and whose circulation sends the entire space to the given partial order on $P$.
	\end{thm:pospaces}

	The point in defining $\DITOP$ is to be able to quotient $\POTOP$-objects in a way that still corresponds to quotienting the underlying topological spaces.  
	For example, we can define a directed analogue of a sphere as the quotient
	$$\vec{\mathbb{S}}^n=\quotient{\vec{\I}^n}{\partial\I^n}$$
	in $\DITOP$ of the $\POTOP$-object $\vec{\I}^n$ by the topological boundary $\partial\I^n$ of $\I^n$ in $\R^n$.  
        Denote the quotiented point in the stream $\vec{\mathbb{S}}^n$ by $\infty$.  
	In fact, it is possible to characterize $\vec{\mathbb{S}}^n$ up to isomorphism as the terminal compactification of $\vec{\mathbb{R}}^n$ in $\DITOP$ in the following sense.  
	There exists a horizontal stream embedding in the diagram below whose image is $\vec{\mathbb{S}}^n-\{\infty\}$.  
        For each vertical stream embedding of $\vec{\mathbb{R}}^n$ into a compact Hausdorff stream $K$ with dense image, there exists a unique dotted stream map making the entire diagram commute.
        \begin{equation*}
          \begin{tikzcd}
		  \vec{\mathbb{R}}^n\ar[r,hookrightarrow]\ar[d,hookrightarrow] & \vec{\mathbb{S}}^n\\
		  K\ar[ur,dotted]
	  \end{tikzcd}
	\end{equation*}

Call the $\DITOP^{\indexcat{1}}$-objects and $\DITOP^{\indexcat{1}}$-morphisms \textit{$\indexcat{1}$-streams} and \textit{$\indexcat{1}$-stream maps}.  
Regard each hom-set $\DITOP^{\indexcat{1}}(X,Y)$ as a stream such that the function
$$f\mapsto(f_g)_g:\DITOP^{\indexcat{1}}(X,Y)\ra\prod_g Y(g)^{X(g)},$$
where $g$ denotes a $\indexcat{1}$-object, is a stream embedding.  
In this manner we regard $\DITOP^{\indexcat{1}}$ as $\DITOP$-enriched.  
A \textit{based stream} is a pair $(X,x)$ of stream $X$ and $x\in X$.  
A \textit{based stream map} $f:(X,x)\ra(Y,y)$ from a based stream $(X,x)$ to a based stream $(Y,y)$ is a stream map $f:X\ra Y$ sending $x$ to $y$.  
We sometimes regard based streams $(X,x)$ as $[1]$-streams sending $0\ra 1$ to the inclusion $\{x\}\ra X$ and based stream maps as $[1]$-stream maps between such $[1]$-streams.  

\subsection{Cubical}\label{sec:cubical}
Directed cubes can be modelled as finite Boolean lattices, more general complexes of such cubes can be modelled as posets, and even more general formal colimits of such cubes can be modelled as cubical sets.

	\subsubsection{Cubes}\label{sec:cubes}
	There are several variants of the cube category (e.g. \cite{buchholtz2017varieties,grandis2003cubical}).
	In order to define a variant for use in this paper, we adopt the following notation.  
	For a monotone function $\phi:\boxobj{n_1}\ra\boxobj{n_2}$ and $1\leqslant i\leqslant n$, let $\phi_{i;n}$ denote the Cartesian monoidal product
	$$\phi_{i;n}=\boxobj{i-1}\otimes\phi\otimes\boxobj{n-i}:\boxobj{n+n_1-1}\ra\boxobj{n+n_2-1}.$$

	\textit{Codegeneracies} are monotone functions of the form $\sigma_{i;n}:\boxobj{n+1}\ra\boxobj{n}$.  
	\textit{Cofaces} are monotone functions of the form $\delta_{\pm i;n}=(\delta_{\pm})_{i;n}:\boxobj{n-1}\ra\boxobj{n}$.

	\begin{eg}
	  The codegeneracy $\sigma_{i;n}$ is exactly the projection
	  $$\sigma_{i;n}:\boxobj{n}\ra\boxobj{n-1}$$
	  onto all but the $i$th factor.  
	\end{eg}

  Let $\BOX_1$ denote the subcategory of $\CATS$ generated by $\delta_{\pm},\sigma$.  
  The submonoidal category of $\CATS$ generated by $\BOX_1$ is the usual, minimal variant of the cube category in the literature, the subcategory of $\CATS$ generated by all cofaces and codegeneracies.
  Instead let $\BOX$ denote the \textit{symmetric} monoidal subcategory of the Cartesian monoidal category $\CATS$ generated by $\BOX_1$, the subcategory of $\CATS$ whose objects are still the lattices $[0],[1],\boxobj{2},\boxobj{3},\ldots$ but whose morphisms are generated by the cofaces, codegeneracies, and coordinate permutations.
  The following observation allows us to extend certain results on the minimal variant of the cube category to the new variant $\BOX$.  

\begin{lem:box-inclusions}
  For each $n$ and interval $I$ in $\boxobj{n}$, there exist unique $m_I$ and composite 
  $$\boxobj{m_I}\ra\boxobj{n}$$
  of cofaces that has image $I$.
\end{lem:box-inclusions}

We will repeatedly use the convenient fact that $\BOX$ is the free strict symmetric monoidal category generated by the category $\BOX_1$ pointed at $[0]$: every solid horizontal functor to a symmetric monoidal category $\modelcat{1}$ sending $[0]$ to the unit uniquely extends to a strict monoidal functor making the following commute by observations made elsewhere \cite{grandis2003cubical}.
	\begin{equation*}
		\begin{tikzcd}
			\BOX_1\ar[d,hookrightarrow]\ar[r] & \modelcat{1}\\
			\BOX\ar[ur,dotted]
		\end{tikzcd}
	\end{equation*}

	There are some advantages to adding coordinate permutations to $\BOX$.
	One is that the class of all directed realizations of cubical sets (see \S\ref{sec:space.comparisons}) includes, for example, all closed conal manifolds whose cone bundles are fibrewise generating and free \cite[Theorem 1.1]{krishnan2019triangulations}.
	A bigger one is an explicit characterization of $\BOX$-morphisms [Theorem \ref{thm:box.characterization}] to which the rest of this section is devoted.

	\begin{eg}
	  In $\BOX$ the \ldots
	  \begin{enumerate}
	    \item \ldots isomorphisms are the coordinate permutations
	    \item \ldots monos are the cofaces up to coordinate permutation
		  \item \ldots epis are the codegeneracies up to coordinate permutation
	  \end{enumerate}
	\end{eg}

	Let $\tau$ denote the coordinate transposition $\boxobj{2}\ra\boxobj{2}$.
	\textit{Principal coordinate transpositions} are $\BOX$-morphisms of the form $\tau_{i;n}:\boxobj{n+2}\ra\boxobj{n+2}$.

	\begin{lem}
	  \label{lem:box.automorphisms}
	  The following are equivalent for a monotone function of the form
	  $$\phi:\boxobj{m}\ra\boxobj{n}.$$
	  \begin{enumerate}
	    \item\label{item:box.bijection} $\phi$ is bijective
	    \item\label{item:interval.preserving.lattice.isomorphism} $\phi$ is an interval-preserving bijection
	    \item\label{item:automorphism} $\phi$ is a lattice isomorphism
	    \item\label{item:permutation} $\phi$ is a coordinate permutation
	    \item\label{item:cosymmetry} $\phi$ is composite of principal coordinate transpositions
	    \item\label{item:box.automorphism} $\phi$ is a $\BOX$-isomorphism
	  \end{enumerate}
	\end{lem}

	The proof uses the fact that the symmetric group on $\{1,2,\ldots,n\}$ is generated by all principal transpositions, transpositions of the form $(i\,i+1)$ for $1\leqslant i<n$ \cite[\S 6.2]{coxeter1980generators}.

	\begin{proof}
	  Let ${\bf 0}$ denote the minimum $(0,\ldots,0)$ of an element in $\BOX$.  
	  Let ${\bf e}_i$ denote the element in $\boxobj{n}$ whose coordinates are all $0$ except for the ith coordinate.
	  
	  It suffices to take $m=n$ because all of the statements imply that $\phi$ is a bijection between finite sets and hence $\phi$ has domain and codomain both with the same cardinality.

  Suppose (\ref{item:box.bijection}).
  Then $\phi$ preserves extrema because it is a monotone surjection.  

  Consider $x\in\boxobj{n}$.
  Every immediate successor to $\phi(x)$ in $\boxobj{n}$ is the image under $\phi$ of an immediate successor to $x$ in $\boxobj{n}$; for otherwise there exists some immediate successor to $\phi(x)$ in $\boxobj{n}$ not in the image of $\phi$ by $\phi$ monotone, contradicting $\phi$ surjective.  
  Therefore $\phi$ restricts and corestricts to a bijection from the finite set of all $n$ immediate successors of $x$ in $\boxobj{n}$ to the finite set of all $n$ immediate successors of $\phi(x)$ in $\boxobj{n}$.     
  In particular, $\phi(y)$ is an immediate successor to $\phi(x)$ in $\boxobj{n}$ if and only if $y$ is an immediate successor to $x$ in $\boxobj{n}$.  
  Thus for all $x\leqslant_{\boxobj{n}}y$, $\phi$ sends maximal chains in $\boxobj{n}$ with extrema $x,y$ to maximal chains in $\boxobj{n}$ with extrema $\phi(x),\phi(y)$.  

  Let $I$ be an interval in $\boxobj{n}$, necessarily isomorphic to a lattice of the form $\boxobj{k}$.
  Then $I$ contains exactly $k!$ distinct maximal chains of length $k$.  
  The function $\phi$ preserves chains of length $k$ because it is a monotone injection between posets.
  Hence there exist $k!$ distinct chains of length $k$ in $\phi(I)$ that are maximal as chains in $\boxobj{n}$ having extrema $\phi(\min\,I)=\min\,\phi(I)$ and $\phi(\max\,I)=\max\,\phi(I)$.
  The only kinds of subposets of $\boxobj{n}$ having a minimum, a maximum, and $k!$ distinct chains which are maximal in $\boxobj{n}$ as chains with a given minimum and given maximum are intervals in $\boxobj{n}$ isomorphic to $\boxobj{k}$ as lattices.  
  Therefore $\phi(I)$ is an interval in $\boxobj{n}$.    

	  Thus $\phi$ maps intervals onto intervals.  
	  Hence (\ref{item:interval.preserving.lattice.isomorphism}).
	 
	  Suppose (\ref{item:interval.preserving.lattice.isomorphism}).
	  Finite non-empty suprema of the ${\bf e}_i$s are the maxima of intervals in $\boxobj{n}$ containing ${\bf 0}$.  
	  And $\phi$ maps intervals in $\boxobj{n}$ containing ${\bf 0}$ onto intervals in $\boxobj{n}$ containing ${\bf 0}$.  
	  It therefore follows that $\phi$ preserves finite non-empty suprema of the ${\bf e}_i$s because monotone surjections preserve maxima.
	  Hence $\phi$ preserves all finite non-empty suprema.  
	  Similarly $\phi$ preserves all finite non-empty infima by duality.    
	  It therefore follows that $\phi$ is a bijective lattice homomorphism and hence a lattice isomorphism.  
	  Hence (\ref{item:automorphism}).

	  Suppose (\ref{item:automorphism}).
	  The function $\phi$, a monoid automorphism with respect to $\vee_{\boxobj{m}}$, permutes the unique minimal set of monoid generators ${\bf e}_1,{\bf e}_2,\ldots,{\bf e}_n$.
	  Thus there exists a permutation $\sigma$ of $\{1,2,\ldots,n\}$ such that $\phi({\bf e}_i)={\bf e}_{\sigma(i)}$ for each $i$.
	  Hence $\phi(x_1,\ldots,x_n)=\phi(\vee_{x_i=1}{\bf e}_i)=\vee_{x_i=1}\phi({\bf e}_{i})=\vee_{x_i=1}{\bf e}_{\sigma(i)}=(x_{\sigma(1)},\ldots,x_{\sigma(n)})$.
	  Hence (\ref{item:permutation}).

	  If (\ref{item:permutation}), then $\phi$ is a composite of transpositions of successive coordinates, principal coordinate transpositions \cite[\S 6.2]{coxeter1980generators}.
	  Then  (\ref{item:cosymmetry}).

	  If (\ref{item:cosymmetry}), then $\phi$ is a composite of $\BOX$-isomorphisms and hence a $\BOX$-isomorphism.
	  Hence (\ref{item:box.automorphism}).

	  If (\ref{item:box.automorphism}), then (\ref{item:box.bijection}) because the forgetful functor $\BOX\ra\SETS$, like all functors, preserves isomorphisms.
	\end{proof}

	\begin{lem}
	  \label{lem:surjective.box.morphisms}
	  The following are equivalent for a function of the form
	  $$\phi:\boxobj{m}\ra\boxobj{n}.$$
	  \begin{enumerate}
	    \item\label{item:box.surjection} $\phi$ is a surjective interval-preserving lattice homomorphism
	    \item\label{item:surjection} $\phi$ is a surjective lattice homomorphism
	    \item\label{item:codegeneracies} $\phi$ is a composite of codegeneracies and principal coordinate transpositions
	  \end{enumerate}
	\end{lem}
	\begin{proof}
	  For clarity, let $\wedge=\wedge_L$ and $\vee=\vee_L$ when the lattice $L$ is clear from context.
          Let ${\bf e}^\perp_i$ denotes the element in $\boxobj{m}$ whose only coordinate having value $0$ is its $i$th coordinate.

	  (\ref{item:box.surjection}) implies (\ref{item:surjection}).

	  Suppose (\ref{item:surjection}).
	  Then $m\geqslant n$ by surjectivity.
	  We show (\ref{item:codegeneracies}) by induction on $m-n$. 

	  In the base case $m=n$, $\phi$ is a bijection because it is a surjection between sets of the same cardinality and hence is a composite of principal coordinate transpositions [Lemma \ref{lem:box.automorphisms}].

	  Consider $m-n>0$.
	  Inductively suppose (\ref{item:codegeneracies}) for the case $m-n<d$ and now consider the case $m-n=d>0$.
	  Then $\phi$ is not injective by $m>n$.
	  Thus there exist distinct $x,y\in\boxobj{m}$ such that $\phi(x)=\phi(y)$.
	  There exists $j$ such that $x_j\neq y_j$ by $x\neq y$.
	  Take $0=x_j<y_j=1$ and $x_i=y_i=1$ for $i\neq j$ without loss of generality by reordering $x$ and $y$ if necessary, replacing $x$ with $x\vee {\bf e}^\perp_j$ and $y$ with $y\vee{\bf e}^\perp_j$, and noting that $\phi(x\vee{\bf e}^\perp_j)=\phi(x)\vee\phi({\bf e}^\perp_j)=\phi(y)\vee\phi({\bf e}^\perp_j)=\phi(y\vee {\bf e}^\perp_j)$, 
	  It suffices to show the existence of a dotted function making 
	  \begin{equation*}
	    \begin{tikzcd}
		    \boxobj{m}\ar{rr}[above]{\phi}\ar{dr}[description]{\sigma_j} 
		    & & \boxobj{n}\\
		    & \boxobj{m-1}\ar[ur,dotted]
	    \end{tikzcd}
	  \end{equation*}
	  commute.  
	  For then the dotted function is a surjective lattice homomorphism by $\phi$ a surjective lattice homomorphism and $\sigma_j$ a projection.  
	  To that end, suppose distinct $x',y'\in\boxobj{m}$ satisfy $\sigma_j(x')=\sigma_j(y')$.
	  It suffices to show $\phi(x')=\phi(y')$.
	  We can take $0=x'_j<y'_j=1$ without loss of generality because $x'$ and $y'$ differ in exactly the $j$th coordinate.
	  Then $\phi(x')=\phi(y'\wedge x)=\phi(y')\wedge\phi(x)=\phi(y')\wedge\phi(y)=\phi(y'\wedge y)=\phi(y')$.
	  Hence (\ref{item:codegeneracies}).
	  
	  (\ref{item:codegeneracies}) implies (\ref{item:box.surjection}) because identities, $\sigma$, and $\tau$ are all surjective interval-preserving lattice homomorphisms and the tensor on $\BOX$ is closed under surjective interval-preserving lattice homomorphisms.
	\end{proof}

	\begin{thm}
	  \label{thm:box.characterization}
	  The following are equivalent for a function $\phi$ of the form
	  $$\phi:\boxobj{m}\ra\boxobj{n}.$$
	  \begin{enumerate}
		\item\label{item:box.characterization.bijection} $\phi$ is an interval-preserving lattice homomorphism
	    \item\label{item:box.characterization.homomorphism} $\phi$ is a $\BOX$-morphism
	\end{enumerate}
	\end{thm}
	\begin{proof}
	  Suppose (\ref{item:box.characterization.bijection}).
	  The function $\phi$ factors into a composite of its corestriction onto its image $I$, regarded as a subposet of $\boxobj{n}$, followed by an inclusion $I\ira\boxobj{n}$.
	  Both functions $\boxobj{m}\ra I$ and $I\ira\boxobj{n}$ are interval-preserving lattice homomorphisms because $\phi$ is an interval-preserving lattice homomorphism.
	  Moreover $I\ira\boxobj{n}$ is isomorphic to a $\BOX$-morphism [\LEMBoxInclusions].
	  Hence to show (\ref{item:box.characterization.homomorphism}), it suffices to take $\phi$ surjective. 
	  In that case $\phi$ factors as a composite of tensor products of identities with $\sigma,\tau$ [Lemma \ref{lem:surjective.box.morphisms}].
	  Hence (\ref{item:box.characterization.homomorphism}).

	  Suppose (\ref{item:box.characterization.homomorphism}).
	  Then $\phi$ is an interval-preserving lattice homomorphism because $\sigma,\delta_{\pm},\tau$ are interval-preserving lattice homomorphisms and $\otimes$ preserves interval-preserving lattice homomorphisms.
	  Hence (\ref{item:box.characterization.bijection}).
	\end{proof}

\subsubsection{Cube configurations}\label{sec:cube.configurations}
Just as posets encode simplicial complexes whose simplices correspond to finite chains, posets can encode cubical complexes whose cubes correspond to finite Boolean intervals.  
Let $\DISLATS$ be the symmetric submonoidal subcategory of the Cartesian monoidal category $\CATS$ whose objects are the finite distributive lattices and whose morphisms are the lattice homomorphisms between such lattices preserving Boolean intervals, lattice homomorphisms $L\ra M$ between finite distributive lattices $L$ and $M$ mapping Boolean intervals in $L$ onto Boolean intervals in $M$.  

\begin{eg}
  The category $\DISLATS$ contains $\BOX$ as a full subcategory [Theorem \ref{thm:box.characterization}].  
\end{eg}

The following technical observations about $\DISLATS$ [Lemma \ref{lem:cubical.pasting.schemes} and Proposition \ref{prop:lattice.subdivision}], which require specialized observations about finite distributive lattices, are proven in \S\ref{sec:modular.lattices}.

\begin{lem}
  \label{lem:cubical.pasting.schemes}
  The following are equivalent for a function
  $$\phi:L\ra M$$
  between finite distributive lattices.
  \begin{enumerate}
    \item $\phi$ is a $\DISLATS$-morphism
    \item each restriction of $\phi$ to a Boolean interval in $L$ corestricts to a surjective lattice homomorphism onto a Boolean interval in $M$. 
  \end{enumerate}
\end{lem}

\begin{figure}
	\begin{equation*}
	     \adjustbox{scale=.5,center}{\begin{tikzcd}
			     & & \mathlarger{\mathlarger{\bullet}} & &                               & &                         & & \cdot
	     \\
	     & \; & & \; &                                       & &                         & \cdot\ar[ur] & & \cdot\ar[ul]
	     \\
	     \cdot\ar[uurr] & & \; & & \cdot\ar[uull]     & &                         \cdot\ar[ur] & &  \mathlarger{\mathlarger{\bullet}}\ar[ul]\ar[ur] & & \cdot\ar[ul]
	     \\
	     & \; & & \; &                                       & &                          & \cdot\ar[ul]\ar[ur] & & \cdot\ar[ur]\ar[ul]
	     \\
	     & &  \mathlarger{\mathlarger{\bullet}}\ar[uull]\ar[uurr] & &             & &                           & & \cdot\ar[ul]\ar[ur]
	  \end{tikzcd}}
  \end{equation*}
  \caption{
  {\bf Order-theoretic Subdivision}.
  The left square and right squares represent the Hasse diagrams for the respective posets $\boxobj{2}$ and $\multiboxobj{2}{2}$.  
  Ordered pairs of elements in the left poset, or equivalently monotone functions $[1]\ra\boxobj{2}$, are in 1-1 correspondence with elements in $\multiboxobj{2}{2}$.  
  Along this correspondence, for example, the ordered pair of larger points on the left corresponds to the large point on the right.
  }
  \label{fig:order.sd}
\end{figure}

For each $k$, we can make the natural identifications 
$$\left(\boxobj{n}\right)^{[k]}\cong\multiboxobj{k+1}{n}$$
of finite distributive lattices, where the left side represents the finite distributive lattice of all monotone functions $[k]\ra\boxobj{n}$ whose lattice operations are defined element-wise and the right side is an $n$-fold $\CATS$-product of the ordinal $[k+1]$. 
For the case $n=1$, the natural identifications are given by unique isomorphisms that send each monotone function $\phi$ in the left poset to the element $\sum_i\phi(i)$ in the right poset.  
More generally, the natural identifications are given by natural Cartesian monoidal $\CATS$-isomorphisms.
Thus the construction $(-)^{[k]}$ intuitively subdivides an $n$-cube, as encoded by the Boolean lattice $\boxobj{n}$, into $(k+1)^n$ subcubes [Figure \ref{fig:order.sd}].  
The following proposition naturally extends this subdivision construction to an endofunctor $\sd_{k+1}$ on $\DISLATS$.

\begin{prop}
  \label{prop:lattice.subdivision}
  Consider the commutative outer rectangle
  \begin{equation*}
    \begin{tikzcd}
	    \BOX\ar[d,hookrightarrow]\ar[r,hookrightarrow] & \CATS\ar{r}[above]{(-)^{[k]}} & \CATS\\
	    \DISLATS\ar[dotted]{rr}[below]{\sd_{k+1}} & & \DISLATS\ar[u,hookrightarrow]
    \end{tikzcd}
  \end{equation*}
  There exists a unique (monoidal) functor $\sd_{k+1}$ such that the entire diagram commutes and $(\DISLATS\ira\CATS)\sd_{k+1}$ is the left Kan extension of the composite of the top horizontal arrows along the inclusion $\BOX\ira\DISLATS$.  
  For each monotone injection $\phi:[n]\ra[m]$, there exists a unique monoidal natural transformation $\sd_{m+1}\ra\sd_{n+1}$ whose $I$-component is the monotone function $I^{[m]}\ra I^{[n]}$ naturally sending each monotone function $\zeta:[m]\ra I$ to the monotone function $\zeta\phi:[n]\ra I$ for each $\BOX$-object $I$.
\end{prop}

Proofs of Lemma \ref{lem:cubical.pasting.schemes} and Proposition \ref{prop:lattice.subdivision} are given at the end of \S\ref{sec:modular.lattices}.
While the restrictions of $(-)^{[k]}$ and $(\DISLATS\ira\CATS)\sd_{k+1}$ to $\BOX$ coincide, $\sd_{k+1}$ is not the restriction and corestriction of the endofunctor on $(-)^{[k]}$ on $\CATS$.  

\begin{eg}
  Note that as lattices, $\sd_2[2]=[4]\ncong[2]^{[1]}$.
\end{eg}

\subsubsection{Cubical sets}\label{sec:cubical.sets}
Take \textit{cubical sets} and \textit{cubical functions} to mean the respective objects and morphisms of $\CUBICALSETS$.
We can regard $\CUBICALSETS$ as closed symmetric monoidal with tensor $\otimes$ characterized by $\BOX[-]:\BOX\ira\CUBICALSETS$ monoidal, or equivalently $\otimes$ defined by Day convolution of the tensor on $\BOX$ \cite{im1986universal}. 
The $\BOX$-morphisms $\boxobj{m}\otimes\boxobj{n}\ra\boxobj{m}$ and $\boxobj{m}\otimes\boxobj{n}\ra\boxobj{n}$ defined by $\CATS$-projections onto first and second factors, natural in $\BOX$-objects $\boxobj{m}$ and $\boxobj{n}$, induce inclusions of the following form, natural in cubical sets $A$ and $B$:
\begin{equation}
 \label{eqn:submonoidal.product}
  A\otimes B\ira A\times B
\end{equation}

The \textit{dimension} of a cubical set $C$, denoted $\dim\,C$, is defined for $C=\varnothing$ to be $-1$ and defined for $C\neq\varnothing$ to be the infimum over all natural numbers $n$ such that $C$ is a colimit of representables of the form $\BOX\boxobj{i}$ for $0\leqslant i\leqslant n$.  
A cubical set $A$ is \textit{atomic} if it is the quotient in $\hat\BOX$ of a representable.  
We write $\REGULARCUBICALSETS$ for the full subcategory of $\CUBICALSETS$ whose objects are those cubical sets whose atomic subpresheaves are all isomorphic to representables; the $\REGULARCUBICALSETS$-objects can be regarded as cubical analogues of regular CW complexes.  
A cubical set is \textit{finite} if it contains finitely many atomic subpresheaves.
For each atomic cubical set $A$, let $\partial A$ denote the maximal subpresheaf of $A$ having dimension $\dim\,A-1$.  
The \textit{vertices} of $C$ are the elements of $C_0$.
Let $\Star_C(v)$ denote the \textit{closed star} of a vertex $v\in C_{0}$ in $C$, the subpresheaf of $C$ consisting of all images $A\subset C$ of representables in $C$ with $v\in A_{0}$.

\begin{eg}
  \label{eg:kan.condition}
  We have inclusions of finite subpresheaves of the form
  \begin{equation}
    \label{eqn:kan.condition}
    \partial\BOX\boxobj{n}\subset\BOX\boxobj{n},\quad n=0,1,\ldots 
  \end{equation}
  For each $n>0$, $\partial\BOX\boxobj{n}$ intuitively models the boundary of an $n$-cube and $\BOX\boxobj{n}$ models an $n$-cube.  
\end{eg}

The \textit{fundamental groupoid functor} $\Pi_1$ is the cocontinuous left Kan extension in
\begin{equation*}
  \begin{tikzcd}
	  \BOX\ar{d}[left]{\BOX[-]}\ar[r,hookrightarrow] & \GROUPOIDS\\
	  \hat\BOX\ar{ur}[description]{\Pi_1}
  \end{tikzcd}
\end{equation*}
of the top horizontal inclusion along the vertical Yoneda embedding $\BOX[-]$. 
The fundamental groupoid $\Pi_1C$ ignores information about edge-orientations in a cubical set $C$.  
The \textit{fundamental category functor} $\Tau_1$ is the cocontinuous left Kan extension in
\begin{equation*}
  \begin{tikzcd}
	  \BOX\ar{d}[left]{\BOX[-]}\ar[r,hookrightarrow] & \CATS\\
	  \hat\BOX\ar{ur}[description]{\Tau_1}
  \end{tikzcd}
\end{equation*}
of the top horizontal inclusion along the vertical Yoneda embedding $\BOX[-]$.  
The functor $\Tau_1$ is monoidal by inclusion $\BOX\ira\CATS$ monoidal and the fact that tensor products commute with colimits in the closed symmetric monoidal category $\hat\BOX$.  
The functor $\Tau_1$, cocontinuous between locally presentable categories, has a right adjoint 
$$\nerve:\CATS\ra\CUBICALSETS.$$

Fix a small category $\smallcat{1}$.  
Call $\cnerve\smallcat{1}$ the \textit{cubical nerve} of $\smallcat{1}$.  
For each $n$, $(\cnerve\smallcat{1})_n$ is the set of all functors $\boxobj{n}\ra\smallcat{1}$.
For each $\BOX$-morphism $\phi:\boxobj{n}\ra\boxobj{m}$,
$$(\cnerve\smallcat{1})(\phi)=\zeta\mapsto\zeta\circ\phi:(\cnerve\smallcat{1})_m\ra(\cnerve\smallcat{1})_n.$$

For each finite distributive lattice $L$, let $\BOX[L]$ denote the subpresheaf of $\cnerve L$ such that $\BOX[L]_n$ is the set of all lattice homomorphisms $\boxobj{n}\ra L$ preserving Boolean intervals.
For each $\DISLATS$-morphism $\phi:L\ra M$, there exists a unique dotted cubical function, which we denote by $\BOX[\phi]$, making the diagram 
\begin{equation*}
  \begin{tikzcd}
	  \BOX[L]\ar[d,hookrightarrow]\ar[dotted]{r}[above]{\BOX[\phi]} & \BOX[M]\ar[d,hookrightarrow]\\
	  \cnerve L\ar{r}[below]{\cnerve\phi} & \cnerve M
  \end{tikzcd}
\end{equation*}
commute. 
Thus $\BOX[-]$ will not only denote the Yoneda embedding but also its extension
$$\BOX[-]:\DISLATS\ra\CUBICALSETS.$$

The unique cocontinuous extension of the restriction of $\sd_{k+1}$ to $\BOX$, which we also denote as $\sd_{k+1}$, extends the original endofunctor $\sd_{k+1}$ on $\DISLATS$.   

\begin{prop}
  \label{prop:cubical.subdivision}
  There exists a unique dotted cocontinuous monoidal functor making
  \begin{equation}
    \label{eqn:cubical.subdivision}
  \begin{tikzcd}
	  \DISLATS\ar{d}[left]{\BOX[-]}\ar{r}[above]{\sd_{k+1}} & \DISLATS\ar{d}[right]{\BOX[-]}\\
	  \CUBICALSETS\ar[dotted]{r}[below]{\sd_{k+1}} & \CUBICALSETS
  \end{tikzcd}
  \end{equation}
  commute up to natural isomorphism.  
\end{prop}
\begin{proof}
  Take the dotted functor in the diagram
  \begin{equation*}
    \begin{tikzcd}
	    \BOX\ar{d}[left]{\BOX[-]}\ar[r,hookrightarrow] & \DISLATS\ar{r}[above]{\sd_{k+1}} & \DISLATS\ar{r}[above]{\BOX[-]} & \CUBICALSETS\\
	    \CUBICALSETS\ar[dotted]{urrr}[description]{\sd_{k+1}}
    \end{tikzcd}
  \end{equation*}
  to be the left Kan extension of the composite of the horizontal arrows along the vertical arrow.
  The dotted functor is the unique cocontinuous functor making the outer triangle commute and hence also monoidal by the horizontal arrows monoidal and $\otimes$ cocontinuous.  
  In particular, uniqueness follows.
  And there exist $\CUBICALSETS$-isomorphisms
  \begin{align*}
	  \BOX[\sd_{k+1}L] 
	  &\cong \colim_{I\subset L}\BOX[\sd_{k+1}I]\\
	  &\cong \colim_{\boxobj{n}\ra L}\BOX[\sd_{k+1}\boxobj{n}]\\
	  &\cong \colim_{\BOX\boxobj{n}\ra\BOX[L]}\BOX[\sd_{k+1}\boxobj{n}]\\
	  &\cong \sd_{k+1}\BOX[L]
  \end{align*}
  natural in $\DISLATS$-objects $L$, where the first colimit is taken over all Boolean intervals $I$ in $L$ and all inclusions between them and the second colimit is taken over all $(\BOX/L)$-objects $\boxobj{n}\ra L$.  
  The first isomorphism follows because every Boolean interval in $\sd_{k+1}L$ is a Boolean interval in $\sd_{k+1}I$ for some Boolean interval $I$ in $L$ [Proposition \ref{prop:lattice.subdivision}].  
  The second isomorphism follows because every $\DISLATS$-morphism  $\boxobj{n}\ra L$ factors as a composite of its surjective corestriction onto its image followed by the inclusion of a Boolean interval into $L$. 
  The third isomorphism follows from $\BOX[-]:\DISLATS\ra\CUBICALSETS$ fully faithful [Lemma \ref{lem:cubical.pasting.schemes}].
  The fourth isomorphism follows by construction of $\sd_{k+1}:\CUBICALSETS\ra\CUBICALSETS$.  
  Thus (\ref{eqn:cubical.subdivision}) commutes up to natural isomorphism.  
\end{proof}

Intuitively, $\sd_{k+1}C$ is the cubical set obtained by taking $(k+1)$-fold edgewise subdivisions of the cubes in $C$.  

\begin{eg}
  There exists a natural isomorphism $\sd_1\cong\id_{\CUBICALSETS}:\CUBICALSETS\cong\CUBICALSETS$.  
\end{eg}

The endofunctor $\sd_{k+1}$, cocontinuous on a locally presentable category, has a right adjoint 
$$\ex_{k+1}:\CUBICALSETS\ra\CUBICALSETS.$$

Regard $\sd_3$ as copointed by the unique monoidal natural transformation $\epsilon$ such that 
$$\epsilon_{\BOX\boxobj{n}}=\BOX\left[(-)^{0\mapsto 1:[0]\ra[2]}\right]:\BOX\multiboxobj{3}{n}\ra\BOX\boxobj{n}.$$

Define a cubical analogue of an $n$-fold loop space as follows.  
Recall that for cubical sets $C$, $C^{(-)}$ denote the right adjoint to the endofunctor $-\otimes C$ on $\CUBICALSETS$.  
Concretely, $C^B$ denotes the cubical set naturally sending each $\OP{\BOX}$-object $\boxobj{n}$ to the set of all cubical functions $B\otimes\BOX\boxobj{n}\ra C$, for all cubical sets $B$ and $C$.  
Define $\Omega^n(C,v)$ by the following Cartesian square natural in a cubical set $C$ equipped with vertex $v\in C_0$, where $\langle v\rangle$ denotes the minimal subpresheaf of $C$ containing $v$ as its unique vertex:
\begin{equation}
  \label{eqn:n-fold.loops}
	\begin{tikzcd}
		\Omega^n(C,v)\ar[rr]\ar[d]\ar[drr,phantom,very near start,"\lrcorner"] & & C^{\BOX\boxobj{n}}\ar{d}[right]{C^{\partial\BOX\boxobj{n}\ira\BOX\boxobj{n}}}\\
		\langle v\rangle^{\partial\BOX\boxobj{n}}\ar{rr}[below]{(\langle v\rangle\ira C)^{\partial\BOX\boxobj{n}}} & & C^{\partial\BOX\boxobj{n}}
	\end{tikzcd}
\end{equation}

A crucial technical tool in classical proofs of simplicial approximation is the factorizability of double barycentric subdivisions through polyhedral complexes \cite[\S 12]{curtis1971simplicial}.
There exists a directed, cubical analogue.
The following three lemmas adapt observations made in a predecessor to this paper \cite[Lemmas 6.11, 6.12, 6.13]{krishnan2015cubical} from the traditional setting of cubical sets to the cubical sets considered in this paper and from sixteen-fold subdivision $\sd_{16}=\sd^4_2$ to nine-fold subdivision $\sd_9=\sd_3^2$; justifications are given after all three lemmas are stated.
After identifying a vertex $v$ in a cubical set $C$ with the cubical function $\star\ra C$ whose image has the unique vertex $v$, $\support_{\sd_3}(v,C)$ denotes the minimal subpresheaf $B$ of $C$ for which $\sd_3B$ has vertex $v$.

 \begin{lem}
  \label{lem:collapse.star}
  For all cubical sets $C$ and $v\in\sd_3C_{0}$, 
  $$\epsilon_C(\Star_{\sd_3C})(v)\subset\support_{\sd_3}(v,C).$$
\end{lem}

\begin{lem}
  \label{lem:star.flower}
  Fix cubical set $C$ and non-empty subpresheaf
  $$\varnothing\neq A\subset\sd_3C$$
  of an atomic subpresheaf of $\sd_3C$.  
  There exist:
  \begin{enumerate}
    \item unique minimal subpresheaf $B\subset C$ with $A\cap\sd_3B\neq\varnothing$
    \item\label{item:star.flower.retraction} retraction $\pi_{(C,A)}:A\ra A\cap\sd_3B$, unique up to isomorphism
  \end{enumerate}
  Moreover, $A\cap\sd_3B$ is isomorphic to a representable if $A$ is atomic and $\epsilon_C(A\ira\sd_3C)=\epsilon_C((A\cap\sd_3B)\ira\sd_3C)\pi_{(C,A)}$.
\end{lem}

\begin{lem}
  \label{lem:natural.retractions}
  Consider the left of the solid commutative diagrams
  \begin{equation*}
    \begin{tikzcd}
	  A'\ar{r}[above]{\alpha}\ar[d,hook] & A''\ar[d,hook]\\
	  \sd_3C'\ar{r}[below]{\sd_3\gamma} & \sd_3C''
    \end{tikzcd}\quad
    \begin{tikzcd}
		A'\ar{r}[above]{\alpha}\ar{d}[left]{\pi'} & A''\ar{d}[right]{\pi''}\\
      A'\cap\sd_3B'\ar[r,dotted] & A''\cap\sd_3B''
    \end{tikzcd}
  \end{equation*}
  where $A',A''$ are non-empty subpresheaves of atomic subpresheaves of the respective cubical sets $C',C''$.
  Suppose $B',B''$ are minimal respective subpresheaves of $C',C''$ such that $A'\cap\sd_3B'\neq\varnothing$ and $A''\cap\sd_3B''\neq\varnothing$.
  Let $\pi',\pi''$ be retractions of inclusions in the right diagram.  
  There exists a unique dotted cubical function making the right square commute.
\end{lem}

The claim that $A\cap\sd_3B$ is representable in Lemma \ref{lem:star.flower} follows from the fact that $B$ and hence also $A\cap\sd_3B$ are atomic and $A\cap\sd_3\partial B=\varnothing$ by minimality.  
To show the other claims, it suffices to take the case where $C$ is representable by naturality and hence the even more special case where $C=\BOX[1]$ because all the functors and natural transformations in sight are monoidal.
In that case, these other claims follow from inspection.
These claims collectively imply that the natural cubical function $\epsilon^2_C:\sd_9C\ra C$ naturally but locally factors through representables in the following sense.
Let $\STARS_{k+1}$ denote the full subcategory of $(\CUBICALSETS/\sd_{k+1})$ whose objects, regarded as cubical functions $S\ra\sd_{k+1}C$, are inclusions of non-empty subpresheaves $S$ of closed stars of vertices in cubical sets of the form $\sd_{k+1}C$.  
Denote a $\STARS_9$-object $S\ira\sd_9C$ as a pair $(C,S)$.  

\begin{lem}
  \label{lem:local.lifts}
   There exist natural number $n_{(C,S)}$ and dotted cubical functions in
  \begin{equation*}
    \begin{tikzcd}
	        S\ar[r,dotted]\ar[d,hookrightarrow] & \BOX\boxobj{n_{(C,S)}}\ar[d,dotted]\\
		\sd_9C\ar{r}[below]{\epsilon^2_C} & C,
	\end{tikzcd}
  \end{equation*}
   natural in $\STARS_9$-objects $(C,S)$, making the diagram commute.
\end{lem}

The proof mimics a proof of an analogous result in a predecessor to this paper \cite[Lemma 8.16]{krishnan2015cubical}.  
That result is stated at the level of streams instead of cubical sets and for $\sd_4=\sd_2^2$ instead of $\sd_3$.  
We therefore include the following proof for completeness.  
\begin{proof}
  The cubical set $\epsilon_{\sd_3C}(S)$ is a subpresheaf of a minimal atomic subpresheaf $A(C,S)$ of $C$ [Lemma \ref{lem:collapse.star}].  
  Thus there exists a unique minimal atomic subpresheaf $C_{S}\subset C$ with $A_{(C,S)}\cap\sd_3C_{S}\neq\varnothing$ [Lemma \ref{lem:star.flower}].
  And $B_{(C,S)}=A_{(C,S)}\cap\sd_3C_{S}$ is independent of our choice of $A(C,S)$ by minimality.
  The inclusion $B_{(C,S)}\ira A_{(C,S)}$ of $B_{(C,S)}=A_{(C,S)}\cap\sd_3C_{S}$ admits a retraction $\pi_{(C,S)}$ making the right square in
  \begin{equation}
	  \label{eqn:bamfl.factorization}
      \begin{tikzcd}
	      S\ar[r,dotted]\ar[d,hookrightarrow] & A_{(C,S)}\ar[dotted]{r}[above]{\pi_{(C,S)}}\ar[d,hookrightarrow] & B_{(C,S)}\ar{d}[right]{\epsilon_C(B_{(C,S)}\ira\sd_3C)}\\
	      \sd_9C\ar{r}[below]{\epsilon_{\sd_3C}} & \sd_3C\ar{r}[below]{\epsilon_C} & C
	\end{tikzcd}
  \end{equation}
  commute [Lemmas \ref{lem:star.flower} and \ref{lem:natural.retractions}].  
  There exists a dotted cubical function, unique by the middle arrow monic, making the left square commute by definition of $A_{(C,S)}$.  
  The cubical set $B_{(C,S)}$ is isomorphic to a representable [Lemma \ref{lem:star.flower}]. 
  The left square above is natural in $(C,S)$.
  It therefore suffices to show that the right square above is natural in $(C,S)$.
  To that end, consider the $\STARS_9$-morphism defined by the left of the diagrams
  \begin{equation*}
    \begin{tikzcd}
	        S'\ar[dd,hookrightarrow]\ar{r}[above]{\alpha} & S''\ar[dd,hookrightarrow] &
	          A_{(C',S')}
		  \ar[d,hookrightarrow]
		  \ar{rrr}[above]{\support_{\sd_3}(\alpha,\beta)}
		  \ar{dr}[description]{\pi_{(C',A')}}
		& 
		& 
		& A_{(C'',S'')}
		  \ar[d,hookrightarrow]
		  \ar{dl}[description]{\pi_{(C'',A'')}}
		\\
		& &
		  \sd_3C'
		  \ar{d}[left]{\epsilon_{C'}}
		& B_{(C',S')}
		  \ar[r,dotted]
		  \ar{dl}[description]{\epsilon_{(C',S')}}
		& B_{(C'',S'')}
		  \ar{dr}[description]{\epsilon_{(C'',S'')}}
		& \sd_3C''
		  \ar{d}[right]{\epsilon_{C''}}
		\\
		\sd_9C'\ar{r}[below]{\sd_9\beta} & \sd_9C'' &
		  C'\ar{rrr}[below]{\beta}
		& 
		& 
		& C''
	\end{tikzcd}
  \end{equation*}

  Let $\epsilon_{(C,S)}$ denote $\epsilon_C(B_{(C,S)}\ira\sd_3C)$.
  Consider the right diagram.  
  There exists a unique dotted cubical function making the upper trapezoid commute [Lemma \ref{lem:natural.retractions}].  
  The triangles commute by (\ref{eqn:bamfl.factorization}) commutative.  
  The lower trapezoid commutes because the outer rectangle commutes and the cubical functions of the form $\pi_{(C,A)}$ are epi.
  The desired naturality of the right square in (\ref{eqn:bamfl.factorization}) follows.
\end{proof}

Call the $\CUBICALSETS^{\indexcat{1}}$-objects and $\CUBICALSETS^{\indexcat{1}}$-morphisms \textit{$\indexcat{1}$-cubical sets} and \textit{$\indexcat{1}$-cubical functions}.  
For each $\indexcat{1}$-cubical set $A$ and cubical set $B$, we write $A\otimes B$ for the $\indexcat{1}$-cubical set naturally defined by the rule $(A\otimes B)(g)=(A(g))\otimes B$. 
Write $(\CUBICALSETS^{\indexcat{1}})_0$ for the category $\CUBICALSETS^{\indexcat{1}}$.
Redefine $\CUBICALSETS^{\indexcat{1}}$ to be the $\CUBICALSETS$-enriched category with enrichment naturally defined by
$$\CUBICALSETS^{\indexcat{1}}(A,B)=(\CUBICALSETS^{\indexcat{1}})_0(A\otimes\BOX[-],B):\OP{\BOX}\ra\SETS.$$

A \textit{based cubical set} is a pair $(C,v)$ of cubical set $C$ and $v\in C_0$.
For a cubical set $C$ with a unique vertex, we write $\star$ for that unique vertex so that $(C,\star)$ denotes a based cubical set.  

\begin{eg}
  Each monoid $M$ determines the based cubical set $(\nerve\,M,\star)$.  
\end{eg}

A \textit{based cubical function} $\psi:(A,a)\ra(B,b)$ from a based cubical set $(A,a)$ to a based cubical set $(B,b)$ is a cubical function $\psi:A\ra B$ such that $\psi_0(a)=b$.  
We sometimes regard based cubical functions $(C,v)$ as $[1]$-cubical sets sending $0\ra 1$ to the cubical function $\star\ra C$ whose image contains $v$ and based cubical functions as $[1]$-cubical functions between such $[1]$-cubical sets.  
\subsection{Comparisons}\label{sec:space.comparisons}
Consider the left and middle of the solid diagrams
\begin{equation*}
  \begin{tikzcd}
	  \BOX_1\ar[r]\ar[d,hookrightarrow]\ar[dr,phantom,near start,"I"] & \TOP\\
	  \BOX\ar[ur,dotted]\ar{r}[below]{\BOX[-]} & \CUBICALSETS\ar[dotted]{u}[right]{|-|}
  \end{tikzcd}
  \quad
  \begin{tikzcd}
	  \BOX_1\ar[r]\ar[d,hookrightarrow]\ar[dr,phantom,near start,"II"]  & \DITOP\\
	  \BOX\ar[ur,dotted]\ar{r}[below]{\BOX[-]} & \CUBICALSETS\ar[dotted]{u}[right]{\direalize{\;-\;}}
  \end{tikzcd}
  \quad
    \begin{tikzcd}
	    \CUBICALSETS\ar{r}[above]{\direalize{\;-\;}}\ar[d,equals] & \DITOP\ar[d]\\
	    \CUBICALSETS\ar{r}[below]{|-|} & \TOP
    \end{tikzcd}
\end{equation*}
in which the top horizontal functors send the $\BOX$-morphism $\delta_{\pm}$ to the respective maps $\star\ra\I$ and $\star\ra\vec{\I}$ both having image $\{\half\pm\half\}$.  
Regard $\DITOP$ as Cartesian monoidal.  
There exist unique dotted diagonal monoidal functors making I and II above commute by $\BOX$ the free symmetric monoidal category on $\BOX_1$ with unit $[0]$.  
We can define unique cocontinuous \textit{topological realization} and \textit{directed realization} left Kan extensions $|-|$ and $\direalize{-}$ of the diagonal functors along the Yoneda embedding, making the left and middle diagrams commute.   
The right diagram, in which the right vertical arrow is the forgetful functor, commutes because that forgetful functor is both monoidal and cocontinuous.

\begin{eg}
  \label{eg:hypercube.quadrangulations}
  We can make the identifications
  $$|\BOX\boxobj{n}|=\I^n\quad \direalize{\BOX\boxobj{n}}=\vec{\I}^n$$
  along the continuous function that naturally sends each vertex $(x_1,\ldots,x_n)\in\boxobj{n}\subset|\BOX\boxobj{n}|$ to $(x_1,\ldots,x_n)\in\I^n$.  
\end{eg}

Directed realization preserves embeddings; the proof, a straightforward adaptation of a proof under the usual definition of cubical sets \cite[Theorem 6.19]{krishnan2015cubical}, is ommitted.

\begin{prop}
  \label{prop:realization.preserves.embeddings}
  For each monic cubical function $\iota$, $\direalize{\iota}$ is a stream embedding.
\end{prop}

\begin{eg}
  There exists a stream embedding of the form
  $$\direalize{A\otimes B\ira A\times B}:\left(\direalize{A}\times\direalize{B}\right)\ira\direalize{(A\times B)},$$
  natural in cubical sets $A$ and $B$, where $A\otimes B\ira A\times B$ denotes the natural inclusion (\ref{eqn:submonoidal.product}) from the tensor product $A\otimes B$ in $\CUBICALSETS$ into the Cartesian product $A\times B$ in $\CUBICALSETS$.  
\end{eg}

For each cubical set $C$, write $\dihomeo_{C;k+1}$ for the component 
$$\dihomeo_{C;k+1}:\direalize{\sd_{k+1}C}\cong\direalize{C}$$
of the natural isomorphism defined by the following proposition.

\begin{prop}
  \label{prop:directed.realizations}
  The following diagram
  \begin{equation*}
  \begin{tikzcd}
	  \CUBICALSETS\ar[dotted]{r}[above]{\direalize{\;-\;}}\ar{d}[left]{\sd_{k+1}}	& \DITOP\ar[d,equals]\\ 
	  \CUBICALSETS\ar[dotted]{r}[below]{\direalize{\;-\;}} & \DITOP
  \end{tikzcd}
  \end{equation*}
  commutes up to a natural isomorphism whose $\BOX\boxobj{n}$-component $\direalize{\sd_{k+1}\BOX\boxobj{n}}\cong\direalize{\BOX\boxobj{n}}$ is linear on each cell and sends each geometric vertex $v\in\multiboxobj{k+1}{n}$ in $|\BOX\multiboxobj{k+1}{n}|$ to $\nicefrac{v}{k+1}\in\I^n$.
\end{prop}

The following lemma is the main method of obtaining information about edge orientations on a cubical set from the circulation on a directed realization. 
Regarding a point $x$ in a stream $X$ as a stream map $\star\ra X$ whose image contains $x$, $\support_{|\BOX[-]|}(x,L)$ is the minimal Boolean interval $I$ in a finite distributive lattice $L$ such that $x\in|\BOX[I]|$.  

\begin{lem}
  \label{lem:orientations}
  Fix a $\DISLATS$-object $L$.  
  Consider $x\leqslant_{\direalize{\;\BOX[L]\;}}y$.  
  Then
  $$\min\support_{|\BOX[-]|:\DISLATS\ra\TOP}(x,L)\leqslant_L\min\support_{|\BOX[-]|:\DISLATS\ra\TOP}(y,L).$$
\end{lem}
\begin{proof}
  For brevity, let $F$ be the functor $|\BOX[-]|:\DISLATS\ra\TOP$.
  In the case $L=\boxobj{n}$, 
  \begin{equation*}
	  \min\support_{F}(x,\boxobj{n})
	  =(\lfloor x_1\rfloor,\ldots,\lfloor x_n\rfloor)\\
	  \leqslant_{\boxobj{n}} (\lfloor y_1\rfloor,\ldots,\lfloor y_n\rfloor)\\
	  =\min\support_{F}(y,L).
  \end{equation*}
  The general case follows from transitivity of preorders.
\end{proof}

\begin{lem}
  \label{lem:geometric.first.vertex.map}
  For each $n$, the following commutes in $\SETS$.
  \begin{equation*}
    \begin{tikzcd}
	    (\boxobj{n})^{[1]}=\sd_2\boxobj{n}\ar[d,hookrightarrow]\ar{r}[above]{(-)^{\delta_{\mins}}} & \boxobj{n}\\
	    \direalize{\BOX[\sd_2\boxobj{n}]}\ar{r}[below]{\dihomeo_{\BOX\boxobj{n};2}} & \direalize{\BOX\boxobj{n}}\ar{u}[right]{x\mapsto\min\,\support_{|\BOX[-]|:\BOX\ra\TOP}(x,\boxobj{n})}
    \end{tikzcd}
  \end{equation*}
\end{lem}
\begin{proof}
  For brevity, let $F$ be the functor $|\BOX[-]|:\BOX\ra\TOP$.
  For each $x\in\sd_2\boxobj{n}=(\boxobj{n})^{[1]}$,
  \begin{align*}
    \min\,\support_{F}(\dihomeo_{\BOX\boxobj{n};2}(x),\boxobj{n})
    &= \min\,\support_{F}((\nicefrac{x_1}{2},\nicefrac{x_2}{2},\ldots,\nicefrac{x_n}{2}),\boxobj{n})\\
    &= (\lfloor\nicefrac{x_1}{2}\rfloor,\lfloor\nicefrac{x_2}{2}\rfloor,\ldots,\lfloor\nicefrac{x_n}{2}\rfloor)\\
    &= [1]^{\delta_{\mins}}(x).
  \end{align*}
\end{proof}

Let $\csing$ denote the right adjoint to $\direalize{-}:\CUBICALSETS\ra\DITOP$, naturally defined by  
$$(\csing X)_n=\DITOP(\direalize{\BOX\boxobj{n}},X).$$
Call $\csing X$ the \textit{directed singular cubical set} of the stream $X$. 
For brevity, we henceforth write $\catfont{S}$ for the monad $\sing\,\direalize{-}$ of the adjunction $\direalize{-}\dashv\sing$.  

\section{Homotopy}\label{sec:homotopy}
This section formalizes and compares different homotopy theories.  
Section \S\ref{sec:abstract.homotopy} fixes some definitions of homotopy in terms of an abstract \textit{interval object}.  
Sections \S\ref{sec:categorical.homotopy}, \S\ref{sec:cubical.homotopy}, and \S\ref{sec:continuous.homotopy} explore specific instances of abstract homotopy, whether classical, directed, or categorical and whether algebraic, cubical, or continuous.
Section \S\ref{sec:homotopical.comparisons} compares the different homotopy theories.  
In particular, section \S\ref{sec:directed.homotopical.comparisons} gives some of the main results, albeit without proof, and some subsequent calculations.
Observations about the classical homotopy theory of cubical sets are essentially formal but included for completeness, given that our operating definition of cubical sets is not standard.  
Observations about classical homotopy theories of small categories and topological spaces, well-known, are included for comparison with their directed counterparts.  

	\subsection{Abstract}\label{sec:abstract.homotopy}
	The simplest way to discuss the variety of homotopy theories of interest is in terms of abstract interval objects.  
	The purpose of this section is to fix notation and terminology for standard concepts at this level of abstraction.  
	Fix a closed monoidal category $\modelcat{1}$ with terminal unit.    
	Fix an \textit{interval object} $\intervalobject{1}$ in $\modelcat{1}$, which we take in this paper to mean a functor $\INTERVALCAT\ra\modelcat{1}$ sending $[0]$ to a terminal object.

	\begin{eg}
	  \label{eg:classical.continuous.interval.object}
	  The interval object in $\TOP$ naturally sending $\delta_{\pm}$ to the functions
	  $$\{\half\pm\half\}\ira\I$$
	  is the prototypical example of an interval object.  
	  Much of homotopy theory on $\TOP$ generalizes to a category equipped with an interval object.  
	\end{eg}

	We fix some general terminology for standard concepts, like relative homotopy and homotopy equivalences, in terms of the interval object $\intervalobject{1}$.
	Fix an $\modelcat{1}$-enriched category $\cat{1}$ tensored and cotensored over $\modelcat{1}$.  
	For convenience, we will identify an $\cat{1}$-object $o$ with $\mathfrak{i}([0])\otimes o$ and $\mathfrak{o}^{\mathfrak{i}([0])}$ along natural $\cat{1}$-isomorphisms.  
	For a pair of parallel $\cat{1}$-morphisms $\zeta_1,\zeta_2:o_1\ra o_2$, \textit{left and right $\intervalobject{1}$-homotopies} from $\zeta_1$ to $\zeta_2$ are choices of dotted $\cat{1}$-morphisms respectively making I,II commute in
	\begin{equation*}
	  \begin{tikzcd}
		  o_1\amalg o_1
		    \ar{d}[left]{o_1\otimes\mathfrak{i}(\delta_-)\amalg o_1\otimes\mathfrak{i}(\delta_+)}\ar{r}[above]{\zeta_1\amalg\zeta_2} 
		    \ar[dr,phantom,"I",very near start]
		  & o_2
			\\
			o_1\otimes\mathfrak{i}([1])\ar[ur,dotted] & {}
	  \end{tikzcd}
	  \quad
	  \begin{tikzcd}
		  o\ar{rr}[above]{o_2^{\intervalobject{1}([1]\ra[0])}\zeta_i\iota}\ar{d}[left]{\iota} 
		  \ar[dr,phantom,"III",near start]
		  & {} & o_2^{\mathfrak{i}([1])}\ar{d}[right]{o_2^{\mathfrak{i}(\delta_-)\amalg\mathfrak{i}(\delta_+)}}
		  \\  	
		  o_1\ar[urr,dotted]\ar{rr}[below]{\zeta_1\times\zeta_2} & {} & 
		  o_2^2
		  \ar[ul,phantom,"II",near start]
	  \end{tikzcd}
	\end{equation*}

	Natural bijections $\cat{1}(o_1\otimes\mathfrak{i}([1]),o_2)\cong\cat{1}(o_1,o_2^{\mathfrak{i}([1])})$ defined by $\cat{1}$ closed monoidal give a 1-1 correspondence between left $\intervalobject{1}$-homotopies and right $\intervalobject{1}$-homotopies.
	Write $\zeta_1\futurehomotopic{\intervalobject{1}}\zeta_2$ to denote a (left or right) $\intervalobject{1}$-homotopy from $\zeta_1$ to $\zeta_2$ or the existence of such an $\intervalobject{1}$-homotopy.
	Say that the dotted right $\intervalobject{1}$-homotopy on the right side is \textit{relative} a morphism $\iota:o\ra o_1$ if additionally III commutes for $i=1$ or equivalently for $i=2$. 
	We will repeatedly use the formal fact that there exists an $\mathfrak{i}$-homotopy (relative a $\cat{1}$-morphism $\iota:o\ra o_1$) between a pair of parallel $\cat{1}$-morphisms $\zeta_1,\zeta_2:o_1\ra o_2$ (whose precomposites with $\iota$ coincide) natural in $\zeta_1,\zeta_2$ and a choice of dotted lift making IV (and V) commute in following diagram commute, where $i=1$ or equivalently $i=2$:
        \begin{equation*}
          \begin{tikzcd}
		  o\ar{d}[left]{\iota}\ar{r}[above]{\zeta_i\iota} & o_2\ar[dr,phantom,near start,"V"]\ar[d,equals]\ar{r}[above]{o_2^{\mathfrak{i}([1]\ra[0])}} & o_2^{\mathfrak{i}([1])}\ar{d}[right]{o_2^{\mathfrak{i}(\delta_-)\amalg\mathfrak{i}(\delta_+)}}\\
		  o_1\ar{r}[below]{\zeta_1\times\zeta_2} & o_2\ar[ur,dotted]\ar[r,equals] & o_2\ar[ul,phantom,near start,"IV"]
	  \end{tikzcd}
	\end{equation*}

	An $\cat{1}$-morphism $\alpha:o_1\ra o_2$ is an \textit{$\intervalobject{1}$-equivalence} if there exists an $\cat{1}$-morphism $\beta:o_2\ra o_1$ with $\id_a\futurehomotopic{\mathfrak{i}}\beta\alpha$ and $\id_b\futurehomotopic{\mathfrak{i}}\alpha\beta$.
	Define the interval object $\mathfrak{i}_n$, informally the n-fold zig-zag of $\mathfrak{i}$, by $\intervalobject{1}_{0}=\mathfrak{i}$ and the following commutative diagrams among which the first is co-Cartesian:
	\begin{equation*}
	  \begin{tikzcd}
		  \mathfrak{i}([1])\ar{r}[above]{\mathfrak{i}(\delta_{\mins})_*\iota_n(\delta_{\mins})} & \mathfrak{i}_{n+1}([1])\\  	
		\mathfrak{i}([0])\ar{u}[left]{\mathfrak{i}(\delta_-)}
	    \ar[ur,phantom,"\urcorner",very near start]
		\ar{r}[below]{\mathfrak{i}_n(\delta_-)} & \mathfrak{i}_{n}([1])\ar{u}[description]{(\iota_n(\delta_-))_*\iota(\delta_-)} 
	 \end{tikzcd}
	 \quad
	 \begin{tikzcd}
		\mathfrak{i}_{n+1}([0])\ar{r}[above]{\mathfrak{i}_{n+1}(\delta_-)} & \mathfrak{i}_{n+1}([1])\\  	
		\mathfrak{i}([0])\ar[u,equals]\ar{r}[below]{\mathfrak{i}(\delta_+)} & \mathfrak{i}([1])\ar{u}[description]{\mathfrak{i}(\delta_{\mins})_*\iota_n(\delta_{\mins})} 
	 \end{tikzcd}
	 \quad
	 \begin{tikzcd}
		\mathfrak{i}_{n+1}([0])\ar{r}[above]{\mathfrak{i}_{n+1}(\delta_+)} & \mathfrak{i}_{n+1}([1])\\  	
		\mathfrak{i}([0])\ar[u,equals]\ar{r}[below]{\mathfrak{i}(\delta_+)} & \mathfrak{i}_n([1])\ar{u}[description]{(\iota_n(\delta_-))_*\iota(\delta_-)}  
	 \end{tikzcd}
	\end{equation*}

	An \textit{$\mathfrak{i}_*$-homotopy} is a $\mathfrak{i}_n$-homotopy for some $n$.
	Write $\zeta_1\dhomotopic{\mathfrak{i}}\zeta_2$ to denote an $\intervalobject{1}_*$-homotopy or the existence of such an $\intervalobject{1}_*$-homotopy from $\zeta_1$ to $\zeta_2$.
	In other words, $\dhomotopic{\mathfrak{i}}$ is the congruence on $\cat{1}$ generated by the relation $\futurehomotopic{\mathfrak{i}}$ on morphisms. 
	An \textit{$\mathfrak{i}_*$-equivalence} is an $\mathfrak{i}_n$-equivalence for some $n$, or equivalently  a $\cat{1}$-morphism representing an isomorphism in the quotient category $\cat{1}/{\dhomotopic{\mathfrak{i}}}$. 

	\begin{lem}
	  \label{lem:localization}
	  Consider the following data:
	  \begin{enumerate}
		  \item closed monoidal category $\modelcat{1}$ with terminal unit
		  \item interval object $\intervalobject{1}$ in $\modelcat{1}$ sending $\delta_{\mins}$ and $\delta_{\pls}$ to $\intervalobject{1}_*$-equivalences  
		  \item category $\cat{1}$ enriched, tensored, and cotensored over $\modelcat{1}$
	  \end{enumerate}
	  The quotient functor $\cat{1}\ra\cat{1}/\dhomotopic{\mathfrak{i}}$ is localization by  the $\mathfrak{i}_*$-equivalences.
	\end{lem}
	\begin{proof}
	  Fix a functor $F:\cat{1}\ra\cat{2}$ mapping the $\mathfrak{i}_*$-equivalences to isomorphisms.  
 	  Consider a pair of $\dhomotopic{\mathfrak{i}}$-equivalent $\cat{1}$-morphisms $\alpha,\beta:o_1\ra o_2$.  
	  Then there exists $n\gg 0$ and $\eta:\alpha\futurehomotopic{\mathfrak{i}_n}\beta$.
	  Both $\delta_{\mins}$ and $\delta_{\pls}$ admit $\sigma$ as a retraction.  
	  Hence both $\intervalobject{1}_n(\delta_{\mins})$ and $\intervalobject{1}_n(\delta_{\pls})$ represent inverses to $\intervalobject{1}_n(\sigma)$ in $\modelcat{1}/\dhomotopic{\intervalobject{1}}$ for the case $n=0$ and hence for the general case by induction.
	  Hence both $o_1\otimes\intervalobject{1}_n(\delta_{\mins})$ and $o_1\otimes\intervalobject{1}_n(\delta_{\pls})$ both represent inverses to the same $\cat{1}$-morphism, $o_1\otimes\intervalobject{1}_n(\sigma)$, in $\cat{1}/\dhomotopic{\intervalobject{1}}$.
	  Hence in the diagram
		\begin{equation*}
			\begin{tikzcd}
				  Fo_1
				      \ar{dr}[description]{F(o_1\otimes\mathfrak{i}_n(\delta_-))}\ar{rr}[above]{F\alpha}\ar[dd,equals] 
				& 
				& Fo_2
				    \ar[dd,equals]
				\\
				& F(o_1\otimes\mathfrak{i}_n[1])
				\ar{ur}[description]{F\eta}\ar{ur}[description]{F\eta}
				\ar{dr}[description]{F\eta}
				\\
				  Fo_1
				  \ar{rr}[below]{F\beta}\ar{ur}[description]{F(o_1\otimes\mathfrak{i}_n(\delta_+))}
				& 
				& Fo_2,
			\end{tikzcd}
		\end{equation*}
		the left triangle commutes; the top and bottom triangle commute by our choice of $\eta$; the right triangle trivially commutes; the outer square commutes; and hence $F\alpha=F\beta$.  
		Thus $F$ factors through the quotient functor $\cat{1}\ra\cat{1}/\dhomotopic{\mathfrak{i}}$.  
	\end{proof}

	Let $[o_1,o_2]_{\intervalobject{1}}=\quotient{\cat{1}(o_1,o_2)}{\dhomotopic{\intervalobject{1}_*}}$, the $\SETS$-coequalizer
	$$\xymatrix{\mathscr{X}\left(o_1,o_2^{\mathfrak{i}([1])}\right)\ar@<.7ex>[rrr]^{\mathscr{X}(o_1,o_2^{\mathfrak{i}(\delta_+)})}\ar@<-.7ex>[rrr]_{\mathscr{X}(o_1,o_2^{\mathfrak{i}(\delta_-)})} & & & \mathscr{X}(o_1,o_2)\ar[r] & {[o_1,o_2]_{\intervalobject{1}}}}.$$

	A natural transformation $\mathfrak{i}'\ra\mathfrak{i}''$ of interval objects implies that
	$$\graph{(\futurehomotopic{\mathfrak{i}''})}\subset\graph{(\futurehomotopic{\mathfrak{i}'})}.$$

	\begin{eg}
	 We have the following chain
	$$\graph{(\futurehomotopic{\mathfrak{i}_0})}\subset\graph{(\futurehomotopic{\mathfrak{i}_1})}\subset\graph{(\futurehomotopic{\mathfrak{i}_2})}\subset\cdots\graph{(\dhomotopic{\mathfrak{i}})}$$
	for each interval object $\mathfrak{i}$ in a cocomplete closed monoidal category.
	\end{eg}

	Define the interval object $\mathfrak{d}$ as the composite
	$$\mathfrak{d}=\BOX[-](\INTERVALCAT\ira\BOX):\INTERVALCAT\ra\CUBICALSETS.$$

	\begin{eg}
	  \label{eg:classical.homotopy}
	  The interval object defining classical homotopy [Example \ref{eg:classical.continuous.interval.object}] is
	  $$|\mathfrak{d}|:\INTERVALCAT\ra\TOP.$$
	  The different homotopies in the classical setting coincide: $|\mathfrak{d}|\cong|\mathfrak{d}|_1\cong|\mathfrak{d}|_2\cdots$ and
	  $$\futurehomotopic{|\mathfrak{d}|}=\futurehomotopic{|\mathfrak{d}|_1}=\futurehomotopic{|\mathfrak{d}|_2}=\cdots=\dhomotopic{|\mathfrak{d}|}.$$
	\end{eg}

	We recall and compare homotopy theories based on the interval objects $\mathfrak{d}$, $\direalize{\mathfrak{d}}$, $|\mathfrak{d}|$ [Example \ref{eg:classical.continuous.interval.object}], $\mathfrak{h}=(\TOP\ira\DITOP)|\mathfrak{d}|$, $\Tau_1\mathfrak{d}:\INTERVALCAT\ira\CATS$, and $\Pi_1\mathfrak{d}$.  

\subsection{Algebraic}\label{sec:categorical.homotopy}
We recall three homotopy theories on the category $\CATS$ of small categories and functors between them, in order of increasing refinement.
All three of these homotopy theories coincide on the full subcategory $\GROUPOIDS$ of small groupoids. 

\subsubsection{Classical}
The class of \textit{Thomason weak equivalences} is the smallest retract-closed class $\mathscr{W}$ of $\CATS$-morphisms having the 2-out-of-3 property and containing all terminal functors such that a functor $\alpha:\smallcat{1}\ra\smallcat{2}$ lies in $\mathscr{W}$ whenever the induced functor $\beta\alpha/o\ra\beta/o$ lies in $\mathscr{W}$ for each functor $\beta:\smallcat{2}\ra\smallcat{3}$ and $\smallcat{3}$-object $o$ \cite[Theorem 2.2.11]{cisinski2004localisateur}.
The localization of $\CATS$ by the Thomason weak equivalences exists \cite{thomason1980cat} and will be referred to as the \textit{classical homotopy category of $\CATS$}.

\begin{eg}
  A sufficient and intrinsic condition for a $\CATS$-morphism
  $$\zeta:\smallcat{1}\ra\smallcat{2}$$ 
  to be a Thomason weak equivalence is if $o/\zeta$ has a terminal object for each $\smallcat{1}$-object $o$ by Quillen's Theorem A.    
\end{eg}

It is difficult to give a complete characterization of the Thomason weak equivalences that is at once explicit and intrinsic, at least without reference to the simplex category $\DEL$ (c.f. \cite{hoyo2008subdivision}.) 
Thomason weak equivalences can be defined more generally for $n$-fold functors between $n$-fold categories (c.f. Theorem \ref{thm:1-ditypes}).  

\subsubsection{Directed}
Much early work in directed homotopy theory went into generalizing categorical equivalences between groupoids to notions of equivalences between small categories that preserve computational behavior of interest (e.g. \cite{fajstrup2004components,goubault2007components,goubault2010future} and [Example \ref{eg:future.equivalences}]).
We examine one of the weakest such generalizations.  
Let $\Tau_1\mathfrak{d}_n$ denote the interval object
$$\Tau_1\mathfrak{d}_n=\Tau_1(\mathfrak{d}_n)=(\Tau_1\mathfrak{d})_n:\INTERVALCAT\ra\CATS.$$

In particular, $\Tau_1\mathfrak{d}$ is the cannonical interval object
$$\Tau_1\mathfrak{d}:\INTERVALCAT\ira\CATS.$$

The homotopy theory in which weak equivalences are the $(\Tau_1\mathfrak{d})_*$-equivalences \cite{minian2002cat}, as well a slightly weaker homotopy theory \cite{hoff1974categories} in which homotopy is defined by a single path object in terms of $\mathfrak{d}_1,\mathfrak{d}_2,\ldots$, have been studied previously. 
The $(\Tau_1\mathfrak{d})_*$-equivalences, while not the weak equivalences of a model structure, are the weak equivalences of a \textit{$\Lambda$-cofibration category} \cite{hoff1974categories} structure on $\CATS$.  
While each $(\Tau_1\mathfrak{d})_*$-equivalence is a Thomason weak equivalence, not each Thomason weak equivalence is a $(\Tau_1\mathfrak{d})_*$-equivalence.

\begin{eg}
  For parallel $\CATS$-morphisms $\alpha,\beta$, a $\Tau_1\mathfrak{d}$-homotopy
  $$\alpha\leadsto\beta$$
  is exactly a natural transformation $\alpha\ra\beta$.  
  In particular, a (left or right) adjoint in $\CATS$ is a $\Tau_1\mathfrak{d}_1$-equivalence.
\end{eg}

\begin{eg}
  \label{eg:future.equivalences}
  Consider a functor of small categories
  $$F:\smallcat{1}\ra\smallcat{2}.$$
  The functor $F$ is sometimes referred to as a \textit{future equivalence} \cite{grandis2005shape} if $F$ is a $\Tau_1\mathfrak{d}$-equivalence and a \textit{past equivalence} \cite{grandis2005shape} if $\OP{F}$ is a future equivalence.
  Future equivalences and past equivalences preserve certain properties of interest in state space analyses, such as terminal objects and initial objects respectively \cite{goubault2010future}.
\end{eg}

Based $\Tau_1\mathfrak{d}_*$-homotopy is trivial on monoid homomorphisms.

\begin{lem}
  \label{lem:based.d-homotopic.monoid.homomorphisms}
  Consider a pair of parallel monoid homomorphisms 
  $$\alpha,\beta:M\ra N.$$
  Then $\alpha=\beta$ if $\alpha$ and $\beta$ are $\Tau_1\mathfrak{d}_*$-homotopic relative $\star\ra M$. 
\end{lem}
\begin{proof}
	A right $\Tau_1\mathfrak{d}_n$-homotopy $\eta:\alpha\dhomotopic{\Tau_1\mathfrak{d}}\beta$ relative $\star\ra M$ factors through $N^{\Tau_1\mathfrak{d}_n[1]\ra\star}$, by relativity and the fact that there exists at most one object in $M$ and thus $\alpha=N^{\delta_{\mins}}\eta=N^{\delta_{\pls}}\eta=\beta$.  
\end{proof}

Unbased $\Tau_1\mathfrak{d}_*$-homotopy on monoid homomorphisms generalizes the conjugacy relation.

\begin{lem}
  \label{lem:conjugacy.classes}
  For a monoid $\tau$, there exists a natural bijection
  $$[\N,\tau]_{\Tau_1\mathfrak{d}}\cong\left(\quotient{\tau}{\equiv}\right)$$
  of sets, where $\equiv$ is the smallest equivalence relation on the underlying set of $\tau$ such that $x\equiv y$ if there exists an element $z$ in $\tau$ such that $xz=zy$.
\end{lem}
\begin{proof}
  Consider a pair of monoid homomorphisms
  $$\alpha,\beta:\N\ra\tau.$$

  Every $\Tau_1\mathfrak{d}$-homotopy $\alpha\futurehomotopic{\Tau_1\mathfrak{d}}\beta$, a functor
  $$\eta:\N\times[1]\ra\tau$$
  such that $\eta(-,0)=\alpha$ and $\eta(-,1)=\beta$, is uniquely determined by the element $x=\eta(1_{\star},0\ra 1)$ in $\tau$.  
  Conversely, every $x\in\tau$ such that $\alpha(1)x=x\beta(1)$ uniquely determines a $\Tau_1\mathfrak{d}$-homotopy $\alpha\futurehomotopic{\Tau_1\mathfrak{d}}\beta$ sending $(1_{\star},0\ra 1)$ to $x$.  
  It therefore follows that the bijection $\CATS(\N,\tau)\cong\tau$ sending each homomorphism $\zeta$ to $\zeta(1)$ passes to the desired natural isomorphism.  
\end{proof}

A monoid $M$ is \textit{cancellative} if $x_1yx_2=x_1zx_2$ for some elements $x_1,x_2,y,z\in M$ implies $y=z$.  
For a commutative monoid $\tau$, the set $[\N,\tau]_{\Tau_1\mathfrak{d}}$ of conjugacy classes of elements in $\tau$ naturally forms a commutative monoid with multiplication inherited from $\tau$.  
The commutative monoid $[\N,\tau]_{\Tau_1\mathfrak{d}}$ is the \textit{maximal cancellative quotient} of $\tau$, the image of $\tau$ under the left adjoint to the full inclusion from the category of cancellative commutative monoids into the category of commutative monoids and monoid homomorphisms.  

\subsubsection{Categorical}
There exist natural isomorphisms
$$\Pi_1\mathfrak{d}\cong\Pi_1(\mathfrak{d}_n)=(\Pi_1\mathfrak{d})_n,\quad n=0,1,2,\ldots$$

A categorical equivalence between small categories is exactly a $\Pi_1\mathfrak{d}$-equivalence.
Every categorical equivalence is a $\Tau_1\mathfrak{d}$-equivalence because localization defines a natural transformation $\Tau_1\mathfrak{d}\ra\Pi_1\mathfrak{d}$.

\subsection{Cubical}\label{sec:cubical.homotopy}
Define $\pi_0C$ as the $\SETS$-coequalizer
$$\xymatrix{C_1\ar@<.7ex>[rrr]^{C(\delta_-)}\ar@<-.7ex>[rrr]_{C(\delta_+)} & & & C_0\ar[r] & \pi_0C}$$
natural in cubical sets $C$.    

\begin{eg}
  We can make the natural identification $[A,B]_{\mathfrak{d}}=\pi_0\CUBICALSETS^{\indexcat{1}}(A,B)$.
\end{eg}

For each based cubical set $(C,v)$, let $\tau_n(C,v)$ be the set
$$\tau_n(C,v)=\pi_0\Omega^n(C,v).$$

\begin{eg}
  The set $\tau_n(C,v)$ is the set of cubical functions 
  $$\BOX\boxobj{n}\ra C$$
  sending $\partial\BOX\boxobj{n}$ to the minimal subpresheaf of $C$ having the vertex $v$, up to the equivalence relation identifying two such cubical functions $\alpha,\beta$ whenever $\alpha\dhomotopic{\mathfrak{d}}\beta$ relative $\partial\BOX\boxobj{n}\ira\BOX\boxobj{n}$.
\end{eg}

A \textit{cubical homotopy equivalence} is a $\mathfrak{d}_*$-equivalence.
Cubical homotopy equivalences generalize to \textit{classical weak equivalences} of cubical sets, which we take to mean cubical functions $\psi$ inducing homotopy equivalences $|\psi|$ between topological realizations.
Classical weak equivalences are part of a model structure, which we call the \textit{classical model structure} on $\CUBICALSETS$, defined by the following proposition.

\begin{prop}
  \label{prop:cubical.model.structure}
  There exists a model structure on $\CUBICALSETS$ in which \ldots
  \begin{enumerate}
    \item \ldots the weak equivalences are the classical weak equivalences
    \item \ldots the cofibrations are the monos
  \end{enumerate}
\end{prop}

A proof of the proposition is given at the end of \S\ref{sec:triangulations}.  
A \textit{fibrant} cubical set will simply refer to a fibrant object in this classical model structure on cubical sets.  
A classical weak equivalence between fibrant cubical sets is precisely a cubical homotopy equivalence between fibrant cubical sets. 
The fundamental groupoid $\Pi_1$ is a classical homotopy invariant in the sense of the following proposition, whose proof is given at the end of \S\ref{subsubsec:classical.comparisons}.

\begin{prop}
  \label{prop:equivalent.fundamental.groupoids}
  For each classical weak equivalence $\psi:A\ra B$ of cubical sets, 
  $$\Pi_1\psi:\Pi_1A\ra\Pi_1B$$ 
  is a categorical equivalence.
\end{prop}

As a consequence, cubical nerves of small groupoids are fibrant (c.f. Proposition \ref{prop:nerves}.)

\begin{cor}
  \label{cor:fibrant.groupoid.nerves}
  For each small groupoid $\groupoid{1}$, $\nerve\,\groupoid{1}$ is fibrant.
\end{cor}
\begin{proof}
  Consider the solid functors in the left of the diagrams
  \begin{equation*}
	  \begin{tikzcd}
		  A\ar[d,hookrightarrow]\ar{r}[above]{\psi} & \nerve\,\groupoid{1}\\
		  B\ar[dotted]{ur}[description]{\phi^*}
	  \end{tikzcd}
	  \quad
	  \begin{tikzcd}
		  \Pi_1A\ar{d}[description]{\cong}\ar{r}[above]{\Pi_1\psi} & \groupoid{1}\\
		  \Pi_1B\ar[dotted]{ur}[description]{\phi}
	  \end{tikzcd}
  \end{equation*}
  Suppose $A\ira B$ is an acyclic cofibration in the classical model structure.  
  There exists a dotted functor $\phi$ making the entire right diagram commute by $\Pi_1(A\ira B)$ an equivalence of small categories [Proposition \ref{prop:equivalent.fundamental.groupoids}] that is injective on objects.  
  Thus the left diagram, in which $\phi^*$ denotes the adjoint of $\phi$, commutes.
\end{proof}

Extend classical $1$-cohomology to \textit{directed $1$-cohomology}
$$\cohomology{1}(C;\tau)=[C,\cnerve\tau]_{\mathfrak{d}}=\pi_0(\cnerve\tau)^C=[\Tau_1C,\tau]_{\Tau_1\mathfrak{d}},$$
a commutative monoid natural in commutative monoids $\tau$ and cubical sets $C$.

\begin{eg}
  For an Abelian group $\pi$, the Abelian group
  $$\cohomology{1}(C;\pi)=[C,\cnerve\pi]_{\mathfrak{d}}=\pi_0(\cnerve\pi)^C$$
  defines the first cubical cohomology of the cubical set $C$.  
  Classical cubical $1$-cohomology sends classical weak equivalences to isomorphisms by $\cnerve\pi$ fibrant [Corollary \ref{cor:fibrant.groupoid.nerves}]. 
\end{eg}

Group-completion $\tau\ra\tau[\tau^{-1}]$ induces a monoid homomorphism
$$\cohomology{1}(C;\tau\ra\tau[\tau]^{-1}):\cohomology{1}(C;\tau)\ra\cohomology{1}(C;\tau[\tau]^{-1})$$
from directed cohomology to classical cohomology, natural in commutative monoid coefficients $\tau$. 

\subsection{Continuous}\label{sec:continuous.homotopy}
We recall homotopy for the continuous setting.
Let $\pi_0$ denote the functor
$$\pi_0:\TOP\ra\SETS$$
naturally sending each topological space $X$ to its set of path-components.

\subsubsection{Classical}\label{sec:continuous.classical.homotopy}
We have the natural identification
$$[X,Y]_{|\mathfrak{d}|}=\pi_0Y^X.$$

A continuous function $f:X\ra Y$ is a classical weak equivalence if 
$$\pi_0f^{|C|}:\pi_0X^{|C|}\cong\pi_0Y^{|C|}.$$
for all cubical sets $C$.
The classical weak equivalences and maps having the right lifting property against all maps of the form $|\BOX[\delta_{+}\otimes\id_{\boxobj{n}}]|:\I^n\ra\I^{n+1}$ define the weak equivalences and fibrations of the \textit{q-model structure} on $\TOP$. 

\subsubsection{Directed}\label{sec:continuous.directed.homotopy}
We can make, by cocontinuity of $\direalize{-}$, the identifications
$$\direalize{\mathfrak{d}_n}=\direalize{\mathfrak{d}}_n,\quad n=0,1,\ldots.$$
A $\direalize{\mathfrak{d}}_*$-homotopy is sometimes referred to in the literature as a \textit{d-homotopy} (e.g. \cite{grandis2003directed}.)
Intuitively, a d-homotopy is a homotopy through stream maps that is additionally piecewise monotone and anti-monotone in its homotopy coordinate. 
Write $\tau_n(X,x)$ for all $\direalize{\mathfrak{d}}_*$-homotopy classes relative $\{\infty\}\ra\vec{\mathbb{S}}^n$ of based stream maps
$$(\vec{\mathbb{S}^n},\infty)\ra(X,x).$$

Fix $n\geqslant 1$.
The based directed sphere $(\vec{\mathbb{S}^n},\infty)$ admits a co-H multiplication in $\DITOP$ that passes to the standard co-H multiplication on the based $n$-sphere.  
In this manner, $\tau_n(X,x)$ naturally has the structure of a monoid just as the $n$th homotopy group $\pi_n(X,x)$ of the underlying space of $X$ based at $x$ naturally has the structure of a group.  
This monoid has been introduced elsewhere as the \textit{$n$th homotopy monoid} of a based directed topological space \cite{grandis2003directed}.  
For $n\geqslant 2$, $\tau_n(X,x)$ is commutative by an Eckmann-Hilton argument and moreover can be shown to be an Abelian group, where group inversion is given by a transposition of two suspension coordinates.    
Forgetting directionality induces a natural monoid homomorphism 
$$\tau_n(X,x)\ra\pi_n(X,x)$$

	\begin{lem}
	  \label{lem:hypercube.convexity}
	  There exists a $\direalize{\mathfrak{d}_1}$-homotopy between both projections of the form
	  $$\direalize{\BOX\boxobj{n}}^2\ra\direalize{\BOX\boxobj{n}}$$
	  natural in $\BOX$-objects $\boxobj{n}$.  
	\end{lem}

	For a pair of parallel stream maps $f,g:X\ra\vec{\I}^n$, linear interpolation
	$$h:X\times\I\ra\vec{\I}^n,\quad h(x,t)=(1-t)f+g$$
	defines a classical homotopy from $f$ to $g$ through stream maps but does not generally define a $\direalize{\mathfrak{d}_*}$-homotopy.  
	The proof requires not only natural convex structure but also natural topological semilattice structure on directed hypercubes.  

	\begin{proof}
	  Let $\pi_{1;n}$ and $\pi_{2;n}$ denote the projections
	  $$\direalize{\BOX\boxobj{n}}^2\ra\direalize{\BOX\boxobj{n}}$$
	  onto first and second factors, respectively.  
	  Linear interpolation defines $\direalize{\mathfrak{d}}$-homotopies 
	  $$\pi_{1;n}\wedge_{\direalize{\;\BOX\boxobj{n}\;}}\pi_{2;n}\futurehomotopic{\direalize{\;\mathfrak{d}\;}}\pi_{1;n},\pi_{2;n}$$ 
	  natural in $\BOX$-objects $\boxobj{n}$ because $\direalize{\BOX[-]}:\BOX\ra\DITOP$ sends each $\BOX$-morphism to a linear map of hypercubes that defines a lattice homomorphism between compact Hausdorff connected topological lattices in $\POTOP$.  
	  Concatenating these $\direalize{\mathfrak{d}}$-homotopies yields the desired $\direalize{\mathfrak{d}_1}$-homotopy.
	\end{proof}

	A simple consequence is that the cubical function
	$$\epsilon_C:\sd_3C\ra C$$
	defines a natural cubical approximation to $\dihomeo_{C;3}:\direalize{\sd_3C}\cong\direalize{C}$.

	\begin{lem}
	  \label{lem:natural.approximations}
	  There exists a $\direalize{\mathfrak{d}_1}$-homotopy
	  $$\direalize{\epsilon_C}\dhomotopic{\direalize{\;\mathfrak{d}_1\;}}\dihomeo_{C;3}:\direalize{\sd_3C}\ra\direalize{C}$$
	  natural in cubical sets $C$.
	\end{lem}
	\begin{proof}
	  There exists the desired $\direalize{\mathfrak{d}_1}$--homotopy natural in representable cubical sets $C$ [Lemma \ref{lem:hypercube.convexity}] and hence natural in general cubical sets $C$ by naturality of $\direalize{\epsilon_C}$ and $\dihomeo_{C;3}$.
	\end{proof}

	Nearby stream maps to directed realizations are $\direalize{\mathfrak{d}_{*}}$-homotopic.

	\begin{lem}
	  \label{lem:close.maps}
	  There exists a $\direalize{\mathfrak{d}_{*}}$-homotopy between stream maps 
	  $$f,g:X_{(f,g)}\ra\direalize{\sd_9C_{(f,g)}},$$
	  natural in objects $f\times g$ in the full subcategory of $(\STREAMS/\!\direalize{\sd_9-}^2)$ consisting of those objects $f\times g:X_{(f,g)}\ra\direalize{\sd_9C_{(f,g)}}^2$ for which $X_{(f,g)}$ is covered by open substreams each of which has images under $f$ and $g$ that lie in the open star of the same vertex.
	\end{lem}
	\begin{proof}
	  For a stream map $e:X\ra\direalize{\sd_9C}$ and substream $U\subset X$, let
	  $$e_U=e(U\ira X):U\ira\direalize{\sd_9C}.$$

	  Let $\cat{1}$ denote the category defined by the proposition. 
	  Let $f\times g:X_{(f,g)}\ra\direalize{\sd_3^2C_{(f,g)}}^2$ denote a $\cat{1}$-object.
	  Let $\mathscr{O}_{(f,g)}$ be the category whose objects are all substreams of $X_{(f,g)}$ whose images under $f$ and $g$ lie in the open star of the same vertex and whose morphisms are all inclusions between such substreams. 
	  Consider a commutative square of the form
	  \begin{equation*}
	    \begin{tikzcd}
		    X_{(f_1,g_1)}\ar[rrrr]\ar{d}[left]{f_1\times g_1} & & & & X_{(f_2,g_2)}\ar{d}[right]{f_2\times g_2}
		    \\
		    \direalize{\sd_9C_{(f_1,g_1)}}^2\ar{rrrr}[below]{\direalize{\;\sd_9C_{(f_1,g_1)}\ra\sd_9C_{(f_2,g_2)}\;}^2} & & & &
	    \direalize{\sd_9C_{(f_2,g_2)}}^2,
	    \end{tikzcd}
	  \end{equation*}
	  in which the vertical arrows are $\cat{1}$-objects.  
	  The image of each $\mathscr{O}_{(f_1,g_1)}$-object under the top horizontal stream map is a $\mathscr{O}_{(f_2,g_2)}$-object because the bottom horizontal stream map, the directed realization of a cubical function, maps open stars of vertices into open stars of vertices. 
	  It is in this sense that the subcategory $\mathscr{O}_{(f,g)}$ of $\DITOP$ is natural in $\cat{1}$-objects $f\times g$. 

	  Let $U$ denote an $\mathscr{O}_{(f,g)}$-object. 
	  Thus $f_U,g_U$ both corestrict to the directed realization of the same closed star in $\sd_9C_{(f,g)}$ [Proposition \ref{prop:realization.preserves.embeddings}].  
	  Thus there exist cubical function $\theta_U$ and dotted stream maps of the following form, natural in $\mathscr{O}_{(f,g)}$-objects $U$, making the diagram 
	  \begin{equation*}
	    \begin{tikzcd}
		    U\ar[rrr,dotted]\ar[d,hookrightarrow] & & & \direalize{\BOX\boxobj{n_{U}}}^2\ar{d}[right]{\direalize{\;\theta_U\;}^{\otimes 2}}\\
		    X_{(f,g)}\ar{r}[below]{f\times g} & \direalize{\sd_9C_{(f,g)}}^2\ar{rr}[below]{\direalize{\;\epsilon^2_{C_{(f,g)}}\;}^{\otimes 2}} & & \direalize{C_{(f,g)}}^2
		\end{tikzcd}
	  \end{equation*}
	  commute [Lemma \ref{lem:local.lifts}]. 
	  It therefore follows that there exists a $\direalize{\mathfrak{d}_1}$-homotopy $f_U\topologicaldhomotopic{} g_U$, natural in $U$ [Lemma \ref{lem:hypercube.convexity}]. 
	  
	  Thus there exists a 
	  $\direalize{\mathfrak{d}_1}$-homotopy 
	  $$\direalize{\epsilon^2_{C_{(f,g)}}}f\topologicaldhomotopic{}\direalize{\epsilon^2_{C_{(f,g)}}}g$$
	  natural in $\cat{1}$-objects $(f\times g)$ by $\mathscr{O}_{(f,g)}$ natural in $(f\times g)$.
	  There exist $\direalize{\mathfrak{d}}_*$-homotopies 
	  $$\dihomeo^{-2}_{C_{(f,g)};3}\direalize{\epsilon^2_{C_{(f,g)}}}f\topologicaldhomotopic{}f,\quad\dihomeo^{-2}_{C_{(f,g)};3}\direalize{\epsilon^2_{C_{(f,g)}}}g\topologicaldhomotopic{}g$$
	  natural in $\cat{1}$-object $(f\times g)$ [Lemma \ref{lem:natural.approximations}].  
	  Concatenating these homotopies yields the desired $\direalize{\mathfrak{d}}_*$-homotopy.
	\end{proof}

Define the interval object $\mathfrak{h}$ as the composite
\begin{equation*}
  \mathfrak{h}=|\BOX[-]|(\BOX_1\ira\BOX):\BOX_1\ra\DITOP.
\end{equation*}
where topological spaces are regarded as streams equipped with initial circulations [Example \ref{eg:initial.circulations}].   
In other words, an $\mathfrak{h}$-homotopy is exactly a classical homotopy \textit{through} stream maps.  
An $\mathfrak{h}$-homotopy is sometimes referred to in the literature as an \textit{h-homotopy} (e.g. \cite{krishnan2019hurewicz}) or a \textit{dihomotopy} \cite{fajstrup2006algebraic}. 
On one hand, $\mathfrak{h}$-homotopy is an equivalence relation because of natural isomorphisms $\mathfrak{h}_0\cong\mathfrak{h}_1\cong\mathfrak{h}_2\cdots$.  
On the other hand, $\direalize{\mathfrak{d}}$-homotopy is not an equivalence relation on the hom-set $\DITOP(X,Y)$ unless every directed path in $Y$, a path in $Y$ defining a stream map $\vec{\I}\ra Y$, is reversible.  
Every left $\direalize{\mathfrak{d}_n}$-homotopy determines a left $\mathfrak{h}$-homotopy by the existence of a natural transformation $\mathfrak{h}\ra\direalize{\mathfrak{d}}_n$.
It turns out that $f\dhomotopic{\mathfrak{h}}g$ if and only if $f\topologicaldhomotopic{}g$ for a pair of stream maps $f,g$ from a compact stream to a directed realization of a cubical set by a straightforward adaptation of a proof under the usual definition of the cube category \cite{krishnan2015cubical}; however we do not examine $\mathfrak{h}$-homotopy theory in this paper.  

\subsection{Comparisons}\label{sec:homotopical.comparisons}
The different homotopy theories can be compared.
The classical homotopy theories of small categories, cubical sets, simplicial sets, and topological spaces are all equivalent to one another.  
The directed homotopy theories of cubical sets and streams are equivalent to one another, with the directed homotopy theory of small categories acting as a special case.  

\subsubsection{Classical}\label{subsubsec:classical.comparisons}
We can compare different classical homotopy theories.  
Proofs of the following two observations, which use intermediate comparisons to the classical homotopy theory of simplicial sets, are given at the end of Appendix \S\ref{sec:triangulations}.

\begin{prop}
  \label{prop:quillen.equivalence}
  Topological realization defines the left map of a Quillen equivalence
  $$|-|:\CUBICALSETS\lras\TOP$$
  between $\CUBICALSETS$ equipped with the classical model structure and $\TOP$ equipped with its q-model structure.
\end{prop}

A proof is given in \S\ref{sec:triangulations}.  
We write $h(\CUBICALSETS)$ and $h(\TOP)$ for the respective localizations of $\CUBICALSETS$ and $\TOP$ by their classical weak equivalences.  
The above proposition then implies that topological realization passes to an adjoint equivalence
\begin{equation*}
  h(\CUBICALSETS)\simeq h(\TOP).
\end{equation*}

We write $h(\CATS)$ for the localization of $\CATS$ by its Thomason weak equivalences.  

\begin{cor}
  \label{cor:cubical.coherent.nerve}
  The functor $\cnerve$ induces a categorical equivalence
  $$h(\CATS)\simeq h(\CUBICALSETS).$$
\end{cor}

\begin{proof}[proof of Proposition \ref{prop:equivalent.fundamental.groupoids}]
   Consider a cubical function $\psi:A\ra B$. 
   The diagram
	\begin{equation*}
    \begin{tikzcd}
		\Pi_1A\ar{r}[above]{\Pi_1\psi}\ar[d,hookrightarrow] & \Pi_1B\ar[d,hookrightarrow]\\
		\Pi_1|A|\ar{r}[below]{\Pi_1|\psi|}\ar{r}[below]{\Pi_1|\psi|} & \Pi_1|B|
    \end{tikzcd}
  \end{equation*}
   in which $\Pi_1$ in the bottom row denotes the fundamental groupoid of a topological space and the vertical arrows are inclusions of fundamental groupoids induced by topological realization, commutes.
   The vertical arrows are categorical equivalences by an application of cellular approximation.
   Thus if $\psi$ is a classical weak equivalence, $|\psi|$ is a classical homotopy equivalence [Proposition \ref{prop:quillen.equivalence}], hence $\Pi_1|\psi|:\Pi_1|A|\ra\Pi_1|B|$ is a categorical equivalence, and hence $\Pi_1\psi:\Pi_1A\ra\Pi_1B$ is a categorical equivalence.
\end{proof}

\subsubsection{Directed}\label{sec:directed.homotopical.comparisons}
We refine and extend the previous equivalences between classical homotopy categories.  
Proofs of our comparisons between \textit{directed homotopy categories}  [Theorems \ref{thm:equivalence} and \ref{thm:1-ditypes}] require a generalization of fibrant cubical sets to suitable directed analogues in \S\ref{sec:cubcats}.  
However, we can formally state these comparisons and some simple applications without proof in this section.  
In fact, our comparisons can be stated at the more general level of $\indexcat{1}$-equivariance.  
Recall that a class $\mathscr{W}$ of morphisms in a category $\cat{1}$ for which the localization  $\cat{1}[\mathscr{W}^{-1}]$ exists is \textit{saturated} if it coincides with the isomorphisms in the localization $\cat{1}[\mathscr{W}^{-1}]$ of $\cat{1}$ by $\mathscr{W}$.  
Define the \textit{directed homotopy categories} $d(\CUBICALSETS)$ and $d(\DITOP)$ of cubical sets and directed topological spaces by the following theorem, whose proof is given at the end of \S\ref{sec:approximation}.

\begin{thm}
  \label{thm:equivalence}
  There exist dotted vertical localizations in the diagram
  \begin{equation*}
	  \begin{tikzcd}
		  \CUBICALSETS^{\indexcat{1}}\ar{r}[above]{\direalize{\;-\;}}\ar[d,dotted] & \DITOP^{\indexcat{1}}\ar[d,dotted]\\
		  d\left(\CUBICALSETS^{\indexcat{1}}\right)\ar[dotted]{r}[below]{\simeq} & d\left(\DITOP^{\indexcat{1}}\right)
	  \end{tikzcd}
  \end{equation*}
  by the following respective saturated classes of morphisms: those $\indexcat{1}$-cubical functions $\psi$ for which $\direalize{\psi}^{\indexcat{1}}$ is a $\direalize{\mathfrak{d}}$-equivalence; and those $\indexcat{1}$-stream maps $f$ for which $\csing^{\indexcat{1}}\,f$ is a $\mathfrak{d}$-equivalence.
  There exists a dotted horizontal adjoint equivalence making the entire diagram commute up to natural isomorphism.
\end{thm}

Define the directed homotopy category $d(\CATS^{\indexcat{1}})$ of $\indexcat{1}$-categories as a small part of the larger directed homotopy category $d(\CUBICALSETS^{\indexcat{1}})\simeq d(\DITOP^{\indexcat{1}})$ by the following theorem,  whose proof is given at the end of \S\ref{sec:approximation}. 

\begin{thm}
  \label{thm:1-ditypes}
  Consider the solid functors in the diagram
  \begin{equation*}
	  \begin{tikzcd}
		  \CATS^{\indexcat{1}}\ar{r}[above]{\cnerve}\ar[d] & \CUBICALSETS^{\indexcat{1}}\ar[d]\\
		  d\left(\CATS^{\indexcat{1}}\right)\ar[r,hookrightarrow,dotted] & d\left(\CUBICALSETS^{\indexcat{1}}\right)	 
	  \end{tikzcd}
  \end{equation*}
  where the left vertical arrow is the quotient functor by $\dhomotopic{\Tau_1\mathfrak{d}}$ and the right vertical arrow is localization by the $\mathfrak{d}$-equivalences.  
  The left vertical arrow is localization by the $\Tau_1\mathfrak{d}$-equivalences.
  There exists a dotted fully faithful embedding making the entire diagram commute up to natural isomorphism.  
\end{thm}

\begin{eg}
  \label{eg:ditypes.are.not.1.ditypes}
  For each $n>1$ and small category $\smallcat{1}$, every stream map 
  $$\direalize{{\BOX\boxobj{n}}/{\partial\BOX\boxobj{n}}}\ra\direalize{\cnerve\smallcat{1}}$$
  is $\topologicaldhomotopic{}$-homotopic to a constant stream map by an application of Theorems \ref{thm:equivalence} and \ref{thm:1-ditypes} and $\Tau_1(\BOX\boxobj{n}/{\partial\BOX\boxobj{n}})=\star$.   
  It therefore follows that higher directed spheres $\direalize{{\BOX\boxobj{n}}/{\partial\BOX\boxobj{n}}}$ do not have the h-homotopy, much less d-homotopy type, type of directed realizations of cubical nerves of small categories.
  Intuitively, the cubical model $\BOX\boxobj{n}/\partial\BOX\boxobj{n}$ of a directed sphere presents a cubcat freely generated by a single $n$-cell between a single vertex.  
  Thus directed homotopy types encode higher category-like structures, albeit up to directed homotopy, more general than small $1$-categories (c.f. \cite{dubut2016directed}).
\end{eg}

One consequence is the following calculation.

\begin{prop}
  \label{prop:homotopy.calculations}
  For each monoid $M$, $\tau_n\left(\direalize{\cnerve M},\direalize{\star}\right)=\begin{cases}M & n=1\\\star & n\neq 1\end{cases}$.
\end{prop}
\begin{proof}
  Let $\mathbb{S}(n)=\BOX\boxobj{n}/\partial\BOX\boxobj{n}$.  
  We can make natural identifications
  \begin{equation*}
  \tau_n(\direalize{\cnerve M},\direalize{\star})= \tau_n(\catfont{S}(\cnerve M),\direalize{\star})= [\Tau_1\mathbb{S}(n),M]_{\Tau_1\mathfrak{d}}=\CATS(\Tau_1\mathbb{S}(n),M).
  \end{equation*}
  the second by the comparison theorems [Theorems \ref{thm:equivalence} and \ref{thm:1-ditypes}] and $\Tau_1\dashv\nerve$, the third by the fact that based $\Tau_1\mathfrak{d}_*$-homotopy between monoid homomorphisms is equality [Lemma \ref{lem:based.d-homotopic.monoid.homomorphisms}].  
  Then note that $\Tau_1\mathbb{S}(n)$ is the terminal monoid $\star$ if $n\neq 1$ and $\N$ if $n=1$.  
  The result then follows.  
\end{proof}

Write $\cohomology{1}_{\sing}(X;\tau)$ for the commutative monoid
$$\cohomology{1}_{\sing}(X;\tau)=\cohomology{1}(\sing\,X;\tau)=[\Tau_1\sing\,X,\tau]_{\Tau_1\mathfrak{d}}$$  
natural in streams $X$ and commutative monoids $\tau$.  
Another consequence is the equivalence between cubical directed $1$-cohomology and singular directed $1$-cohomology.

\begin{prop}
  \label{prop:cohomological.equivalence}
  There exists an isomorphism 
  $$\cohomology{1}(\eta_C:C\ra\catfont{S}C;\tau):H^1_{\csing}(\direalize{C};\tau)\cong H^1(C;\tau)$$
  of commutative monoids, natural in cubical sets $C$ and commutative monoids $\tau$, where $\eta$ is the unit of $\catfont{S}$.  
\end{prop}
\begin{proof}
  Let $\catfont{D}=d(\CATS)$.  
  We can make the natural identifications
  $$\cohomology{1}(C;\tau)=[C,\nerve\,\tau]_{\mathfrak{d}}=\catfont{D}(\Tau_1C,\tau),$$
	the second by $\Tau_1\dashv\nerve$ and the third by the fully faithful embedding $\catfont{D}\ira d(\CUBICALSETS)$ [Theorem \ref{thm:1-ditypes}].  
  
  In the commutative diagram
  \begin{equation*}
	  \begin{tikzcd}
		  {\cohomology{1}_{\sing}(\direalize{C};\tau)}\ar{rr}[above]{\cohomology{1}(\eta_C;\tau)}
		  \ar{d}[left]{\cong} & & {\cohomology{1}(C;\tau)}\ar{d}[right]{\cong}\\
		  \catfont{D}(\Tau_1\catfont{S}C,\tau)
		  \ar{rr}[below]{\catfont{D}(\Tau_1\eta_C,\tau)} 
		  & & {\catfont{D}(\Tau_1C,\tau)}
	  \end{tikzcd}
  \end{equation*}
  in which the vertical arrows are induced by $\direalize{-}$, the vertical arrows are bijections [Theorem \ref{thm:1-ditypes}], the bottom horizontal arrow is a bijection by $\Tau_1\eta_{C}$ a $\catfont{D}$-isomorphism [Theorems \ref{thm:equivalence} and \ref{thm:1-ditypes}], and hence the top horizontal arrow is also a bijection.  
\end{proof}

\begin{prop}
  \label{prop:cohomological.calculation}
  Consider the following data.
  \begin{enumerate}
    \item cancellative commutative monoid $\tau$
    \item cubical set $C$ having a unique vertex
  \end{enumerate}
  Then $\cohomology{1}_{\sing}(\direalize{C};\tau)=\CATS(\Tau_1C,\tau)$.
\end{prop}
\begin{proof}
  There exist natural identifications
  \begin{align*}
    \cohomology{1}_{\sing}(\direalize{C};\tau)=[\Tau_1C,\tau]_{\Tau_1\mathfrak{d}}=[\N,\tau^{\Tau_1C}]_{\Tau_1\mathfrak{d}}=\CATS(\N,\tau^{\Tau_1C})=\CATS(\Tau_1C,\tau),
  \end{align*}
  the first by the comparison theorems [Theorems \ref{thm:equivalence} and \ref{thm:1-ditypes}], the second and fourth by $C$ having a unique vertex and hence $\Tau_1C$ a monoid, and the third by $\tau$ and hence $\tau^{\Tau_1C}$ a cancellative monoid [Lemma \ref{lem:conjugacy.classes}].
\end{proof}

Some easy calculations follow for the unique two closed and connected $(1+1)$-spacetimes that arise as directed realizations of (finite) cubical sets.  

\begin{figure}
		\includegraphics[width=4in,height=1.5in]{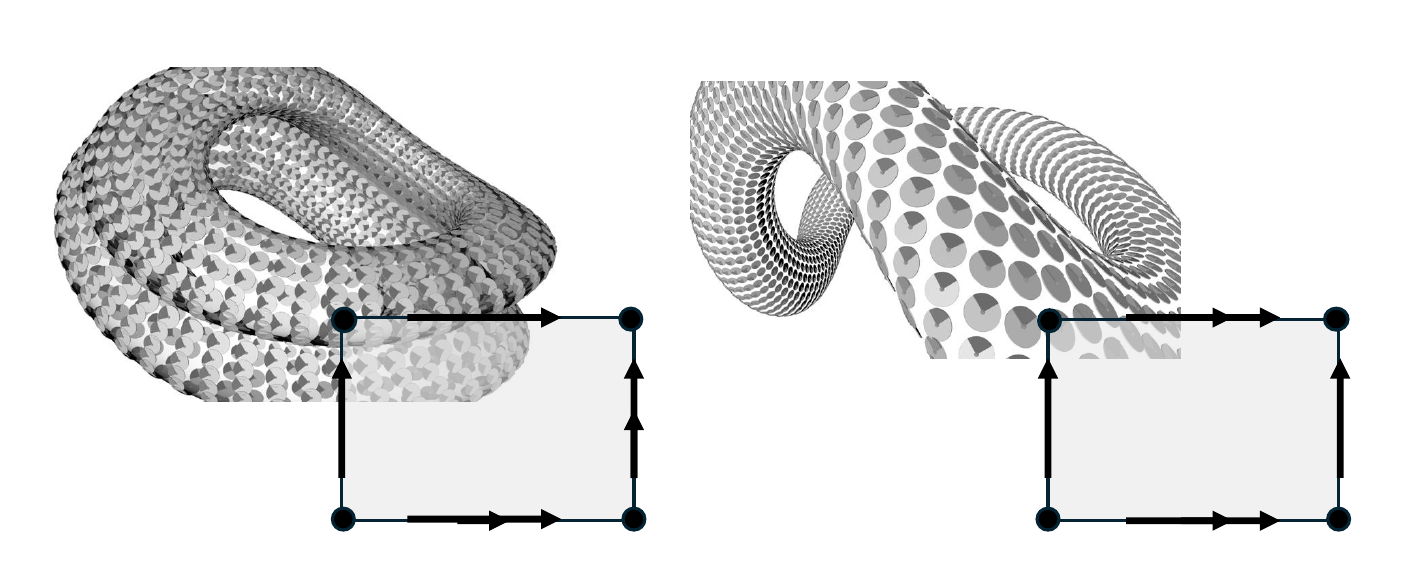}
	  \caption{
	    {\bf Conal manifolds}
	    \textit{Conal manifolds}, smooth manifolds whose tangent spaces are all equipped with convex cones, naturally encode state spaces of processes under some causal constraints.
	    The convex cones define partial orders on an open basis of charts that uniquely extend to circulations on the entire manifold.
	    The time-oriented Klein bottle (left) and time-oriented torus (right) depicted above are examples of conal manifolds that arise as directed realizations of cubical sets. 
	    Over cancellative commutative monoid coefficients $\tau$, their directed $1$-cohomologies are $\tau\times_{2\tau}\tau$ (left) and $\tau^2$ (right) by a simple application of cubical approximation [Examples \ref{eg:klein.calculation} and \ref{eg:torus.calculation}].  
	    }
	  \label{fig:compact.timelike.surfaces}
\end{figure}

\begin{eg}
  \label{eg:torus.calculation}
  Fix a cancellative commutative monoid $\tau$.
  Then 
  $$\cohomology{1}_{\sing}(T;\tau)=\tau^2,$$
  where $T$ is the unique underlying stream of a time-oriented Lorentzian torus [Figure \ref{fig:compact.timelike.surfaces}], by identifying $T$ with the directed realization of $(\BOX[1]/\partial\BOX[1])^{\otimes 2}$, calculating the fundamental category of that cubical set to be $\N^2$, and applying Proposition \ref{prop:cohomological.calculation}.  
\end{eg}  

\begin{eg}
  \label{eg:klein.calculation}
  Fix a cancellative commutative monoid $\tau$. 
  Then
  $$\cohomology{1}_{\sing}(K;\tau)=\tau\times_{2\tau}\tau,$$
  where $K$ is the unique underlying stream of a time-oriented Lorentzian Klein bottle [Figure \ref{fig:compact.timelike.surfaces}], by identifying $K$ with the directed realization of the quotient $C$ of $\BOX\boxobj{2}$ by the smallest equivalence relation identifying $\delta_{\pm 1;2}$ with $\delta_{\mp 2;2}$, calculating $\Tau_1C$ to be the monoid $\N*_{2\N}\N$, and applying Proposition \ref{prop:cohomological.calculation}. 
\end{eg}  

We summarize many of the above observations above in the following commutative diagram.
Write $h(\FIBRANTCUBICALSETS)$ for the full subcategory of $h(\CUBICALSETS)$ consisting of its fibrant objects, the category of fibrant cubical sets and cubical homotopy classes of cubical functions between them.  
Write $h(\GROUPOIDS)$ for the full subcategory of $h(\CATS)$ consisting of all small groupoids, the category of small groupoids and natural isomorphism classes of functors between them.
There exists a commutative diagram
  \begin{equation*}
    \begin{tikzcd}
		d(\CATS)\ar[rr,hookrightarrow]\ar[dd,twoheadrightarrow] & & d(\CUBICALSETS)\ar{rr}[description]{\simeq}\ar[dd,twoheadrightarrow] & & d(\DITOP)\ar[dd,twoheadrightarrow]\\
		& h(\GROUPOIDS)\ar[ul,hookrightarrow]\ar[dr,hookrightarrow]\ar[ur,hookrightarrow]\ar[rr,hookrightarrow]\ar[dl,hookrightarrow]\ar[ul,hookrightarrow] & & h(\FIBRANTCUBICALSETS)\ar{dl}[description]{\simeq}\ar[ul,hookrightarrow]\ar[ur,hookrightarrow]\ar{dr}[description]{\simeq}\\
		h(\CATS)\ar{rr}[description]{\simeq} & & h(\CUBICALSETS)\ar{rr}[description]{\simeq} & & h(\TOP)
    \end{tikzcd}
  \end{equation*}
in which the vertical arrows are induced from forgetful functors, the leftmost horizontal arrows in each row are induced from the cubical nerve, the rightmost horizontal arrows in the top and bottom rows are induced from realization functors, and the diagonal arrows pointing towards the left are induced from inclusions.  
Functors denoted as $\hookrightarrow$ are fully faithful. 
Functors denoted as $\twoheadrightarrow$ are essentially surjective.  
Functors are categorical equivalences if and only if they are labelled with $\simeq$.
The commutativity of the outer rectangle formalizes how our directed homotopy theory refines classical homotopy theory.  
The commutativity of the upper triangles formalizes how our directed homotopy theory extends classical homotopy theory.  

\section{Cubcats}\label{sec:cubcats}
We develop a theory of \textit{cubcats}.  
We first give a definition in terms of compatible unary and composition operations [Definition \ref{defn:cubcats}] and later give an equivalent characterization as an algebra structure over a directed analogue $\cubcatreplacement{}$ of fibrant replacement [Proposition \ref{prop:free.cubcat.monad}].  
We give some basic examples [Propositions \ref{prop:nerves}, \ref{prop:singular.cubcats}, and \ref{prop:fibrant.cubcats}].  
We then give a cubical approximation theorem for directed homotopy [Theorem \ref{thm:derived.formula}], in which cubcats play a role analogous to the role that fibrant cubical sets play in a cubical approximation theorem for classical homotopy.  
Afterwards, we list a number of consequences [Corollaries \ref{cor:formula} and \ref{cor:directed.milnor}] and proofs of comparison results from the previous section.  

\subsection{Definition}
Take $(-)^{\sharp}$ to be the dotted left Kan extension in
  \begin{equation*}
    \begin{tikzcd}
	    \BOX\ar{d}[left]{\BOX[-]}\ar{rr}[above]{\catfont{S}\BOX[-]} & & \CUBICALSETS\\
	    \CUBICALSETS\ar[dotted]{urr}[description]{(-)^{\sharp}}
    \end{tikzcd}
  \end{equation*}
of the top horizontal arrow along the left vertical arrow. 
The monad structure on the monad $\catfont{S}$ of the adjunction $\direalize{-}\dashv\sing$ induce a monad structure on $(-)^{\sharp}$.
The Yoneda embedding $\BOX[-]:\BOX\ra\CUBICALSETS$ induces a map of monads 
$$(-)^\sharp\ira\catfont{S},$$
component-wise monic and thus henceforth regarded as a component-wise inclusion.

\begin{eg}
  We can make the natural identification $(\BOX\boxobj{n})^{\sharp}=\catfont{S}\BOX\boxobj{n}$.  
\end{eg}



For convenience, take $\dihomeo^\sharp_{C;k+1}$ to be the natural cubical function
$$\dihomeo^\sharp_{C;k+1}:C^{\sharp}\ra\ex_{k+1}C^\sharp$$
defined for the case $C=\BOX\boxobj{n}$ by the commutative square
\begin{equation*}
	\begin{tikzcd}
		(\BOX\boxobj{n})^\sharp\ar[d,equals]\ar{rrrr}[above]{\dihomeo^\sharp_{\BOX\boxobj{n};k+1}} & & & & \ex_{k+1}((\BOX\boxobj{n})^\sharp)\ar[d,equals]
  \\
		\catfont{S}\BOX\boxobj{n}\ar{r}[below]{\catfont{S}\eta_{\BOX\boxobj{n}}} & \ex_{k+1}\catfont{S}(\sd_{k+1}\BOX\boxobj{n})\ar{rrr}[below]{\ex_{k+1}\sing\,\dihomeo_{\BOX\boxobj{n};k+1}} & & & \ex_{k+1}\catfont{S}(\BOX\boxobj{n})
	\end{tikzcd}
\end{equation*}

\begin{defn}
  \label{defn:cubcats}
  A \textit{$\indexcat{1}$-cubcat} is a $\indexcat{1}$-cubical set $C$ if in the diagrams
       \begin{equation}
	       \label{eqn:cubcats}
                    \begin{tikzcd}
			    C\ar{d}[left]{\eta_C}\ar{r}[above]{\id_C}\ar[dr,phantom,very near start, "I"] & C\\
			    C^\sharp\ar[dotted]{ur}[description]{\rho} & \;
		    \end{tikzcd}\quad
\begin{tikzcd} 
	C^\sharp\ar{rrrr}[above]{\dihomeo^\sharp_{C;2}} 
	\ar{d}[left]{\rho}
	& & \;\ar[d,phantom,near start,"II"] & & \ex_2C^\sharp\ar{d}[description]{\ex_2\rho}\\
	C & & \; & & \ex_2C,\ar[dotted]{llll}[below]{\mu}
\end{tikzcd}
		  \end{equation}
		  in which $\eta$ denotes the unit of $(-)^{\sharp}$, 
	there exist dotted $\indexcat{1}$-cubical functions $\rho$ and $\mu$ making I and II commute.
\end{defn}

A \textit{cubcat} will simply refer to a $\indexcat{1}$-cubcat for the case $\indexcat{1}=\star$.  
The retraction $\rho:C^\sharp\ra C$ in I in effect extends the face, degeneracy, and transposition operations on a cubical set to a larger collection of unary operations parametrized by monotone maps between directed topological cubes in the following sense.
Let $\zeta$ denote a stream map
$$\zeta:\direalize{\BOX\boxobj{n}}\ra\direalize{\BOX\boxobj{m}}.$$  
Each such $\zeta$ defines a natural transformation $\zeta_*:(C^\sharp)_m\ra(C^\sharp)_n$ defined for $C$ representable as precomposition $(\catfont{S}\BOX[-])_m\ra(\catfont{S}\BOX[-])_n$ with $\zeta$ and hence for general cubical sets $C$ by cocontinuity of $(-)^\sharp$.  
For $\rho$ making I commute in (\ref{eqn:cubcats}), commutative diagrams
\begin{equation*}
  \begin{tikzcd}
	  C_m\ar[rr,dotted]\ar[d,hookrightarrow] & & C_n\\
	  (C^\sharp)_m\ar{rr}[below]{\zeta_*} & & (C^\sharp)_n,\ar{u}[right]{\mu_n}
  \end{tikzcd}
\end{equation*}
natural in $\zeta$, define dotted functions that, in the case $\zeta$ is of the form $\direalize{\BOX[\phi]}$ for a $\BOX$-morphism $\phi$, coincides with the unary operation $C(\phi):C_m\ra C_n$.  
A cubical function $\ex_2C\ra C$ in II composes a composable configuration of cubes, shaped like a subdivided cube, into a single cube.  
The commutativity of II asserts that the composite of a subdivision of cubes yields the original cube.  
It is possible if more tedious to reformulate composition in the definition as a family of category-like compositions
\begin{equation}
  \label{eqn:compositions}
  C_n\times_{C(\delta_{+i;n}),C(\delta_{-i;n})}C_n\ra C_n,\quad 1\leqslant i\leqslant n
\end{equation}
compatible with the face, degeneracy, and symmetry operations on a cubical set $C$ and additionally satisfying interchange laws.  

\begin{eg}
  Take a \textit{framed cubical set} to mean a presheaf
  $$\OP{(\BOX_1^*)}\ra\SETS$$
	over the minimal variant $\BOX_1^*$ of the cube category, the subcategory of $\SETS$ generated by $\delta_{\mins},\delta_{\pls}:[0]\ra[1]$ and $\sigma:[1]\ra[0]$.  
  A globular $\omega$-category is equivalent to the data of a framed cubical set $C$ together with an extension like the dotted functor in diagram I of (\ref{eqn:cubcats}), extra composition operations on $C$ like (\ref{eqn:compositions}), and a compatibility condition like the commutativity of the right diagram in \ref{eqn:cubcats}, except that $\DITOPOLOGICALBOX$ is replaced with the minimal variant of the cube category containing \textit{coconnections} of one kind and the compositions are additionally required to be unital and associative in an appropriate sense \cite{agl2002multiple}.   
  Unlike strict globular $\omega$-categories, the composition operations witnessing the property of being a cubcat are not part of the structure of a cubcat.
\end{eg}

We can characterize cubcats in terms of an endofunctor defined as follows.  

\begin{defn}
  Define $\cubcatreplacement\,C$ to be the $\CUBICALSETS$-colimit
	$$\cubcatreplacement\,C=\colim\left(C\ira C^\sharp\xra{\dihomeo^\sharp_{C;2}}\ex_2C^\sharp\xra{\ex_2\dihomeo^\sharp_{C;2}}\ex_4C^\sharp\xra{\ex_4\dihomeo^\sharp_{C;2}}\cdots\right)$$
  natural in cubical sets $C$.  
\end{defn}

Let $C$ denote a cubical set.  
Consider the following commutative diagram
\begin{equation}
    \label{eqn:free.cubcat.monad}
    \begin{tikzcd}
	    C^\sharp\ar[d,hookrightarrow]\ar{rr}[above]{\dihomeo^\sharp_{C;2}\ar[d,hookrightarrow]} & & \ex_2C^\sharp\ar[dll]\ar{rr}[above]{\ex_2\dihomeo^\sharp_{C;2}} & & \ex_4C^\sharp\ar{rr}[above]{\ex_4\dihomeo^\sharp_{C;2}}\ar[dllll] & & \cdots\\
	    \catfont{S}C,
    \end{tikzcd}
\end{equation}
natural in $C$, in which each diagonal arrow $\ex_{2^i}C^\sharp\ra\catfont{S}C$ is the composite of the mono $\ex_{2^i}(C^\sharp\ira\catfont{S}C):\ex_{2^i}C^\sharp\ra\ex_{2^i}\catfont{S}C$ with the mono $\ex_{2^i}\catfont{S}C\ra\catfont{S}C$ sending each $n$-cube $\theta:\direalize{\sd_{2^i}\boxobj{n}}\ra\direalize{C}$ in $\ex_{2^i}\catfont{S}C$ to the $n$-cube $\theta\dihomeo_{\BOX\boxobj{n};2^i}:\direalize{\BOX\boxobj{n}}\ra\direalize{C}$ in $\catfont{S}C$.  
The above component-wise monic cocone from the top row to the cubical set $\catfont{S}C$ induces a mono $\cubcatreplacement{}C\ra\catfont{S}C$ which we henceforth regard as an inclusion of subpresheaves.  
Concretely, $(\cubcatreplacement\,C)_n$ is the set of stream maps $f:\direalize{\BOX\boxobj{n}}\ra\direalize{C}$ such that for $k\gg 1$ and each $(\BOX/\sd_{2}^k\BOX\boxobj{n})$-object $\gamma$, there exists a $(\BOX/C)$-object $\theta(\gamma,f)$ such that $f\dihomeo_{C;2^k}\direalize{\gamma}$ lifts along $\direalize{\theta(f,\gamma)}$.  

\begin{prop}
  \label{prop:free.cubcat.monad}
  The following are equivalent for a $\indexcat{1}$-cubical set $C$:
	\begin{enumerate}
          \item  $C$ is a $\indexcat{1}$-cubcat 
	  \item the corestriction $C\ra\cubcatreplacement\,C$ of the cubical function $C\ra\catfont{S}C$ defined by the unit of $\catfont{S}$ admits a retraction
	\end{enumerate}
\end{prop}

In particular, we henceforth regard $\cubcatreplacement\,C$ as a subpresheaf of $\catfont{S}C$.
The monad multiplication on $\catfont{S}$ does not restrict and corestrict to a monad multiplication on $\cubcatreplacement{}$.  
But the unit of $\catfont{S}$ defines natural cubical functions $C\ra C^\sharp$ for the case $C$ representable and hence uniquely corestricts to a natural transformation $\id_{\CUBICALSETS}\ra\cubcatreplacement{}$.

\begin{proof}
  Let $\eta^{\catfont{S}}$ and $\mu^{\catfont{S}}$ be the unit and multiplication of $\catfont{S}$.
  Let $\eta^\sharp$ and $\mu^\sharp$ be the unit and multiplication of $(-)^\sharp$.  
  Let $\epsilon^{(j)}$ be the counit of $\sd_2^j\dashv\ex_{2}^j$. 
  Let $\ex=\ex_2$.   
  Let $C$ denote a cubical set.  
  
  Let $\iota$ denote the component-wise inclusion $\cubcatreplacement\ra\catfont{S}$.
  Let $\eta_C$ denote the corestriction $C\ra\cubcatreplacement{}C$ described in the proposition.  

  Consider the case $\indexcat{1}=\star$.
  Consider the solid diagram
  \begin{equation}
	  \label{eqn:free.cubcat.monad.multiplication}
    \begin{tikzcd}
	    C\ar{r}[above]{\eta^\sharp_C} 
	    & C^\sharp\ar[dl,dotted]\ar{rr}[above]{\dihomeo^{\sharp}_{C;2}}
	    & & \ex_2C^\sharp\ar[dlll,dotted]\ar{rr}[above]{\ex_2\dihomeo^\sharp_{C;2}} 
	    & & \ex_4C^\sharp\ar{rr}[above]{\ex_4\dihomeo^\sharp_{C;2}}\ar[dlllll,dotted] 
	    & & \cdots\\
    C\ar{u}[description]{\id_C}
    \end{tikzcd}
  \end{equation}

  Suppose (1). 
  Then there exist cubical functions $\rho$ and $\mu$ making (\ref{eqn:cubcats}) commute.  
  We can define a commutative diagram (\ref{eqn:free.cubcat.monad.multiplication}), natural in $\rho$ and $\mu$, so that the $i$th unlabelled arrow counting from the left is $\mu^{i}_{C}(\ex^i\rho):\ex^iC^\sharp\ra C$.
  The cocone (\ref{eqn:free.cubcat.monad.multiplication}) induces a retraction to $\eta_C$.
  Thus (2).   

  Suppose (2).  
  Then there exists a retraction $\mu$ to $\eta_C$.
  Then there exist unique dotted cubical functions, natural in $\mu$, making (\ref{eqn:free.cubcat.monad.multiplication}) commute.
  Let $\rho$ be the first dotted vertical cubical function from the left in (\ref{eqn:free.cubcat.monad.multiplication}).  
  Let $\mu$ be the composite of $\ex\,\eta^\sharp_C$ with the second dotted vertical cubical function in (\ref{eqn:free.cubcat.monad.multiplication}).  
  Then (\ref{eqn:cubcats}) commutes because (\ref{eqn:free.cubcat.monad.multiplication}) commutes.
  Thus (1).

  The naturality of the constructions implies the general case.
\end{proof}

The class of unary operations required on a cubcat can likely be reduced in the sense that the construction $(-)^\sharp$ is probably larger than necessary for the main results.  
What is not clear is whether a smaller alternative to $(-)^{\sharp}$ is as simple to define.
A simpler, smaller alternative to $(-)^{\sharp}$ would be useful for formalizing cubical directed homotopy types on a computer (c.f. \cite{awodey2009homotopy,licata2011directed,riehl2017type}).  

\subsection{Examples}
Cubcats at once generalize 1-categories [Proposition \ref{prop:nerves}], directed singular cubical sets [Proposition \ref{prop:singular.cubcats}], and higher groupoids [Proposition \ref{prop:fibrant.cubcats}].

\begin{prop}
  \label{prop:nerves}
  For each $\indexcat{1}$-category $\smallcat{1}$, $\nerve\,\smallcat{1}$ is a $\indexcat{1}$-cubcat. 
\end{prop}
\begin{proof}
  Let $\eta$ denote the unit of $(-)^{\sharp}$.  
  Let $\smallcat{1}$ denote a small category.

  It suffices to take $\indexcat{1}=\star$ and construct dotted cubical functions, natural in $\smallcat{1}$, making (\ref{eqn:cubcats}) commute for $C=\nerve\,\smallcat{1}$.  
	There exists a unique functor $\rho_{\boxobj{n}}:\catfont{S}(\nerve\boxobj{n})\ra\nerve\boxobj{n}$, natural in $\BOX$-objects $\boxobj{n}$, whose adjoint $\Tau_1\catfont{S}(\nerve\boxobj{n})\ra\boxobj{n}$ sends each object $x$ to $\min\,\support_{|\nerve-|:\BOX\ra\TOP}(x,\boxobj{n})$ by an application of Lemma \ref{lem:orientations}.  
  These functors define the components of a unique natural transformation
	$$\rho:\catfont{S}(\nerve(\BOX\ira\CATS))\ra\nerve(\BOX\ira\CATS):\BOX\ra\CUBICALSETS.$$
  
	The functor $\rho_{[0]}:\catfont{S}(\nerve[0])\ra\nerve[0]$ is an isomorphism $\star\cong\star$ of terminal cubical sets. 
	Therefore $\rho$ defines a retraction to the natural transformation $\eta_{\nerve(\BOX\ira\CATS)}:\nerve\ra\catfont{S}\,\nerve$ by naturality because each component of $\rho$, a functor to a poset, is uniquely determined by its values on objects.
  Thus the $\rho_{\boxobj{n}}$'s induce a cubical function $\rho_{\smallcat{1}}$ making I in (\ref{eqn:cubcats}) commute when $C=\nerve\,\smallcat{1}$ and $\rho=\rho_{\smallcat{1}}$.

  Define a cubical function $\mu_{\smallcat{1}}:\ex_2\nerve\smallcat{1}\ra\nerve\smallcat{1}$, natural in $\smallcat{1}$, by the rule
  $$(\mu_{\smallcat{1}})_n=\CATS((\boxobj{n})^{[1]\ra[0]},\smallcat{1}):\CATS(\multiboxobj{2}{n},\smallcat{1})\ra\CATS(\boxobj{n},\smallcat{1})$$
  natural in $\BOX$-objects $\boxobj{n}$.  
  Consider rectangle II in (\ref{eqn:cubcats}) for the case $C=\nerve\,\smallcat{1}$, $\rho=\rho_{\smallcat{1}}$, and $\mu=\mu_{\smallcat{1}}$.
  In the case $\smallcat{1}=[0]$, the diagram is a diagram of terminal cubical sets and therefore commutes.  
  In the case $\smallcat{1}$ is a $\BOX$-object, the diagram commutes by naturality because both maximal composites $(\nerve\boxobj{n})^\sharp\ra\nerve\,\boxobj{n}$ of arrows in the diagram have as their adjoints functors $\Tau_1(\nerve\boxobj{n})^\sharp\ra\boxobj{n}$ to posets and are hence determined by their values on objects.  
  It therefore follows that the diagram commutes for general small categories $\smallcat{1}$ by naturality.
\end{proof}

Just as singular cubical sets of topological spaces are fibrant, directed singular cubical sets of streams are cubcats.  

\begin{prop}
  \label{prop:singular.cubcats}
  For each $\indexcat{1}$-stream $X$, $\sing\,X$ is a $\indexcat{1}$-cubcat.  
\end{prop}
\begin{proof}
  Let $\eta$ denote the unit of $\direalize{-}\dashv\sing$.  
  The $\indexcat{1}$-cubical function $\eta_{\sing\,X}$ admits a retraction by the zig-zag identities, hence its corestriction to $\cubcatreplacement\,\sing\,X$ admits a retraction, and hence $\sing\,X$ is a $\indexcat{1}$-cubcat [Proposition \ref{prop:free.cubcat.monad}].
\end{proof}

Fibrant cubical sets themselves are cubcats.

\begin{prop}
  \label{prop:fibrant.cubcats}
  Every fibrant cubical set is a cubcat.
\end{prop}

A proof, which requires cubical approximation, is deferred until the end of \S\ref{sec:approximation}.  

\subsection{Approximation}\label{sec:approximation}
We first compare $\cubcatreplacement$ with $\catfont{S}$.  

\begin{lem}
  \label{lem:cubcats.are.sing.algebras}
  There exists a dotted cubical function in the diagram
  \begin{equation*}
    \begin{tikzcd}
	    \catfont{S}(\sd_9C)\ar[r,dotted] & \cubcatreplacement\,C\\
	    \cubcatreplacement\,\sd_9C\ar[u,hookrightarrow]\ar{ur}[description]{\cubcatreplacement\,\epsilon^2_{C}},
    \end{tikzcd}
  \end{equation*}
  natural in cubical sets $C$, making the entire diagram commute.
\end{lem}
\begin{proof}
  Let $C$ denote a cubical set.  
  Consider the outer naturality square of
  \begin{equation}
    \label{eqn:cubcats.are.sing.algebras}
    \begin{tikzcd}
	    \catfont{S}(\sd_9C)\ar{rr}[above]{\catfont{S}\epsilon^2}\ar[drr,dotted] & & \catfont{S}C\\
	    \cubcatreplacement\,\sd_9C\ar[u,hookrightarrow]\ar{rr}[below]{\cubcatreplacement\,\epsilon^2_{C}} & & \cubcatreplacement\,C\ar[u,hookrightarrow].
    \end{tikzcd}
  \end{equation}
  
	Consider the data of a cubical function $\BOX\boxobj{n}\ra\catfont{S}(\sd_9C)$, equivalently given by its adjoint stream map $f:\direalize{\BOX\boxobj{n}}\ra\direalize{\sd_9C}$.    
  For $k\gg 0$, $f\dihomeo_{\BOX\boxobj{n};k+1}$ maps each closed cell into the open star of a vertex in $|\sd_9C|$ by $|\BOX\boxobj{n}|$ compact and hence $f$ uniformly continuous with respect to the $\ell_\infty$-metric on $|\BOX\boxobj{n}|=\I^n$.  
  Thus for each cubical function $\theta:\BOX\boxobj{m}\ra\sd^k\BOX\boxobj{n}$, $\direalize{\epsilon^2_C}f\dihomeo_{\BOX\boxobj{n};k+1}\direalize{\theta}$ lifts along a stream map of the form $\direalize{\BOX\boxobj{n(\theta,f)}\ra C}$ for some natural number $n(\theta,f)$ [Lemma \ref{lem:local.lifts}].  
  Thus the adjoint of $f$ defines a cubical function $\BOX\boxobj{n}\ra\cubcatreplacement\,C$ [Proposition \ref{prop:free.cubcat.monad}].  

  It therefore follows that the top horizontal arrow in (\ref{eqn:cubcats.are.sing.algebras}) corestricts to a diagonal dotted cubical function, unique by the right vertical arrow in (\ref{eqn:cubcats.are.sing.algebras}) monic and hence natural in $C$, making all of (\ref{eqn:cubcats.are.sing.algebras}) commute.
\end{proof}

\begin{lem}
  \label{lem:cubcats.are.homotopy.sing.algebras}
	There exist cubical function $\mu_C:\catfont{S}C\ra\cubcatreplacement\,C$ and $\mathfrak{d}_*$-homotopies
	$$(\cubcatreplacement\,C\ira\catfont{S}C)\mu_C\dhomotopic{\mathfrak{d}}\id_{\catfont{S}C},\quad\nu_C(\cubcatreplacement\,C\ira\catfont{S}C)\dhomotopic{\mathfrak{d}}\id_{\cubcatreplacement\,C}$$
  natural in cubical sets $C$.
\end{lem}
\begin{proof}
  Let $C$ denote a cubical set.
  Let 
  $$e_C=\direalize{\epsilon^2_C}\dihomeo^{-1}_{C;9}:\direalize{C}\ra\direalize{C}.$$

  There exists a dotted cubical function $\nu_C$, unique by the diagonal inclusion monic and hence natural in $C$, making the following diagram commute [Lemma \ref{lem:cubcats.are.sing.algebras}]:
  \begin{equation}
    \label{eqn:cubcats.are.homotopy.sing.algebras}
    \begin{tikzcd}
	    \catfont{S}C\ar{rr}[above]{\sing\,e_C}\ar[dotted]{dr}[description]{\nu_C} & & \catfont{S}C \\
	& \cubcatreplacement C\ar[ur,hookrightarrow]
    \end{tikzcd}
  \end{equation}
   
  There exist $\mathfrak{d}_*$-homotopies, natural in $C$, of the following forms [Lemma \ref{lem:natural.approximations}]:
  \begin{equation}
    \label{eqn:homotopies.making.cubcats.homotopy.sing.algebras}
	  \id_{\catfont{S}C}=\catfont{S}(\id_{C})\dhomotopic{\mathfrak{d}}\sing\,e_C.
  \end{equation}
  
  Let $e^\sharp_C$ denote the cubical function $C^\sharp\ra C^\sharp$ induced by $\sing\,e_{\BOX\boxobj{n}}$.  
  The restriction of $\sing\,e_C$ to $\ex_2^iC^\sharp$ corestricts to a cubical function 
  $$\ex_2^ie^\sharp_{C}:\ex_2^iC^\sharp\ra\ex_2^iC^\sharp.$$
  The $\mathfrak{d}_*$-homotopies (\ref{eqn:homotopies.making.cubcats.homotopy.sing.algebras}) for $C$ representable induce $\mathfrak{d}_*$-homotopies 
  $$\id_{C^\sharp}\dhomotopic{\mathfrak{d}}e^\sharp_{C}.$$
  
  Application of $\ex^i$ to these $\mathfrak{d}_*$-homotopies defines $\mathfrak{d}_*$-homotopies 
  $$\id_{\ex^iC^{\sharp}}\dhomotopic{\mathfrak{d}}\ex_2^ie^\sharp_{C}$$
  natural in $C$. 
  These $\mathfrak{d}_*$-homotopies in turn induce $\mathfrak{d}_*$-homotopies
  $$\id_{\cubcatreplacement\,C}\dhomotopic{\mathfrak{d}}\nu_C(\cubcatreplacement C\ira\catfont{S}C).$$
  natural in $C$.
\end{proof}

The following lemma asserts that $\catfont{S}$ is homotopy idempotent when restricted to $\REGULARCUBICALSETS$, the full subcategory of $\CUBICALSETS$ consisting of those cubical sets whose atomic subpresheaves are all isomorphic to representables.  

\begin{lem}
  \label{lem:restricted.homotopy.idempotent.comonad}
  There exist $\direalize{\mathfrak{d}}_*$-homotopies
  $$\id_{\direalize{\;\catfont{S}C\;}}\dhomotopic{\direalize{\;\mathfrak{d}\;}}\direalize{\eta_C}\epsilon_{\direalize{\;C\;}}$$
  natural in $\REGULARCUBICALSETS$-objects $C$, where $\eta,\epsilon$ denote the unit and counit of $\direalize{-}\dashv\sing$.  
\end{lem}

The idea of the proof is to lift a stream map $\theta^*:\direalize{\BOX\boxobj{n_\theta}}\ra\direalize{C_\theta}$ along a stream map $\direalize{C^*_{\theta;k}\ra C}$ up to natural $\direalize{\mathfrak{d}}_*$-homotopy based on a $2^k$-fold subdivision of $\BOX\boxobj{n_\theta}$ to a stream admitting a $\direalize{\mathfrak{d}}$-convexity structure natural in $\theta$ and, in a suitable sense, $k$.  

\begin{proof}
  Let $\iota$ denote inclusion $\REGULARCUBICALSETS\ira\CUBICALSETS$. 
  Let $\theta$ denote a $(\BOX/(\cubcatreplacement\iota))$-object
  $$\theta:\BOX\boxobj{n_\theta}\ra\cubcatreplacement\,C_\theta,$$
	a choice of $\REGULARCUBICALSETS$-object $C_\theta$ and cube $\BOX\boxobj{n_\theta}\ra\cubcatreplacement\,C_\theta$.
	Let $\theta^*$ denote the adjoint
	$$\theta^*:\direalize{\BOX\boxobj{n_\theta}}\ra\direalize{C_\theta}$$ 
	to $((\cubcatreplacement\,C_\theta)\ira\catfont{S}C_\theta)\theta$. 

  \vspace{.1in}
  \textit{goal}:
  It suffices to construct $\direalize{\mathfrak{d}}_*$-homotopies
  $$(\direalize{\cubcatreplacement\,C\ira\catfont{S}C})\dhomotopic{\direalize{\;\mathfrak{d}\;}}\direalize{\eta_C}\epsilon_{\direalize{\;C\;}}(\direalize{\cubcatreplacement\,C\ira\catfont{S}C})$$
  natural $\REGULARCUBICALSETS$-objects $C$ [Lemma \ref{lem:cubcats.are.homotopy.sing.algebras}].

	  \vspace{.1in}
	  \textit{constructing lift $\theta_k$ of $\theta^*$}:
	  Let $k\gg 0$.  
	  Let $\mathcal{I}_{\theta;k}$ denote the poset, natural in $\theta$, of all Boolean intervals in $\sd_2^k\boxobj{n_\theta}$ ordered under inclusion. 
	  Let $I$ denote a $\mathcal{I}_{\theta;k}$-object.
	  Let $C_{\theta;I}=\support_{|-|}\left(\theta^*|\BOX[I]|,C_\theta\right)$.
	  The stream map $\theta^*\dihomeo_{\BOX\boxobj{n_\theta};2^k}$ restricts and corestricts to a stream map $\theta_I:\direalize{\BOX[I]}\ra\direalize{C_{\theta;I}}$ [Proposition \ref{prop:realization.preserves.embeddings}].  
          Define $C^*_{\theta;k}$ as the $\CUBICALSETS$-colimit
	  $$C^*_{\theta;k}=\colim_{I}\left(C_{\theta;I}\otimes\BOX[I]\right).$$
	  Then $C^*_{\theta;k}$ is an iterated pushout of inclusions of $\REGULARCUBICALSETS$-objects and hence inductively is itself an $\REGULARCUBICALSETS$-object because $C_{\theta;I}\otimes\BOX[I]$'s are $\REGULARCUBICALSETS$-objects and $\mathcal{I}_{\theta;k}$ is a poset of the form $(\bullet\ra\bullet\la\bullet\ra\cdots\bullet)^n$.  
	  
	  Define $\REGULARCUBICALSETS$-morphisms $\rho_{\theta;k}$ and $\nu_{\theta;k}$, natural in $\theta$, by commutative diagrams
	  \begin{equation*}
	    \begin{tikzcd}
		    C_{\theta;I}\otimes\BOX[I]\ar{rr}[above]{C_{\theta;I}\otimes(\BOX[I\ra\star])}\ar[d,hookrightarrow]
		    & & C_{\theta;I}\ar[d,hookrightarrow]
	\\
	      C^*_{\theta;k}\ar{rr}[below]{\rho_{\theta;k}} & & C
	    \end{tikzcd}
	    \quad
	\begin{tikzcd}
	        C_{\theta;I}\otimes\BOX[I]\ar{rr}[above]{(C_{\theta;I}\ra\star)\otimes\BOX[I]}\ar[d,hookrightarrow]
		& & \BOX[I]\ar[d,hookrightarrow]
	\\
		C^*_{\theta;k}\ar{rr}[below]{\nu_{\theta;k}} & & \sd_2^k\boxobj{n_\theta}
	    \end{tikzcd}
	  \end{equation*}
	  The left of the solid commutative diagrams 
	  \begin{equation}
		  \label{eqn:lift}
	     \begin{tikzcd}
		\direalize{\BOX[I]}\ar{r}[above]{\theta_I\times\id_{\direalize{\;\BOX[I]\;}}}\ar{d}[description]{\direalize{\;\BOX[I\ira\sd_2^k\boxobj{n_\theta}]\;}} & \direalize{C_{\theta;I}}\times\direalize{\BOX[I]}\ar[r,hookrightarrow] & \direalize{C^*_{\theta;k}}\ar{d}[description]{\direalize{\;\rho_{\theta;k}\;}}\\
		\direalize{\sd_2^k\BOX\boxobj{n_\theta}}\ar{r}[below]{\dihomeo_{\BOX\boxobj{n_\theta};2^k}}\ar[dotted]{urr}[description]{\theta^*_k} & \direalize{\BOX\boxobj{n_\theta}}\ar{r}[below]{\theta^*} & \direalize{C_\theta},
	    \end{tikzcd}\quad
	    \begin{tikzcd}
		    \cubcatreplacement C^*_{\theta;k}\ar[r,hookrightarrow] & \catfont{S}C^*_{\theta;k}\ar{d}[right]{\cubcatreplacement \rho_{\theta;k}}
		    \\
	    \BOX\boxobj{n}\ar[dotted]{u}[left]{\theta_k}\ar[ur,dotted]\ar{r}[below]{\theta} & \catfont{S}C_\theta
	    \end{tikzcd}
	  \end{equation}
	  natural in $\mathcal{I}_{\theta;k}$-objects $I$, induces a dotted stream map $\theta_k$, natural in $\theta$, making the entire left diagram commute.  
	  The adjoint of $\theta^*_k\dihomeo^{-k}_{\BOX\boxobj{n_\theta};2}$ defines a dotted diagonal arrow making the lower triangle commute in the right diagram above.  
	  There exists a dotted cubical function $\theta_k$, unique by the top horizontal arrow monic, making the entire right diagram commute by an application of Proposition \ref{prop:free.cubcat.monad}.  
	  
	  \vspace{.1in}
	  \textit{convexity structure on $\direalize{C^*_{\theta;k}}$}:
	  Let $\pi_{\theta;k;1}$ and $\pi_{\theta;k;2}$ denote respective projections
	  $$\pi_{\theta;k;1},\pi_{\theta;k;2}:\direalize{C^*_{\theta;k}}^2\ra\direalize{C^*_{\theta;k}}$$
	  onto first and second factors.  
	  Define $s_{\theta;k}$ by the commutative diagram
	  \begin{equation}
	    \label{eqn:convexities}
	    \begin{tikzcd}
		    \direalize{C^*_{\theta;k}}\ar{rr}[above]{s_{\theta;k}}\ar{dr}[description]{\nu_{\theta;k}} & & \direalize{C^*_{\theta;k}}\\
		    & \direalize{\sd_2^k\BOX\boxobj{n_\theta}}\ar{ur}[description]{\theta_k}
	    \end{tikzcd}
	  \end{equation}
	  natural in $\theta$.
	  For each $x\in|C^*_{\theta;k}|$, $s_{\theta;k}(x)$ and $x$ both lie in the same closed cell in $|C^*_{\theta;k}|$, the directed realization of an atomic subpresheaf of $C^*_{\theta;k}$ and hence the directed realization of a representable up to isomorphism by our assumption on $C$.  
	  Thus there exists a $\direalize{\mathfrak{d}_1}$-homotopy $s_{\theta;k}\topologicaldhomotopic{}\id_{\direalize{\;C^*_{\theta;k}\;}}$ natural in $\theta$ [Lemma \ref{lem:hypercube.convexity}].  
	  The stream maps $s_{\theta;k}\pi_{\theta;k;1},s_{\theta;k}\pi_{\theta;k;2}$ both naturally factor through $\direalize{\BOX\boxobj{n_\theta}}$.
	  Thus there exists a $\direalize{\mathfrak{d}_1}$-homotopy $s_{\theta;k}\pi_{\theta;k;1}\topologicaldhomotopic{}s_{\theta;k}\pi_{\theta;k;2}$ natural in $\theta$ [Lemma \ref{lem:hypercube.convexity}].   
	  Concatenating the $\direalize{\mathfrak{d}_1}$-homotopies
	  $$\pi_{\theta;k;1}\topologicaldhomotopic{}s_{\theta;k}\pi_{\theta;k;1}\topologicaldhomotopic{}s_{\theta;k}\pi_{\theta;k;2}\topologicaldhomotopic{}\pi_{\theta;k;2}$$
	  yields a $\mathfrak{d}_3$-homotopy $h^*_{\theta;k}:\pi_{\theta;k;1}\topologicaldhomotopic{}\pi_{\theta;k;2}$ natural in $\theta$. 

          \vspace{.1in}
	  \textit{constructing the requisite directed homotopy}:
  Consider the solid arrows in the diagram
  \begin{equation}
	    \label{eqn:deformation}
    \begin{tikzcd}
	    \direalize{\cubcatreplacement\,{C_{\theta;k}^*}}
	    \ar[drrr] 
	    \ar{rr}[above]{\direalize{\;\cubcatreplacement{C_{\theta;k}^*}\ira\catfont{S}C_{\theta;k}^*\;}}
	    \ar{dd}[description]{\catfont{S}\direalize{\;\rho_{\theta;k}\;}}
	    & & \direalize{\catfont{S}C_{\theta;k}^*}
	    \ar{dr}[description]{\direalize{\;\sing\,h^*_{\theta;k}\;}}
	      \ar{rr}[above]{\id_{\direalize{\;\catfont{S}C^*_{\theta;k}\;}}\times\epsilon_{\catfont{S}C^*_{\theta;k}}\direalize{\;\eta_{C^*_{\theta;k}}\;}}
	    & & \direalize{\catfont{S}\direalize{C^*_{\theta;k}}}^2
	    \ar{dd}[description,crossing over]{\catfont{S}\direalize{\;\rho_{\theta;k}^{\otimes 2}\;}}
	      \\
	     & \direalize{\BOX\boxobj{n_\theta}}
	     \ar{dl}[description]{\direalize{\;\theta\;}}
	     \ar{ul}[description]{\theta_k}
	     \ar[rr,dotted]
	     \ar[dotted]{ddrr}[description]{h_{\theta;k}}
	     & & \direalize{\catfont{S}\direalize{C^*_{\theta;k}}}^{\direalize{\;\mathfrak{d}_3[1]\;}}
	     \ar{dd}[description,near start]{{\catfont{S}\direalize{\;\rho_{\theta;k}\;}^{{\direalize{\;\mathfrak{d}_3[1]\;}}}}}
	     \ar[ur]
	    	     \\
		     \direalize{\cubcatreplacement\,C_\theta}
		     \ar{rr}[description,crossing over]{\direalize{\;\cubcatreplacement{C_{\theta}}\ira\catfont{S}C_{\theta}\;}}
	             \ar[drrr,dotted] 
		     & & 
                     \direalize{\catfont{S}C_{\theta}}
	      \ar{rr}[description,near start]{\id_{\direalize{\;\catfont{S}\direalize{\;C_\theta\;}\;}}\times\epsilon_{\catfont{S}\direalize{\;C\;}}\direalize{\;\eta_{C}\;}}
	        & & \direalize{\catfont{S}\direalize{C_\theta}}^2
	    \\
	      & & & \direalize{\catfont{S}\direalize{C_\theta}}^{\direalize{\;\mathfrak{d}_3[1]\;}}
	     \ar[ur]
    \end{tikzcd}
  \end{equation}
  The top triangle commutes by construction of $h^*_{\theta;k}$ and the left triangle commutes by construction of $\theta_k$.
  Therefore there exist remaining dotted stream maps, such as $h_{\theta;k}$, making the entire diagram commute.

  \vspace{.1in}
  \textit{invariance of $h_{\theta;k}$ in $k$}:
  The top horizontal stream maps in the diagram
  \begin{equation*}
    \begin{tikzcd}
	    \direalize{\BOX[C_{\theta;I}]}\times\direalize{\BOX[I]}\ar{rrrrrrr}[above]{\direalize{\;C_{\theta;I}\ira C_{\theta;I'}\;}\otimes\dihomeo_{\BOX[I'];2}\direalize{\;\BOX[I\ira\sd_2\BOX[I']]\;}}\ar[d]
	    & & & & & & & \direalize{\BOX[C_{\theta;I'}]}\times\direalize{\BOX[I']}\ar[d]
\\
\direalize{C^*_{\theta;i}}\ar[dotted]{rrrrrrr}[below]{e_{\theta;i}} & & & & & & & \direalize{C^*_{\theta;i+1}}
    \end{tikzcd}
  \end{equation*}
  where $I$ denotes a $\mathcal{I}_{\theta;i+1}$-object and $I'=\support_{\sd_2}(I,\sd_2^k\boxobj{n_\theta})$, induce dotted stream maps $e_{\theta;k}$ natural in $\theta$.  
  The stream maps $e_{\theta;i}$, $\dihomeo_{\BOX[I'];2}\direalize{\BOX[I]\ira\BOX[\sd_2I']}$, $\direalize{C_{\theta;I}\ira C_{\theta;I'}}$, and $\dihomeo_{\sd_2^i\BOX\boxobj{n_\theta};2}$ define the $k$-indexed components of natural transformations between diagrams in (\ref{eqn:lift}), (\ref{eqn:convexities}), and (\ref{eqn:deformation}) from the case $k=i+1\gg 0$ to the case $k=i\gg 0$ natural in $f$.  
  In particular, $h_{\theta;k}$ is independent of $k$ and therefore defines the desired $\direalize{\mathfrak{d}}_*$-homotopy.  
\end{proof}

The comonad of the adjunction $\direalize{-}^{\indexcat{1}}\dashv\sing^{\indexcat{1}}$ is directed homotopy idempotent by applying the following lemma for the case $C=\sing\,X$ for a $\indexcat{1}$-stream $X$.  
A proof boostraps the case $C$ a $\REGULARCUBICALSETS$-object [Lemma \ref{lem:restricted.homotopy.idempotent.comonad}] by using the fact that $C$ can be essentially replaced by a colimit of $\REGULARCUBICALSETS$-objects.  
This proof, which requires the language of pro-objects, is relegated to the end of Appendix \S\ref{sec:diagram.replacement}.

\begin{lem}
  \label{lem:homotopy.idempotent.comonad}
  Let $\eta,\epsilon$ denote the unit and counit of $\direalize{-}\dashv\sing$.  
  There exist $\direalize{\mathfrak{d}}_*$-homotopies
  \begin{equation}
    \label{eqn:cubcat.homotopy.idempotency}
	  \id_{\direalize{\;\catfont{S}C\;}}\dhomotopic{\direalize{\;\mathfrak{d}\;}}\direalize{\eta_C}\epsilon_{\direalize{\;C\;}}
  \end{equation}
  natural in cubical sets $C$.
\end{lem}

\begin{lem}
  \label{lem:topological.we}
  The following are equivalent for a $\indexcat{1}$-cubical function $\psi:A\ra B$.
  \begin{enumerate}
    \item $\direalize{\psi}$ is a $\direalize{\mathfrak{d}}_*$-equivalence.
    \item $\catfont{S}\psi$ is a $\mathfrak{d}_*$-equivalence
  \end{enumerate}
\end{lem}
\begin{proof}
  Let $\eta$ denote the unit of $\direalize{-}^{\indexcat{1}}\dashv\sing^{\indexcat{1}}$.
  Consider the commutative diagram
  \begin{equation*}
    \begin{tikzcd}
	    \direalize{\catfont{S}A}\ar{rr}[above]{\direalize{\;\catfont{S}\psi\;}}\ar{d}[left]{\epsilon_{\direalize{\;A\;}}} & & \direalize{\catfont{S}B}\ar{d}[right]{\epsilon_{\direalize{\;B\;}}}\\
	    \direalize{A}\ar{rr}[below]{\direalize{\;\psi\;}} & & \direalize{B}
    \end{tikzcd}
  \end{equation*}

  The vertical arrows are $\direalize{\mathfrak{d}}_*$-equivalences [Lemma \ref{lem:homotopy.idempotent.comonad}]. 
  If (1), then $\catfont{S}\psi$ is the image of a $\direalize{\mathfrak{d}}_*$-equivalence under $\sing$ and hence (2).  
  If (2), the top horizontal arrow is the image of a $\mathfrak{d}_*$-equivalence under $\direalize{-}$ and hence (1).
\end{proof}

The monad of the adjunction $\direalize{-}^{\indexcat{1}}\dashv\sing^{\indexcat{1}}$ is directed homotopy idempotent by applying the following lemma for the case $X=\direalize{C}$ for a $\indexcat{1}$-cubical set $C$. 

\begin{lem}
  \label{lem:homotopy.idempotent.monad}
  Let $\eta,\epsilon$ denote the unit, counit of $\direalize{-}^{\indexcat{1}}\dashv\sing^{\indexcat{1}}$.  
  There exist $\mathfrak{d}_*$-homotopies
  \begin{equation}
    \label{eqn:cubcat.homotopy.idempotency}
	  \id_{\catfont{S}(\sing\,X)}\dhomotopic{\mathfrak{d}}(\eta_{\sing\,X})(\sing\,\epsilon_{X})
  \end{equation}
  natural in $\indexcat{1}$-streams $X$.
\end{lem}
\begin{proof}
  Let $\catfont{Q}$ be the comonad of $\direalize{-}^{\indexcat{1}}\dashv\sing^{\indexcat{1}}$ with counit $\epsilon$.   
  The comultiplication
  $$\catfont{Q}X\ra\catfont{Q}^2X$$
  admits two retractions $\catfont{Q}\epsilon_X$ and $\epsilon_{\catfont{Q}X}$ by the zig-zag identities and defines a $\direalize{\mathfrak{d}}_*$-equivalence [Lemma \ref{lem:homotopy.idempotent.comonad}].  
  Therefore there exist $\direalize{\mathfrak{d}}_*$-homotopies $\catfont{Q}\epsilon_X\dhomotopic{\direalize{\;\mathfrak{d}\;}}\epsilon_{\catfont{Q}X}:\catfont{Q}^2X\ra\catfont{Q}X$.
  Taking adjoints gives the desired $\mathfrak{d}_*$-homotopies.
\end{proof}

\begin{lem}
  \label{lem:cubical.we}
  The following are equivalent for a $\indexcat{1}$-stream map $f:X\ra Y$.
  \begin{enumerate}
	  \item $\direalize{\sing\,f}$ is a $\direalize{\mathfrak{d}}_*$-equivalence.
	  \item $\sing\,f$ is a $\mathfrak{d}_*$-equivalence
  \end{enumerate}
\end{lem}
\begin{proof}
  Let $\eta$ denote the unit of $\direalize{-}^{\indexcat{1}}\dashv\sing^{\indexcat{1}}$.
  Consider the commutative diagram
  \begin{equation*}
    \begin{tikzcd}
	    \sing\,X\ar{rr}[above]{\sing\,f}\ar{d}[left]{\eta_{\sing\,X}} & & \sing\,Y\ar{d}[right]{\eta_{\sing\,X}}\\
	    \catfont{S}(\sing\,X)\ar{rr}[below]{\catfont{S}(\sing\,f)} & & \catfont{S}(\sing\,Y).
    \end{tikzcd}
  \end{equation*}

  The vertical arrows are $\mathfrak{d}_*$-equivalences [Lemma \ref{lem:homotopy.idempotent.monad}]. 
  If (1), the bottom horizontal arrow is the image of a $\direalize{\mathfrak{d}}_*$-equivalence under $\sing$ and hence (2).   
  If (2), then $\direalize{\sing\,f}$ is the image of a $\mathfrak{d}_*$-equivalence under $\direalize{-}$ and hence (1).
\end{proof}

\begin{lem}
  \label{lem:cubcat.multiplication}
  Let $\eta$ be the unit of $\cubcatreplacement$.  
  Fix a $\indexcat{1}$-cubical set $C$.  
  If the corestriction
  $$\eta_C:C\ra\cubcatreplacement\,C$$
  of the cubical function $C\ra\catfont{S}C$ defined by the unit of $\catfont{S}$ admits a retraction $\mu$, then $\id_{\cubcatreplacement\,C}\dhomotopic{\mathfrak{d}}\eta_C\mu$.
\end{lem}
\begin{proof}
  Let $\eta^{\catfont{S}}$ be the unit of $\catfont{S}$.  

  Consider a retraction $\mu$ to $\eta_C$.  
  Consider the naturality square
  \begin{equation*}
    \begin{tikzcd}
	    C\ar{r}[above]{\eta_C} & \cubcatreplacement C\\
	    \cubcatreplacement C\ar{u}[left]{\mu}\ar{r}[below]{\eta_{\cubcatreplacement C}} & \cubcatreplacement^2C\ar{u}[right]{\cubcatreplacement\,\mu}
    \end{tikzcd}
  \end{equation*}
  The cubical functions $\eta_{\catfont{S}C}$ and $\sing\eta^{\catfont{S}}_C$ are both sections to a common $\mathfrak{d}_*$-equivalence by an application of Lemma \ref{lem:cubcats.are.homotopy.sing.algebras} and the zig-zag identities and hence are $\mathfrak{d}_*$-homotopic.  
  Therefore the bottom horizontal arrow is $\mathfrak{d}_*$-homotopic to $\cubcatreplacement\,\eta_{C}$ because inclusion $\cubcatreplacement C\ira\catfont{S}C$ is a $\mathfrak{d}_*$-equivalence [Lemma \ref{lem:homotopy.idempotent.monad}].
  Hence the composite diagonal arrow is $\mathfrak{d}_*$-homotopic to the identity $\id_{\cubcatreplacement C}$. 
  Hence the desired $\mathfrak{d}_*$-homotopy follows.
\end{proof}

\begin{thm}
  \label{thm:derived.formula}
  For $\indexcat{1}$-cubcats $C$, natural $\indexcat{1}$-stream maps of the form
  $$\direalize{\CUBICALSETS^{\indexcat{1}}(B,C)}\ra\DITOP^{\indexcat{1}}(\direalize{B},\direalize{C})$$
  represent $d(\DITOP^{\indexcat{1}})$-isomorphisms. 
\end{thm}
\begin{proof}
  There exist $\mathfrak{d}_*$-equivalences
  $$\CUBICALSETS^{\indexcat{1}}(B,C)\simeq\CUBICALSETS^{\indexcat{1}}(B,\cubcatreplacement\,C)\simeq\CUBICALSETS^{\indexcat{1}}(B,\catfont{S}C)=\sing\,\DITOP^{\indexcat{1}}(\direalize{B},\direalize{C}),$$
	the first by Lemma \ref{lem:cubcat.multiplication} because there exists a $\mathfrak{d}_*$-equivalence $C\ra\catfont{R}C$ [Lemma \ref{lem:cubcat.multiplication}] and the second by Lemma \ref{lem:cubcats.are.homotopy.sing.algebras}.  
  Taking the adjoint of the composite cubical homotopy equivalence yields the desired $\direalize{\mathfrak{d}}_*$-equivalence.
\end{proof}

We now state and prove a cubical approximation theorem for directed topology. 

\begin{cor}
  \label{cor:formula}
  For each $\indexcat{1}$-cubcat $C$, $\direalize{-}$ induces a natural bijection
  $$\pi_0\CUBICALSETS^{\indexcat{1}}(-,C)\cong[\direalize{-},\direalize{C}]_{\direalize{\;\mathfrak{d}\;}}:\OP{\left(\CUBICALSETS^{\indexcat{1}}\right)}\ra\SETS.$$
\end{cor}
\begin{proof}
  There exist natural bijections
  $$\pi_0\CUBICALSETS^{\indexcat{1}}(-,C)\cong\pi_0\CUBICALSETS^{\indexcat{1}}(-,\cubcatreplacement\,C)\cong\pi_0\CUBICALSETS^{\indexcat{1}}(-,\catfont{S}C)\cong[\direalize{-},\direalize{C}]_{\direalize{\;\mathfrak{d}\;}}$$
  [Lemmas \ref{lem:cubcat.multiplication} and \ref{lem:cubcats.are.homotopy.sing.algebras}].
\end{proof}

The calculation $\tau_n(\direalize{\nerve\,M},\direalize{\star})=\tau_n(\nerve\,M,\star)$ [Proposition  \ref{prop:homotopy.calculations}] generalizes. 

\begin{cor}
  \label{cor:directed.milnor}
  For each cubcat $C$ and $v\in C_0$, the unit of $\direalize{-}\dashv\sing$ induces a bijection
  $$\tau_n\left(C,v\right)\cong\tau_n\left(\catfont{S}C,v\right)\quad n=0,1,2,\ldots$$
\end{cor}
\begin{proof}
  Let $\eta$ denote the unit to $\cubcatreplacement$.  
  The based cubical set $(C,v)$, regarded as a $[1]$-cubical set $\star\ra C$ whose image has vertex $v$, is a $[1]$-cubcat because $C$ is a cubcat and $\cubcatreplacement(C,v)=(\cubcatreplacement C,v)$.    
  The identification then follows [Corollary \ref{cor:formula} in the case $\indexcat{1}=[1]$].
\end{proof}

\begin{proof}[proof of Proposition \ref{prop:fibrant.cubcats}]
  Let $\eta$ denote the unit of $\direalize{-}\dashv\csing$.
  Then $\direalize{\eta_C}$ is a $\direalize{\mathfrak{d}}_*$-equivalence [Lemma \ref{lem:homotopy.idempotent.comonad}], hence $|\eta_C|$ is a $|\mathfrak{d}|$-equivalence, hence $\eta_C$ is an acyclic cofibration from a fibrant object in a model structure on $\CUBICALSETS$ [Proposition \ref{prop:quillen.equivalence}], and hence $\eta_C$ admits a retraction. 
  The result follows [Proposition \ref{prop:free.cubcat.monad}].
\end{proof}

\begin{proof}[proof of Theorem \ref{thm:equivalence}]
  Let $\hat{d}(\hat\BOX^{\indexcat{1}})$ and $\hat{d}(\DITOP^{\indexcat{1}})$ denote the quotient categories
  $$\hat{d}(\hat\BOX^{\indexcat{1}})=\quotient{\hat\BOX^{\indexcat{1}}}{\dhomotopic{\mathfrak{d}}}\quad\hat{d}(\DITOP^{\indexcat{1}})=\quotient{\DITOP^{\indexcat{1}}}{\topologicaldhomotopic{}}$$

  Let $\unit,\epsilon$ denote the unit, counit of the adjunction $\direalize{-}^{\indexcat{1}}\dashv\sing^{\indexcat{1}}$.  

  Let $\catfont{Q}$ denote the comonad of $\direalize{-}^{\indexcat{1}}\dashv\sing^{\indexcat{1}}$.
  Then $\catfont{Q}\epsilon_X,\epsilon_{\catfont{Q}X}$, both retractions to $\direalize{\eta_{\sing\,X}}$, represent the same $\hat{d}(\DITOP^{\indexcat{1}})$-isomorphism because $\direalize{\eta_{\sing\,X}}$ represents a $\hat{d}(\DITOP^{\indexcat{1}})$-isomorphism [Lemma \ref{lem:homotopy.idempotent.comonad}].  
  Thus $\catfont{Q}$ induces an idempotent comonad on $\hat{d}(\DITOP^{\indexcat{1}})$.  

  Thus we can define $d(\DITOP^{\indexcat{1}})$ to be the localization $\cat{1}[\mathscr{W}^{-1}]$ of $\cat{1}$ by a saturated class $\mathscr{W}$ where: $\cat{1}=\DITOP^{\indexcat{1}}$ and $\mathscr{W}$ is the class of all $\indexcat{1}$-stream maps $f$ for which $\catfont{Q}f$ represents a $\hat{d}(\DITOP^{\indexcat{1}})$-isomorphism; or equivalently $\cat{1}=\DITOP^{\indexcat{1}}$ and  $\mathscr{W}$ is the class of all $\indexcat{1}$-stream maps $f$ for which $\sing\,f$ are $\mathfrak{d}_*$-equivalences [Lemma \ref{lem:topological.we}].
  
  Let $\catfont{S}$ denote the monad of $\direalize{-}^{\indexcat{1}}\dashv\sing^{\indexcat{1}}$.
  Then $\catfont{S}\eta_C,\eta_{\catfont{S}C}$, both sections to $\sing\,\epsilon_{\direalize{\;C\;}}$, represent the same $\hat{d}(\CUBICALSETS^{\indexcat{1}})$-isomorphism because $\sing\,\epsilon_{\direalize{\;C\;}}$ represents a $\hat{d}(\DITOP^{\indexcat{1}})$-isomorphism [Lemma \ref{lem:homotopy.idempotent.monad}]. 
  Thus $\catfont{S}$ induces an idempotent monad on $\hat{d}(\CUBICALSETS^{\indexcat{1}})$.  

  Thus we can take $d(\CUBICALSETS^{\indexcat{1}})=\cat{1}[\mathscr{W}^{-1}]$ where $\cat{1}=\CUBICALSETS^{\indexcat{1}}$ and $\mathscr{W}$ consists of the $\indexcat{1}$-cubical functions $\psi$ for which $\catfont{S}\psi$ represents a $\hat{d}(\CUBICALSETS^{\indexcat{1}})$-isomorphism, or equivalently $\cat{1}=\CUBICALSETS^{\indexcat{1}}$ and $\mathscr{W}$ consists of the $\indexcat{1}$-cubical functions $\psi$ for which $\direalize{\psi}$ are $\direalize{\mathfrak{d}}_*$-equivalences [Lemma \ref{lem:cubical.we}]. 
  
  The adjunction $d(\CUBICALSETS^{\indexcat{1}})\lras d(\DITOP^{\indexcat{1}})$ induced by $\direalize{-}^{\indexcat{1}}\dashv\sing^{\indexcat{1}}$ is therefore a categorical equivalence because the unit and counit are natural isomorphisms.
\end{proof}

\begin{proof}[proof of Theorem \ref{thm:1-ditypes}]
  Let $d(\CATS^{\indexcat{1}})$ denote the quotient category
  $$d\left(\CATS^{\indexcat{1}}\right)=\quotient{\CATS}{\dhomotopic{\Tau_1\mathfrak{d}}}.$$

  The $\BOX$-morphisms $\Tau_1\mathfrak{d}(\delta_{\mins})=\delta_{\mins}$ and $\Tau_1\mathfrak{d}(\delta_{\pls})=\delta_{\pls}:[0]\ra[1]$ are adjunctions in $\CATS$ and therefore $\Tau_1\mathfrak{d}_1$-equivalences.  
  Therefore the quotient functor $\CATS^{\indexcat{1}}\ra d(\CATS^{\indexcat{1}})$ is localization by the $\Tau_1\mathfrak{d}_*$-equivalences [Lemma \ref{lem:localization}].  
  Additionally, the fully faithful embedding $\nerve^{\indexcat{1}}:\CATS^{\indexcat{1}}\ra\CUBICALSETS^{\indexcat{1}}$ passes to a fully faithful embedding $d(\CATS^{\indexcat{1}})\ra d(\CUBICALSETS^{\indexcat{1}})$ because $\indexcat{1}$-cubical sets of the form $\nerve\,\smallcat{1}$ for $\indexcat{1}$-categories $\smallcat{1}$ are $\indexcat{1}$-cubcats [Proposition \ref{prop:nerves}] and consequently $d\CATS^{\indexcat{1}}(\smallcat{1},\smallcat{2})=d(\CUBICALSETS^{\indexcat{1}})(\nerve\,\smallcat{1},\nerve\,\smallcat{1})$ [Corollary \ref{cor:formula}].
\end{proof}

\section{Conclusion}\label{sec:conclusion}


Recent years have seen computations modelled abstractly by homotopy types.
A (dependently) typed higher order programming language for reversible computations has been shown to admit semantics in $\infty$-groupoids (fibered over other $\infty$-groupoids) \cite{awodey2009homotopy}.
Objects represent states, $1$-morphisms represent reversible executions, and higher order morphisms represent reversible transformations of those executions, or equivalently, concurrent executions of sequential computations. 
Since $\infty$-equivalences between $\infty$-groupoids ignore differences like subdivisions, state space reduction is built into the very syntax of the language.  
This language, higher order, can thus be used to reason efficiently about computations expressed in the same language.  
The recent literature has seen extensions of (dependent) type theory to synthetic theories of (fibered) higher categories \cite{licata2011directed,riehl2017type}.
These more expressive languages model irreversible computations \cite{licata2011directed} because the morphisms in higher categories need not be invertible.
Ideally, (dependent) type theory can be alternatively extended so that edges in (fibered) cubcats represent computations, higher cubes in (fibered) cubcats represent higher order transformations, and directed homotopy invariance is built into the syntax of the language (c.f. \cite{north2018type}).  
Such a language ought to share both some of the efficiency of automated reasoning within dependent type theory as well as some of the expressiveness of synthetic higher category theory.  

\section{Acknowledgements}
This work was supported by AFOSR grant FA9550-16-1-0212.
The author is grateful to Robert Ghrist for conceiving of and producing the visualizations of conal manifolds behind Figure \ref{fig:compact.timelike.surfaces}.  
The author would like to thank Emily Rudman for pointing out some simplifications in earlier proofs of Lemmas \ref{lem:box.automorphisms} and \ref{lem:surjective.box.morphisms}.  
The author would like to thank Christian Sattler for pointing out a problem with the test model structure on cubical sets, used in earlier drafts of the paper.    
The author would also like to acknowledge Drs. David Balduzzi and Vikram Duvvuri.  

\appendix
\addcontentsline{toc}{section}{Appendix}
\addtocontents{toc}{\protect\setcounter{tocdepth}{0}}

\section{Modular lattices}\label{sec:modular.lattices}
We recall and extend some observations about \textit{modular} lattices, especially distributive lattices, in this section.
We first establish some notation and terminology.  
For an ordered pair $x\leqslant_Py$ in a poset $P$, write $[x,y]_P$ for the interval in $P$ containing $x$ as its minimum and $y$ as its maximum.  
Call $y$ an \textit{immediate successor to $x$} and $x$ an \textit{immediate predecessor to $y$} in a poset $P$ if $x\leqslant_Py$ and $[x,y]_P\cong[1]$. 
A lattice $L$ is \textit{modular} if for all $x,y,z\in L$,
$$(x\wedge_Ly)\vee_L(x\wedge_Lz)=((x\wedge_Ly)\vee_L z)\wedge_Lx.$$

\begin{eg}
  Distributive lattices are modular.
\end{eg}

The following \textit{Diamond Isomorphism Theorem} characterizes modular lattices. 

\begin{thm:diamond.isomorphism}
  The following are equivalent for a lattice $L$.
  \begin{enumerate}
	  \item $L$ is modular.
	  \item For each $x,y\in L$, the rules $x\vee_L-$ and $y\wedge_L-$ define respective bijections 
		  $$[x\wedge_Ly,y]_L\cong[x,x\vee_Ly]_L,\quad[x,x\vee_Ly]_L\cong[x\wedge_Ly,y]_L.$$
          forming an adjoint equivalence of posets.  
  \end{enumerate}
\end{thm:diamond.isomorphism}

\begin{thm:intervalic.distributivity.characterization}
  The following are equivalent for a finite lattice $L$.
  \begin{enumerate}
    \item $L$ is distributive
    \item For all $x,y,z\in L$ with $y,z$ either both immediate successors to $x$ or both immediate predecessors to $x$ in $L$, $\{y\wedge_Lz,y\vee_Lz,y,z\}$ is a Boolean interval in $L$.
    \item The smallest interval in $L$ containing Boolean intervals $I,J$ in $L$ with $\max\,I=\max\,J$ or $\min\,I=\min\,J$ is also Boolean.
  \end{enumerate}
\end{thm:intervalic.distributivity.characterization}

\begin{lem}
  \label{lem:interval.lattices}
  For Boolean intervals $I,J$ in a finite distributive lattice $L$, the images
  $$I\vee_LJ,I\wedge_LJ$$
  of $I\times J$ under $\vee_L,\wedge_L$ are Boolean intervals in $L$.   
\end{lem}
\begin{proof}
  The intervals $I\vee_L\min\,J,J\vee_L\min\,I$ are Boolean by the Diamond Isomorphism for Modular Lattices and hence $I\vee_LJ$ is Boolean [\THMIntervalicDistributivityCharacterization{}].  
  Thus $I\wedge_LJ$ is also Boolean by duality.
\end{proof}

While every finite poset, including every finite lattice, is a colimit of its ($1$-dimensional) Boolean intervals in the category of posets and monotone functions, not every finite lattice is such a colimit \textit{in the category $\CATS$}. 
\begin{lem}
  \label{lem:distributive.lattices}
  Every finite distributive lattice is a $\CATS$-colimit of its Boolean intervals.
\end{lem}
\begin{proof}
  Consider a finite distributive lattice $L$.  
  Let $\smallcat{1}$ be the $\CATS$-colimit 
  $$\smallcat{1}=\colim_{I\ra L}I$$
  over the Boolean intervals $I$ in $L$.  
  The object sets of $\smallcat{1}$ and $L$ coincide.  
  There exists a relation $x\leqslant_Ly$ if and only if there exists a $\smallcat{1}$-morphism $x\ra y$ because $\smallcat{1}$ admits as generators relations of the form $x\leqslant_Ly$ with $y$ an immediate successor to $x$ in $L$.
  Consider parallel $\smallcat{1}$-morphisms $\alpha,\beta:x\ra y$.
  It thus suffices to show 
  $$\alpha=\beta.$$
  
  We induct on the length $k$ of a maximal chain in $L$ having minimum $x$ and maximum $y$.
  In the base case $k=1$, $\alpha,\beta$ are identities and hence $\alpha=\beta$.  
  Inductively assume that $\alpha=\beta$ when there exists a maximal chain in $L$ having minimum $x$ and maximum $y$ with length less than $k$.  

  Consider the case $k>1$. 
  Then $\alpha,\beta$ factor as respective composites $x\ra a\ra y$ and $x\ra b\ra y$ with $a$ and $b$ both immediate successors to $x$ in $L$.
  Consider the diagram
  \begin{equation*}
    \begin{tikzcd}
		a\ar[rr]\ar[dr] & & a\vee_{L}b\ar[dl,dotted]
		\\
		& y
		\\
		x\ar[ur,dotted]\ar[rr]\ar[uu] & & b\ar[uu]\ar[ul],
    \end{tikzcd}
  \end{equation*}
  in $\smallcat{1}$, where $a\ra y$ and $b\ra y$ are unique $\smallcat{1}$-morphisms with the given domains and codomains by the inductive hypothesis.  
  There exist choices of dotted monotone function making the top and right triangles respectively commute.  
  These choices coincide by the inductive hypothesis.  
  The outer square, whose elements form a Boolean interval in $L$ [\THMIntervalicDistributivityCharacterization], commutes in $\smallcat{1}$.
  There exists a dotted morphism making the top and right triangles commute by the inductive hypothesis. 
  It therefore follows that $\alpha$, the choice of dotted monotone function making the left triangle commute, coincides with $\beta$, the choice of dotted monotone function making the bottom triangle commute.
\end{proof}

We can now give a proof of Lemma \ref{lem:cubical.pasting.schemes}.

\begin{proof}[proof of Lemma \ref{lem:cubical.pasting.schemes}]
  Suppose (1). 
  Let $I$ be a Boolean interval in $L$.  
  The restriction of $\phi$ to $I$ corestricts to a surjection $\phi_{I}:I\ra J_I$ with $J_I$ a Boolean interval in $M$ because $\phi$ preserves Boolean intervals.  
  The function $\phi_I:I\ra J_I$, surjective by construction, is a lattice homomorphism by $I\ira L$ and $J\ira M$ both inclusions of sublattices into lattices.  
  Thus (2).  

  Suppose (2).
  Consider $x,y\in L$.  
  It suffices to show that 
  \begin{equation}
    \label{eqn:cubical.pasting.scheme}
    \phi(x\vee_Ly)=\phi(x)\vee_M\phi(y). 
  \end{equation}
  by double induction on the minimal lengths $m,n$ of maximal chains in $L$ having as their extrema $x\wedge_Ly$ and, respectively, $x$ and $y$.  
  For then $\phi$ preserves binary suprema, hence also binary infima by duality, and hence $\phi$ is a lattice homomorphism, mapping Boolean intervals onto Boolean intervals by (2).

  In the case $m=1$, $x\wedge_Ly=x$, hence $x\vee_Ly=y$, hence $\phi(x)\leqslant_M\phi(x\vee_Ly)=\phi(y)$, and consequently (\ref{eqn:cubical.pasting.scheme}).
  The case $n=1$ follows by symmetry.  

  Consider the case $m=n=2$.  
  Then $x,y,x\wedge_Ly,x\vee_Ly$ form the elements of a Boolean interval $I$ in $L$ [\THMIntervalicDistributivityCharacterization].  
  Then the restriction of $\phi$ to $I$ corestricts to a Boolean interval $J_I$ in $M$.  
  It therefore follows from (2) and the preservation of finite non-empty suprema and infima by $I\ira L$ and $J_I\ira M$ that 
  $$\phi(x\vee_Ly)=\phi(x\vee_Iy)=\phi(x)\vee_{J_I}\phi(y)=\phi(x)\vee_M\phi(y).$$

  Consider the case $m\leqslant 2$.  
  The subcase $n\leqslant 2$ having just been established, take $k>2$, inductively assume the lemma for the subcase $2\leqslant n<k$, and now suppose $n=k$.  
  Then there exists an immediate successor $y'\neq y$ to $x\wedge_Ly$ such that $y'\leqslant_Ly$.
  Then $y\wedge_L(x\vee_Ly')=(x\wedge_Ly)\vee_Ly'=y'$ by $L$ distributive and hence the length of a maximal chain in $L$ having as its extrema $y\wedge_L(x\vee_Ly')$ and $y$ is strictly less than $n$.  
  And $x\wedge_Ly'=x\wedge_Ly$ and hence the length of a maximal chain in $L$ having as its extrema $x\wedge_Ly'$ and $y'$ is $2$.    
  It therefore follows from the inductive hypothesis that 
  \begin{align*}
    \phi(x\vee_Ly)
    &=\phi(x\vee_Ly'\vee_Ly)\\
    &=\phi(x\vee_Ly')\vee_M\phi(y)\\
    &=\phi(x)\vee_M\phi(y')\vee_M\phi(y)\\
    &=\phi(x)\vee_M\phi(y).
  \end{align*}

  The cases $m=1,2$ having been established, take $k>2$, inductively assume the lemma for $1\leqslant m<k$, and now suppose $m=k$. 
  Then there exists an immediate successor $x'\neq x$ to $x\wedge_Ly$ such that $x'\leqslant_Lx$.
  Then $x\wedge_L(x'\vee_Ly)=(x\wedge_Ly)\vee_Lx'=x'$ by $L$ distributive and hence the length of a maximal chain in $L$ having as its extrema $x\wedge_L(x'\vee_Ly)$ and $x$ is strictly less than $m$. 
  And $x'\wedge_Ly=x\wedge_Ly$ and hence the length of a maximal chain from $x'\wedge_Ly$ to $x$ is $2$. 
  Then 
  \begin{align*}
    \phi(x\vee_Ly)
    &=\phi(x\vee_Lx'\vee_Ly)\\
    &=\phi(x)\vee_M\phi(x'\vee_Ly')\\
    &=\phi(x)\vee_M\phi(x')\vee_M\phi(y)\\
    &=\phi(x)\vee_M\phi(y).
  \end{align*}
\end{proof}

The preservation of fully faithful embeddings by $\CATS$-pushouts \cite{trnkova1965sum} is used in the following proof of Proposition  \ref{prop:lattice.subdivision}.

\begin{proof}[proof of Proposition \ref{prop:lattice.subdivision}]
  Uniqueness follows by the right vertical inclusion monic.
  Let $F_k$ denote the left Kan extension of the composite $\BOX\ra\CATS$ of the top horizontal arrows along $\BOX\ira\DISLATS$.  
  To show existence, it suffices to show $F_k$ preserves $\DISLATS$-objects, $\DISLATS$-morphisms, and $\DISLATS$-tensors.    

  \vspace{.1in}
  \textit{$F_k$ preserves $\DISLATS$-objects:}
  Let $I$ denote a Boolean interval in $L$.  
  Inclusions of the forms $(I\ira L)^{[k]}:I^{[k]}\ra L^{[k]}$ are fully faithful embeddings.  
  It follows that the natural functor $F_kL\ra L^{[k]}$, an iterated pushout of inclusions of the form $(I\ra L)^{[k]}$ [Lemma \ref{lem:distributive.lattices}], is a full and faithful embedding and can henceforth be regarded as an inclusion of posets.
  In other words, we can identify $F_{k}L$ with the poset of all monotone functions $[k]\ra L$ which corestrict to Boolean intervals in $L$, with partial order $\leqslant_{F_{k}L}$ defined by $\alpha\leqslant_{F_{k}L}\beta$ if and only if $\alpha(i)\leqslant_L\beta(i)$ for each $0\leqslant i\leqslant k$.  
  
  Consider $\alpha,\beta\in\sd_{k+1}L$.
  The monotone functions $\alpha\vee_L\beta$ and $\alpha\wedge_L\beta$ corestrict to Boolean intervals in $L$ [Lemma \ref{lem:interval.lattices}].
  Thus $F_kL$ is a sublattice of the finite distributive lattice $L^{[k]}$ and hence is both finite and distributive. 

  \vspace{.1in}
  \textit{$F_k$ preserves $\DISLATS$-morphisms:}
  Consider a general $\DISLATS$-morphism $\phi:L\ra M$.  
  To show that $F_k\phi$ is a $\DISLATS$-morphism, it suffices to take the case $\phi$ a $\BOX$-morphism [Lemma \ref{lem:cubical.pasting.schemes}].
  Then $\phi$ is an iterated Cartesian monoidal product in $\CATS$ of $\delta_{\pm},\sigma$.
  Then $\phi^{[k]}=F_k\phi$ is an iterated Cartesian monoidal product in $\CATS$ of $\delta_{\pm}^{[k]}$ and $\sigma^{[k]}$ up to $\BOX$-isomorphisms by $(-)^{[k]}:\CATS\ra\CATS$ a right adjoint and hence product-preserving.  
  The functions $\delta_{\pm}^{[k]}$ and $\sigma^{[k]}$ are monotone functions to or from a terminal object and hence $\DISLATS$-morphisms.
  Hence $\phi$ is a $\DISLATS$-morphism.

  \vspace{.1in}
  \textit{$F_k$ is monoidal:}
  The right adjoint $(-)^{[k]}$ on $\CATS$, a right adjoint, is Cartesian monoidal.  
  Thus $F_k(\BOX\ira\DISLATS)$ sends tensor products to Cartesian monoidal products.
  Every tensor product in $\DISLATS$ is a $\CATS$-colimit of tensor products of finite Boolean lattices because every finite distributive lattice is a $\CATS$-colimit of its Boolean intervals [Lemma \ref{lem:distributive.lattices}] and every $\DISLATS$-tensor product commutes with colimits by $\CATS$ Cartesian closed.
  Therefore $F_k$ sends tensor products to Cartesian monoidal products.  

  \vspace{.1in}
  \textit{last claim:}
  In order to show that the natural transformation $F_m\ra F_n$ component-wise induced from $\phi$ defines the desired natural transformation $\sd_{m+1}\ra\sd_{n+1}$, it suffices to show that $J^\phi$ is a $\DISLATS$-morphism for each $\BOX$-object $J$ [Lemma \ref{lem:cubical.pasting.schemes}]. 
  It therefore suffices to take the case $J=[1]$ because $(-)^{\phi}$ is a Cartesian monoidal natural transformation $F_m\ra F_n$.
  In that case, non-singleton Boolean intervals in $J^{[m]}=[m+1]$ and $J^{[n]}=[n+1]$ are intervals between elements and their immediate successors.  
  Consider a non-singleton Boolean interval $I$ in $J^{[m]}=[1]^{[m]}$. 
  Let $\zeta_-=\min\,I$ and $\zeta_+=\max\,I$.  
  Then there exists $0\leqslant j\leqslant m$ such that $\zeta_-(i)=\zeta_{-}(i+1)=\zeta_+(i)=0$ for all $i<j$, $\zeta_+(i)=1$ for all $i\geqslant j$, and $\zeta_-(i)=1$ for all $i>j$.  
  The preimage of $j$ under $\phi$ is either a singleton or empty by $\phi$ injective.

  In the case that the preimage is empty, then $\zeta_-\phi=\zeta_+\phi$. 
  
  In the case that the preimage contains the unique element $j^*$, then $\phi(i)<j$ for all $i<j^*$, $\phi(i)\geqslant j$ for all $i\geqslant j^*$, and consequently $\zeta_-\phi(i)=\zeta_{-}\phi(i+1)=\zeta_+(i)=0$ for all $i<j^*$, $\zeta_+\phi(i)=1$ for all $i\geqslant j^*$, $\zeta_-\phi(i)=1$ for all $i>j^*$, and consequently $\zeta_+\phi$ is an immediate successor to $\zeta_-\phi$ in $[1]^{[n]}$.  
 
  In either case, $\{\phi\zeta_-,\phi\zeta_+\}$ is a Boolean interval in $[1]^{[n]}$.

  Thus $J^{\phi}$ is a $\DISLATS$-morphism.
\end{proof}

\section{Triangulations}\label{sec:triangulations}

Write $\DEL$ for the simplex category, whose objects are the non-empty finite ordinals 
$$[0],[1],\ldots$$
and whose morphisms are the monotone functions.
Write $\nerve_{\shape{1}}$ for the functor
$$\nerve_{\shape{1}}:\CATS\ra\hat{\shape{1}}$$
naturally defined on small categories $\smallcat{1}$ by the rule
$$\nerve_{\shape{1}}\smallcat{1}=\CATS(-,\smallcat{1})\OP{(\shape{1}\ira\CATS)}:\OP{\shape{1}}\ra\SETS,$$
for each small subcategory $\shape{1}$ of $\CATS$.

\begin{eg}
  The cubical nerve functor $\nerve$ is just $\nerve_{\BOX}:\CATS\ra\hat\BOX$.  
\end{eg}

\begin{eg}
  The usual simplicial nerve functor is just $\nerve_{\DEL}:\CATS\ra\hat\DEL$.  
\end{eg}

\textit{Triangulation} $\tri:\CUBICALSETS\ra\SIMPLICIALSETS$ is the dotted left Kan extension of the composite of the horizontal functors along the vertical Yoneda embedding in
\begin{equation*}
  \begin{tikzcd}
	  \BOX\ar[r,hookrightarrow]\ar{d}[left]{\BOX[-]} & \CATS\ar{r}[above]{\snerve} & \SIMPLICIALSETS
	  \\
	  \CUBICALSETS\ar[dotted]{urr}[description]{\tri}
  \end{tikzcd}
\end{equation*}

Write $\qua$ for the right adjoint to $\tri$.  

\begin{lem}
  \label{lem:qt.cocontinuous}
  The composite $\qua\,\tri:\CUBICALSETS\ra\CUBICALSETS$ is cocontinuous.  
\end{lem}

The idea behind the proof is that a simplicial function of the form 
$$\theta:\tri\,\BOX\boxobj{m}\ra\tri\,C,$$
an $m$-cube in $\qua\,\tri\,C$, has support in a cube in $C$ that is completely determined by where $\theta$ sends the diagonal $1$-simplex, the unique $1$-simplex connecting the extrema in $\boxobj{n}$, because $\BOX$-morphisms preserve extrema.  
The formal proof, identical to a proof in the setting where $\BOX$ is replaced with its minimal variant \cite[7.2]{krishnan2015cubical}, is omitted.

\vspace{.1in}
The usual Kan-Quillen model structure on $\SIMPLICIALSETS$ is left induced along topological realization $\SIMPLICIALSETS\ra\TOP$.  
There also exists a model structure on $\CUBICALSETS$ left induced along topological realization $\CUBICALSETS\ra\TOP$, whose proof we furnish at the end of this section. 
We will refer to both model structures as \textit{classical model structures}.  
Triangulation is the left map of a Quillen equivalence between such classical model structures.

\begin{prop}
  \label{prop:tri.equivalence}
  The adjunction $\tri\dashv\qua$ defines a Quillen equivalence
  $$\CUBICALSETS\simeq\SIMPLICIALSETS$$
  between presheaf categories equipped with their classical model structures.
  Moreover, the counit is a component-wise weak equivalence.  
\end{prop}
\begin{proof}
  Let $\epsilon$ denote the counit of the adjunction.
  It suffices to show the last line of the proposition.  For then the rest of the proposition would formally follow because triangulation creates the weak equivalences.  Let $\epsilon$ denote the counit of the $\tri\dashv\qua$. 
	Every simplicial set has the same geometric realization as a cubical set. 
            It therefore follows that $|\epsilon_S|$ induces a surjection on homotopy groups and path-components.  
	    Moreover, these surjections are all injective by an application of cubical approximation for directed topology \cite[Corollary 8.15]{krishnan2015cubical}.
\end{proof}

Consider a left adjoint functor $L:\mathscr{A}\ra\mathscr{B}$ where $\mathscr{B}$ is equipped with a model structure.  
Recall that a model structure on $\mathscr{A}$ is \textit{left induced} along $L$ if its weak equivalences and cofibrations are characterized as those $\mathscr{A}$-morphisms $\zeta$ whose images under $L$ are respectively weak equivalences and cofibrations in $\mathscr{B}$. 
A proof of Proposition \ref{prop:cubical.model.structure} uses the fact that $\mathscr{A}$ admits a model structure left transferred along $L$ if $\mathscr{A}$ and $\mathscr{B}$ are both locally presentable, the model structure on $\mathscr{B}$ is additionally cofibrantly generated, and for each $\mathscr{A}$-object $a$, the $\mathscr{A}$-morphism $\id_o\amalg\id_o:o\amalg o\ra o$ factors into a composite $\rho_o\iota_o$ with $L\iota_o$ a cofibration in $\mathscr{B}$ and $L\rho_o$ a weak equivalence in $\mathscr{B}$ \cite[special case of Theorem 2.2.1]{hess2017necessary}.  

\begin{proof}[proof of Proposition \ref{prop:cubical.model.structure}]
  Each cubical function of the form $\id_C\amalg\id_C:C\amalg C\ra C$ factors into a composite of the mono $C\otimes(\BOX[\delta_{\mins}]\amalg\BOX[\delta_{\pls}]):C\amalg C\ra C\otimes\BOX[1]$, whose image under triangulation is a mono, followed by the cubical function $C\otimes\BOX[\sigma]:C\otimes\BOX[1]\ra C$, whose image $(\tri\,C)\times\DEL[\sigma]:(\tri\,C)\times\DEL[1]\ra(\tri\,C)$ under triangulation is a weak equivalence.  
  Therefore $\CUBICALSETS$ admits a model structure left transferred along $\tri$ \cite[special case of Theorem 2.2.1]{hess2017necessary}.  
  This model structure is the desired model structure because topological realization $|-|:\CUBICALSETS\ra\TOP$ factors as the composite of triangulation followed by topological realization $|-|:\SIMPLICIALSETS\ra\TOP$ and the classical model structure on $\hat\DEL$ is left transferred from the q-model structure along topological realization $\SIMPLICIALSETS\ra\TOP$.
\end{proof}

A proof of Proposition \ref{prop:quillen.equivalence} uses the fact that topological realization $\SIMPLICIALSETS\ra\TOP$ defines the left map of a Quillen equivalence, where $\SIMPLICIALSETS$ is equipped its classical model structure and $\TOP$ is equipped with its q-model structure \cite{quillen1967homotopical}.

\begin{proof}[proof of Proposition \ref{prop:quillen.equivalence}]
  Topological realization $|-|:\CUBICALSETS\ra\TOP$ factors as a composite of triangulation $\tri:\CUBICALSETS\ra\SIMPLICIALSETS$ and topological realization $|-|:\SIMPLICIALSETS\ra\TOP$.
  The former is the left map of a Quillen equivalence [Proposition \ref{prop:tri.equivalence}].  
  The latter is the left map of a Quillen equivalence \cite{quillen1967homotopical}.  
\end{proof}

\begin{proof}[proof of Corollary \ref{cor:cubical.coherent.nerve}]
  Consider the following diagram
  \begin{equation*}
    \begin{tikzcd}
	    \CUBICALSETS\ar{rr}[above]{\tri} & & \SIMPLICIALSETS\\
	    & \CATS\ar{ur}[description]{\nerve_\DEL}\ar{ul}[description]{\nerve_\BOX}
    \end{tikzcd}
  \end{equation*}
	This triangle commutes up to natural classical weak equivalence in $\SIMPLICIALSETS$ because $\tri\circ\nerve_\BOX=\tri\circ\qua\circ\nerve_\DEL$ and the counit $\tri\circ\qua\ra\id_{\SIMPLICIALSETS}$ defines a component-wise classical weak equivalence [Proposition \ref{prop:tri.equivalence}].
	Therefore the claim follows because $\tri$ induces a categorical equivalence $h(\CUBICALSETS)\simeq h(\SIMPLICIALSETS)$ [Proposition \ref{prop:tri.equivalence}] and $\nerve_\DEL$ induces a categorical equivalence $h(\CATS)\simeq h(\SIMPLICIALSETS)$.
\end{proof}

\section{Pro-completions}\label{sec:pro-completions}
A category $\cat{1}$ is \textit{cofiltered} if every non-empty finite diagram in $\cat{1}$ has a cone.
A \textit{cofiltered limit} is the limit of a diagram shaped like a cofiltered small category. 
Fix a category $\cat{1}$.  
Intuitively, the \textit{pro-completion} (or \textit{pro-category}) $\PROOBJECTS{\cat{1}}$ of $\cat{1}$ is a formal completion of $\cat{1}$ with respect to cofiltered limits.
The pro-completion $\PROOBJECTS{\cat{1}}$ of $\cat{1}$ always exists and is characterized up to categorical equivalence by the following universal property.  
Take $\PROOBJECTS{\cat{1}}$ to be the category containing $\cat{1}$ as a full subcategory, unique up to all such categories up to categorical equivalence under $\cat{1}$, closed under all cofiltered limits such that for each functor $F:\cat{1}\ra\modelcat{1}$, there exists a dotted functor to a category $\mathscr{M}$ closed under all cofiltered limits, unique up to natural isomorphism, preserving cofiltered limits and making the following commute:
\begin{equation*}
  \begin{tikzcd}
		\cat{1}\ar{r}[above]{F}\ar[d,hookrightarrow] & \modelcat{1}\\
		\PROOBJECTS{\cat{1}}\ar[ur,dotted]
	\end{tikzcd}
\end{equation*}

We can model a cube of infinite dimension as the $(\PROOBJECTS{\BOX})$-object
\begin{equation}
  \label{eqn:infinite.cube}
  \boxobj{\infty}=\lim\left(\cdots\xra{\sigma_{4;4}}\boxobj{3}\xra{\sigma_{3;3}}\boxobj{2}\xra{\sigma_{2;2}}\boxobj{1}\ra[0]\right).
\end{equation}

Concretely, $\PROOBJECTS{\cat{1}}$ can be taken to be the category: whose objects are formal expressions of the form $\lim\,F$ for cofiltered diagrams $F$ in $\cat{1}$; whose hom-sets are defined by
$$\PROOBJECTS{\cat{1}}(\lim\,X,\lim\,Y)=\lim_{y}\mathop{\mathrm{colim}}_{x}\cat{1}(X(x),Y(y)),$$
where the colimit is taken over all objects in the domain of the cofiltered diagram $X$ in $\cat{1}$, the limit is taken over all objects in the domain of the cofiltered diagram $Y$ in $\cat{1}$; and whose composition is induced by composition on $\cat{1}$. 
Unravelling the construction, $(\PROOBJECTS{\cat{1}})$-morphisms $\lim\,X\ra\lim\,Y$ between limits of cofiltered diagrams $X$ and $Y$ in $\cat{1}$ determine and are determined by the following data.  
Consider a choice of $\mathcal{X}$-object $x_y$ and $\cat{1}$-morphism $\zeta_y:X(x_y)\ra Y(y)$ for each choice $y$ of $\mathcal{Y}$-object such that for each $\mathcal{Y}$-morphism $\upsilon:y_1\ra y_2$, there exist $\mathcal{X}$-morphisms $\chi_1:x\ra x_{y_1}$ and $\chi_2:x\ra x_{y_2}$ making the left of the diagrams
\begin{equation*}
  \begin{tikzcd}
	  X(x)\ar{r}[above]{X(\chi_1)}\ar{d}[left]{X(\chi_2)} & X(x_{y_1})\ar{r}[above]{\zeta_{y_1}} & Y(y_1)\ar{d}[right]{Y(\upsilon)}\\
	  X(x_{y_2})\ar{r}[below]{\zeta_{y_2}} & Y(y_2)\ar{r}[below]{\id_{Y(y_2)}} & Y(y_2)
  \end{tikzcd}
  \quad
  \begin{tikzcd}
	  \lim\,X\ar{r}[above]{\zeta}\ar[d] & \lim\,Y\ar[d]\\
	  X(x_{y})\ar{r}[below]{\zeta_{y}} & Y(y),
  \end{tikzcd}
\end{equation*}
commute.
Then there exists a unique $(\PROOBJECTS{\cat{1}})$-morphism $\zeta:X\ra Y$ making the right of the diagrams commute for each $\mathcal{Y}$-object $y$.
We call the $\PROOBJECTS{\cat{1}}$-morphism $\zeta:\lim\,X\ra\lim\,Y$ \textit{component-wise monic} if all the $\zeta_y$'s can be chosen to be monos in $\cat{1}$.  
The reader is referred elsewhere \cite{isaksen2002calculating} for details.

\begin{eg}
  Fix a cubical set $C$.  
  A $(\PROOBJECTS{\CUBICALSETS})$-morphism of the form
  $$C\ra\BOX\boxobj{\infty},$$
  from $C$ to the image of $\boxobj{\infty}$ (\ref{eqn:infinite.cube}) under the extension of the Yoneda embedding to a functor $\PROOBJECTS{\BOX[-]}:\PROOBJECTS{\BOX}\ra\PROOBJECTS{\CUBICALSETS}$, informally, is the data of a cubical function from $C$ to a representable up to cube projections.  
\end{eg}

For a pair of categories $\mathscr{A}$ and $\mathscr{B}$, there exists a natural categorical equivalence
$$\PROOBJECTS{(\mathscr{A\times B})}\simeq\PROOBJECTS{A}\times\PROOBJECTS{B}.$$

For each functor $\zeta:\smallcat{1}\ra\smallcat{2}$ of small categories, we also write $\zeta$ for the extension $\PROOBJECTS{\cat{1}}\ra\PROOBJECTS{\cat{2}}$, unique up to natural isomorphism, making the diagram above commute when $\modelcat{1}=\PROOBJECTS{\cat{2}}$ and $F=(\cat{2}\ira\PROOBJECTS{\cat{2}})\zeta$.  
Thus for each bifunctor $\otimes:\modelcat{1}^2\ra\modelcat{1}$ we also write $\otimes$ for the dotted bifunctor, unique up to natural isomorphism, making 
\begin{equation*}
  \begin{tikzcd}
	  \modelcat{1}^2\ar[d,hookrightarrow]\ar{r}[above]{\otimes} & \modelcat{1}\ar[d,hookrightarrow]\\
	  \PROOBJECTS{\modelcat{1}}\times\PROOBJECTS{\modelcat{1}}\ar[dotted]{r}[below]{\otimes} & \PROOBJECTS{\modelcat{1}}
  \end{tikzcd}
\end{equation*}
commute and preserving cofiltered limits.  
Uniqueness implies that the bottom bifunctor is associative and unital if the top bifunctor is associative and unital.  
Thus for each monoidal category $\modelcat{1}$, we will regard $\PROOBJECTS{\modelcat{1}}$ as a monoidal category whose tensor product is the extension of the tensor product on $\modelcat{1}$, unique up to natural isomorphism, that preserves cofiltered limits.  

\begin{eg}
  There exists a $(\PROOBJECTS{\BOX})$-isomorphism $\boxobj{\infty}\otimes\boxobj{\infty}\cong\boxobj{\infty}$.  
\end{eg}

\begin{thm:pro-diagrams}
  Fix categories $\indexcat{1}$ and $\cat{1}$ with $\indexcat{1}$ finite.  
  The quotient functor
  \begin{equation}
    \label{eqn:pro-diagram.lift}
    \PROOBJECTS{\left(\cat{1}^{\indexcat{1}}\right)}\ra\left(\PROOBJECTS{\cat{1}}\right)^{\indexcat{1}}
  \end{equation}
  is a categorical equivalence if $\indexcat{1}$ has only identities for its endomorphisms.
\end{thm:pro-diagrams}

\section{Diagram replacement}\label{sec:diagram.replacement}
Fix a category $\modelcat{1}$.
A typical construction in homotopy theory, for a given diagram $D$ in $\modelcat{1}$, is a diagram $D'$ whose colimit has some desired regularity condition and natural transformation $q_D:D'\ra D$ satisfying some desired right lifting property.
An example of such a construction $q_D$ is cofibrant replacement in a diagram model structure associated to a model structure on the base category $\modelcat{1}$.  
We are interested in the case where, intuitively, $D$ is a diagram of supports for restrictions of a stream map $f:\vec{\I}^{n_f}\ra\direalize{C_f}$ to small enough isothetic subcubes and the lifts need only hold up to natural directed homotopy.
The desired construction $q_D$ and the dotted lifts in the diagram
\begin{equation*}
  \xymatrix{
	  & **[r]\direalize{\colim\,QD}\ar[d]|{{\direalize{\;\colim\,q_D\;}}}\\
	  **[l]\direalize{\BOX\boxobj{n_f}}\ar[r]_{f}\ar@{.>}[ur] & **[r]\direalize{\colim\,D}=\direalize{C_f}
  }
\end{equation*}
should be natural in stream maps $f:\direalize{\BOX\boxobj{n_f}}\ra\direalize{C_f}$ and atomic subpresheaves of $\colim\,QD$ should be isomorphic to representables.
The starting point is the observation that $\epsilon^2_C:\sd_9C\ra C$ not only locally factors through representables [Lemma \ref{lem:local.lifts}], but in a manner where the representables vary monically at the cost of only being defined up to cube projections.  
Recall definitions of $\STARS_{k+1}$ [\S\ref{sec:cubical.sets}] and $\boxobj{\infty}$ [\S\ref{sec:pro-completions}].
Also recall our idiosyncratic notation $(C,S)$ for $\STARS_{k+1}$-objects $S\ira\sd_{k+1}C$ [\S\ref{sec:cubical.sets}].  

\begin{lem}
  \label{lem:star.rlp}
  There exist $(\PROOBJECTS{\BOX})$-object $Q(C,S)$ and dotted $(\PROOBJECTS{\CUBICALSETS})$-morphisms in
  \begin{equation*}
    \begin{tikzcd}
	    S\otimes\BOX\boxobj{\infty}\ar[dotted]{r}[above]{\alpha_{(C,S)}}\ar{d}[description]{(S\ira\sd_{27}C)\otimes\BOX[\boxobj{\infty}\ra[0]]} & \BOX[Q(C,S)]\ar[dotted]{d}[right]{\beta_{(C,S)}}\\
	    \sd_{27}C\ar{r}[below]{\epsilon^3_{C}} & C,
	\end{tikzcd}
  \end{equation*}
  all natural in $\STARS_{27}$ -objects $(C,S)$ such that the functor $Q:\STARS_{27}\ra\PROOBJECTS{\BOX}$ sends each object to a cofiltered $(\PROOBJECTS{\BOX})$-limit of $\BOX$-epis and each $\STARS_{27}$-mono to a component-wise monic $(\PROOBJECTS{\BOX})$-morphism.
\end{lem}

The proof uses the fact that parallel epis in $\BOX$ are always isomorphic to one another in the arrow category $\BOX^{[1]}$.
For this reason, the proof does not work for larger variants of $\BOX$ that include coconnections of one or both kinds.  

\begin{proof}
  Let $\ATOMICS_{k+1}$ denote the full subcategory of $\STARS_{k+1}$ consisting of those $\STARS_{k+1}$-objects $(C,A)$ with $A$ atomic.
  For each $\ATOMICS_{9}$-object $(C,A)$ with $C$ atomic, let $\mathcal{I}_{(C,A)}$ be the essentially small category whose objects are all $\ATOMICS_{9}$-objects $(D,B)$ with $D$ atomic, $C\subset D$, and $A\subset B$ and whose morphisms are all $\ATOMICS_{k+1}$-morphisms defined by inclusions of subpresheaves.
  For each $\ATOMICS_9$-object $(C,A)$, let 
  $$A_{C}=\epsilon_{\sd_3C}(\Star_{\sd_9C}(A)).$$

  \vspace{.1in}
  \textit{$\OP{\mathcal{I}_{(C,A)}}$ is cofiltered}: 
  Fix an $\ATOMICS_{9}$-object $(C,A)$.  
  Consider the solid diagram 
  \begin{equation*}
    \begin{tikzcd}
	    (C,A)\ar[r,hookrightarrow]\ar[d,hookrightarrow] & (C',A')\ar[d,hookrightarrow,dotted]\\
	    (C'',A'')\ar[r,dotted,hookrightarrow] & (D,B)
    \end{tikzcd}
  \end{equation*}
  in $\mathcal{I}_{(C,A)}$.
  To show $\OP{\mathcal{I}_{(C,A)}}$ is cofiltered, it suffices to show that there exists an $\mathcal{I}_{(C,A)}$-object $(D,B)$ and dotted $\mathcal{I}_{(C,A)}$-morphisms, necessarily making the diagram commute by $\mathcal{I}_{(C,A)}$ equivalent to a poset.
  In the case that $A,A',A'',C,C',C''$ are isomorphic to representables of dimension at most $1$, the desired $(D,B)$ exists by inspection.  
  In the more general case where $A,A',A'',C,C',C''$ are isomorphic to representables, the desired $(D,B)$ exists from the previous case, the fact that $\sd_3$ is monoidal, and the fact that tensor products of cubical sets preserve monos.  
  In the general case, $A,A',A'',C,C',C''$ are atomic, the above solid diagram lifts to a diagram between representables and hence the desired $B$ and $D$ can be obtained by quotienting the $B$ and $D$ obtained from the previous case.

  \vspace{.1in}
  \textit{construction of $\alpha_{(C,A)},\beta_{(C,A)}$}: 
  Let $(C,A)$ denote an $\ATOMICS_9$-object.
  There exists a unique minimal subpresheaf $C_A\subset C$, atomic by minimality, with $A_C\cap\sd_3C_A\neq\varnothing$ [Lemmas \ref{lem:collapse.star} and \ref{lem:star.flower}].  
  Let $n_{(C,A)}=\dim\,C_A$.  
  There exists a choice $\sigma_{(C,A)}$ of cubical function, unique up to $(\BOX/C)$-isomorphism, having image $C_A$ and of the following form by $C_A$ atomic:
  $$\sigma_{(C,A)}:\BOX\boxobj{n_{(C,A)}}\ra C.$$

  For each $\ATOMICS_9$-morphism $\psi:(D',B')\ra(D'',B'')$, 
  $$(\sd_9\psi)(D')\cap\Star_{\sd_9D''}B''\subset(\sd_9\psi)(D')\cap\Star_{\sd_9D''}\psi(B')=(\sd_9\psi)(\Star_{\sd_9D'}B'),$$
  hence $(\sd_3\psi)(D')\cap B''_{D''}\subset(\sd_3\psi)(B'_{D'})$, and hence $\psi$ restricts and corestricts to a cubical function $D'_{B'}\ra D''_{B''}$ by minimality.  

  Consider the special case where $\psi$ is a $\mathcal{I}_{(C,A)}$-morphism.  
  Then we can conclude $C_A\subset D'_{B'}\subset D''_{B''}$.   
  The cubical set $B'_{D'}\cap\sd_3D''_{B''}$ is an atomic subpresheaf of $\sd_3D''_{B''}$ by $B'_{D'}$ and $D''_{B''}$ both atomic.
  The top dimensional cube in the atomic cubical set $B'_{D'}\cap\sd_3D''_{B''}$ is not a cube in $\sd_3\partial D''_{B''}$ because $B'_{D'}\cap\sd_3D''_{B''}$ contains an atomic subpresheaf, $B''_{D''}\cap\sd_3D''_{B''}$, which does not intersect $\sd_3\partial D''_{B''}$ by minimality of $D''_{B''}$.
  Therefore $B'_{D'}\cap\sd_3D''_{B''}$ has a unique atomic preimage under $\sd_3\sigma_{(D'',B'')}$.  
	Therefore there exists a minimal and hence atomic subpresheaf $I(D'',B'',B')\subset\BOX\boxobj{n_{(D'',B'')}}$ with $\sd_3I(D'',B'',B')$ intersecting the preimage of $B'_{D'}\cap\sd_3D''_{B''}$ under $\sd_3\sigma_{(D'',B'')}$ [Lemma \ref{lem:star.flower}].  
  The image of $I(D'',B'',B')$ under $\sigma_{(D'',B'')}$ is $D'_{B'}$ by minimality. 
  Thus there exist $\BOX$-epi $\sigma_{(D'',B'',B')}$ and $\BOX$-mono $\delta_{(D'',B'',B')}$, unique up to isomorphism, making II in the diagram
  \begin{equation*}
	\begin{tikzcd}
	  \BOX\boxobj{n_{(D'',B'',A)}}\ar[rr,above,"\sigma_{(D'',B'',B',A)}"]\ar{d}[description]{\delta_{(D'',B'',B',A)}} 
          \ar[drr,phantom,near start,"III"]
	  & & \BOX\boxobj{n_{(D',B',A)}}\ar{rr}[above]{\sigma_{(D',B',A)}}
	  \ar{d}[description]{\delta_{(D',B',A)}} & & \BOX\boxobj{n_{(C,A)}}
	  \ar[d,phantom,near start,"I"]
		\ar{r}[above]{\sigma_{(C,A)}} & C\ar[d,hookrightarrow]
	    \\
		\BOX\boxobj{n_{(D'',B'',B')}}\ar{rr}[description,twoheadrightarrow]{\sigma_{(D'',B'',B')}}\ar{d}[description]{\delta_{(D'',B'',B')}} 
	    & & \BOX\boxobj{n_{(D',B')}}
	    \ar[d,phantom,near start,"II"]
		\ar{rrr}[description]{\sigma_{(D',B')}} & {} & {} & D'\ar[d,hookrightarrow]
	    \\
		\BOX\boxobj{n_{(D'',B'')}}\ar{rrrrr}[below]{\sigma_{(D'',B'')}} & {} & {} & {} & {} & D''
	\end{tikzcd}
  \end{equation*}
  commute. 
  Moreover, we can take $\sigma_{(D'',B'',B')}$ to be the $\BOX$-epi
  $$\sigma_{(D'',B'',B')}=\boxobj{n_{(D',B')}}\otimes(\boxobj{n_{(D'',B'',B')}-n_{(D'',B'')}}\ra[0])$$
  without loss of generality, because all parallel $\BOX$-epis are isomorphic to one another in the arrow category $\BOX^{[1]}$.  
  An application of the previous observations to the case $A=B'$ gives that there exists a $\BOX$-mono $\delta_{(D',B',A)}$ making I in the diagram above commute.  
  The cubical set $I(D'',B'',A)$ is a subpresheaf of $I(D'',B'',B')$ by minimality.  
  The image of $I(D'',B'',A)$ under the composite of the arrows in the middle row is $D''_{B''}$ by minimality.  
  Thus there exist $\BOX$-epi $\sigma_{(D'',B'',B',A)}$ and $\BOX$-mono $\delta_{(D'',B'',B',A)}$, unique up to isomorphism, making III above commute.
  We can take $\sigma_{(D'',B'',B',A)}$ to be the $\BOX$-epi
  $$\sigma_{(D'',B'',B',A)}=\boxobj{n_{(D'',B'',B'}}\otimes(\boxobj{n_{(D'',B'',A)}-n_{(D'',B'',B')}}\ra[0])$$
  without loss of generality, again because parallel $\BOX$-epis are all isomorphic to one another in the arrow category $\BOX^{[1]}$.  
  
  We can let $P(C,A)$ be the $(\PROOBJECTS{\BOX})$-limit of the diagram 
  $$\OP{\mathcal{I}_{(C,A)}}\ra\BOX$$
  sending each object $(D,B)$ to $\boxobj{n_{(D,B,A)}}$ and each morphism $(D',B')\ra(D'',B'')$ to the $\BOX$-epi $\sigma_{(D'',B'',B',A)}$ for each $\ATOMICS_{9}$-object $(C,A)$, by the domain of the diagram small and cofiltered.

  To regard $P$ as functorial in $(C,A)$, it suffices to take $C$ atomic because $P(C,A)=P(\support_{\sd_9}(C,A),A)$ by minimality.
  To that end, consider an $\ATOMICS_9$-morphism $\psi:(C',A')\ra(C'',A'')$ defined by a cubical function $\psi:C'\ra C''$ between atomic cubical sets.
  We claim there exists a choice of ${\mathcal{I}_{(C',A')}}$-object $(D',B')$ and $\ATOMICS_9$-morphism, natural in ${\mathcal{I}_{(C',A')}}$-objects $(D'',B'')$, extending $\psi:(C',A')\ra(C'',A'')$ and mapping $D'$ onto $D''$ and $B'$ onto $B''$ such that $D'=D''$ and $B'=B''$ in the case $\psi$ is an inclusion.  
  The desired extension can be constructed for: $A',A'',B'',C',C'',D''$ representables of dimension at most $1$ by inspection; for $A',A'',B'',C',C'',D''$ representables by taking tensor products; then for $C,C',D''$ atomic by taking quotients of the extension obtained from the previous case. 
  This extension restricts and corestricts to a cubical function $D'_{B'}\ra D''_{B''}$, whose composite with $\sigma_{(D',B')}$ uniquely lifts to a cubical function of the form $\BOX[\boxobj{n_{(D',B')}}\ra\boxobj{n_{(D'',B'')}}]$.
  Composition of this latter cubical function with inclusion $I(D',B',A')\ira\BOX\boxobj{n_{(D',B')}}$ corestricts to a cubical function $I(D',B',A')\ra I(D'',B'',A'')$ by a minimality argument.  
  In the case $\psi$ is an inclusion/identity, then we can take $D'=D''$ and $A'=A''$ so that the latter cubical function is an inclusion.
  Thus there exist $\BOX$-morphisms $\boxobj{n_{(D',B',A')}}\ra\boxobj{n_{(D'',B'',A'')}}$, $\delta_{(D'',B'',B',A')}$ in the special case $\psi$ is inclusion, defining a component-wise monic $(\PROOBJECTS{\BOX})$-morphism $P\psi:P(C',A')\ra P(C'',A'')$.
  The construction $P$ can be shown to preserve identities and composites of $\ATOMICS_{9}$-morphisms by the aforementioned uniqueness properties.

  Let $(C,S)$ denote an $\STARS_{27}$-object.
  Let $E_S=\epsilon_{\sd_9C}(S)$. 
  Then the rule 
  $$Q(C,S)=P(C,E_S)$$
  naturally defines a functor, of the form
  $$Q:\STARS_{27}\ra(\PROOBJECTS{\BOX}),$$
  sending monos to component-wise monic morphisms.
  Consider the solid diagram
  \begin{equation*}
    \begin{tikzcd}
	    S\otimes\BOX\boxobj{i_{(D,B,E_S)}}\ar[dotted]{rrr}[above]{\epsilon_{(C,S)}\otimes\BOX\boxobj{i_{(D,B,E_S)}}}\ar[dotted]{d}[description]{S\otimes(\BOX\boxobj{i_{(D,B,E_S)}}\ra\star)}
		    & & & E_S\otimes\BOX\boxobj{i_{(D,B,E_S)}}\ar[dotted]{rrr}[above]{\lambda_{(C,E_S)}\otimes\BOX\boxobj{i_{(D,B,E_S)}}}\ar[dotted]{d}[description]{E_S\otimes(\BOX\boxobj{i_{(D,B,E_S)}}\ra\star)}
		    & & & \BOX\boxobj{n_{(D,B,E_S)}}\ar[dotted]{d}[description,twoheadrightarrow]{\sigma_{(D,B,E_S)}}
	\\
	S\ar[dotted]{rrr}[description]{\epsilon_{(C,S)}}\ar[d,hookrightarrow]
	& & & E_S\ar[dotted]{rrr}[description]{\lambda_{(C,E_S)}}\ar[d,hookrightarrow]
	& & & \BOX\boxobj{n_{(C,E_S)}}\ar{d}[description]{\sigma_{(C,E_S)}}
	\\
	\sd_{27}C\ar{rrr}[below]{\epsilon_{\sd_9C}}
	& & & \sd_9C\ar{rrr}[below]{\epsilon^2_{C}}
	& & & C,
    \end{tikzcd}
  \end{equation*}
  There exist dotted epi $\epsilon_{(C,S)}$, unique by the lower vertical inclusion in the middle column monic and hence natural in $(C,S)$, making the lower left rectangle above commute because restrictions of cubical functions corestrict onto their images.  
  There exists a dotted cubical function $\lambda_{(C,S)}$, natural in $(C,S)$, making the lower right rectangle above commute [Lemma \ref{lem:local.lifts}].
  The desired $\alpha_{(C,S)},\beta_{(C,S)}$ can be represented by composites of the respective top horizontal and right vertical arrows in the commutative diagram above, where $(D,B)$ denote an $\mathcal{I}_{(C,A)}$-object and $i_{(D,B,E_S)}=n_{(D,B,E_S)}-n_{(C,E_S)}$.  
\end{proof}

\begin{lem}
  \label{lem:bamfl}
  Let $f$ denote a $(\direalize{\BOX[-]}/\direalize{\sd_{27}-})$-object as in the diagram
  \begin{equation*}
    \begin{tikzcd}
	    \direalize{\BOX\boxobj{n_f}}\ar{d}[left]{f}\ar[dotted]{r}[above]{f^*} & \direalize{C^*_{f}}\ar[dotted]{d}[right]{\direalize{\;\Lambda^*_{f}\;}}\\
	    \direalize{\sd_{27}C_f}\ar{r}[below]{\direalize{\;\epsilon^3_{C_f}\;}} & \direalize{C_f}
    \end{tikzcd}
  \end{equation*}
  There exist dotted $\PROOBJECTS{(\REGULARCUBICALSETS/C_f)}$-object $\Lambda^*_{f}$ and $\PROOBJECTS{(\DITOP)}$-morphism $f^*$, both natural in $f$, making the diagram above commute.  
\end{lem}

Like Reedy cofibrant replacement, the diagram replacement made in the proof relies on the acyclicity of the diagram shape.
Again recall our idiosyncratic notation $(C,A)$ for $\STARS_{k+1}$-objects $A\ira\sd_{k+1}C$ [\S\ref{sec:cubical.sets}].  

\begin{proof} 
  Let $\mathcal{I}_f$ be the poset of non-empty preimages under $f$ of open stars.  
  Then 
  $$\catfont{S}_fX=(C_f,\support_{|-|}(f(X),\sd_{27}C_f))$$
  is a $\STARS_{27}$-object, for each $\mathcal{I}_f$-object $X$.  
  Thus $\catfont{S}_f$ defines a functor $\mathcal{I}_f\ra\STARS_{27}$ natural in $f$.

  There exists a $(\PROOBJECTS{\BOX})$-morphism $\delta_\infty:\star\ra\boxobj{\infty}$ induced by minima-preserving functions of the form $\star\ra\boxobj{n}$ because $\BOX$-epis preserve minima. 
  The restriction of the stream map $f$ to each $\mathcal{I}_f$-object $X$ corestricts to a stream map $f_X:X\ra\direalize{\catfont{S}_fX}$ [Proposition \ref{prop:realization.preserves.embeddings}].
  There exist functor $\catfont{Q}:\STARS_{27}\ra\PROOBJECTS{\BOX}$ sending each $\STARS_{27}$-morphism defining an inclusion of cubical sets to a component-wise monic $(\PROOBJECTS{\BOX})$-morphism and $(\PROOBJECTS{\CUBICALSETS})$-morphisms $\alpha_{X},\beta_{X}$, natural in $\mathcal{I}_f$-objects $X$, making the outer rectangle below commute [Lemma \ref{lem:star.rlp}].  
  \begin{equation*}
    \begin{tikzcd}
	    X\ar[dd,hookrightarrow]\ar[drr]\ar{rr}[above]{\direalize{\;\delta_\infty\;}(X\ra\star)\times f_X} & & \direalize{\BOX\boxobj{\infty}\otimes \catfont{S}_fX}\ar{rr}[above]{\direalize{\;\alpha_{X}\;}} & & \direalize{\BOX[\catfont{Q}(C_f,\catfont{S}_fX)]}\ar[dll]\ar{dd}[right]{\direalize{\;\beta_{X}\;}}
	    \\
	    & & \direalize{C_f^*}\ar{drr}[description]{\direalize{\;\Lambda_f^*\;}}
	    \\
	    \direalize{\BOX\boxobj{n_f}}\ar{rr}[below]{f}\ar[dotted]{urr}[description]{f^*} & & \direalize{\sd_{27}C_f}\ar{rr}[below]{\direalize{\;\epsilon^3_{C_f}\;}} & & \direalize{C_f}
    \end{tikzcd}
  \end{equation*}

  The composite $\catfont{Q}\catfont{S}_f$ lifts under the natural quotient functor $\PROOBJECTS{(\BOX^{\mathcal{I}_f})}\ra(\PROOBJECTS{\BOX})^{\mathcal{I}_f}$ to a $(\PROOBJECTS{\BOX^{\mathcal{I}_f}})$-object $D^*_f$ by $\mathcal{I}_f$ finite and acyclic. 
  Observe that $D^*_f$ is a cofiltered limit of diagrams $\mathcal{I}_f\ra\BOX$ sending each morphism to a mono by our choice of $\catfont{Q}$.  
	It therefore follows that $C_f^*=\colim\,\BOX[D^*_f]$ is a cofiltered limit in $\PROOBJECTS{(\REGULARCUBICALSETS)}$ because a colimit of representables and inclusions between them indexed over a poset is a $\REGULARCUBICALSETS$-object.
  Let $\Lambda_f^*$ denote the cubical function $C^*_f\ra C_f$ induced from the natural cocone $\BOX[D^*_f]\ra C_f$.  
  Commutative rectangles above natural in $\mathcal{I}_f$-objects $X$, in which the unlabelled solid arrows are cannonically defined, induce the desired dotted lift.  
\end{proof}

\begin{proof}[proof of Lemma \ref{lem:homotopy.idempotent.comonad}]
  Let $r_C$ be the stream map, natural in cubical sets $C$, defined by
  $$r_C=\direalize{\epsilon^3_{C}}\dihomeo^{-1}_{C;27}:\direalize{C}\ra\direalize{C}.$$ 

	Let $\theta$ denote a $(\BOX[-]/\catfont{S})$-object and define $n_\theta$ and $C_\theta$ so that
	$$\theta:\BOX\boxobj{n_\theta}\ra\catfont{S}C_\theta.$$

	There exist dotted $(\PROOBJECTS{(\REGULARCUBICALSETS/\CUBICALSETS)})$-object $\Lambda^*_{\theta}:B_\theta\ra C_\theta$ and $(\PROOBJECTS{\CUBICALSETS})$-morphism 
	$$\theta^*:\BOX\boxobj{n_\theta}\ra\catfont{S}B_\theta,$$
	both natural in $\theta$, such that the following holds [Lemma \ref{lem:bamfl}]:
	\begin{equation}
	  \label{eqn:idempotent.lift}
	  \sing(r_{C_\theta})\theta=\catfont{S}(\Lambda^*_\theta)\theta^*. 
	\end{equation}
	    
        There exist $\direalize{\mathfrak{d}}_*$-homotopies of the following forms
	\begin{align*}
		\direalize{\theta}\dhomotopic{\direalize{\;\mathfrak{d}\;}}&\direalize{\sing(r_{C_\theta})\theta}
		=\direalize{\catfont{S}(\Lambda_\theta^*)\theta^*}\\
		\dhomotopic{\direalize{\;\mathfrak{d}\;}}&\direalize{\catfont{S}(\Lambda_\theta^*)\eta_{B_\theta}}\epsilon_{\direalize{\;B_\theta\;}}\direalize{\theta^*}
		=\direalize{\sing(r_{C_\theta})}\direalize{\eta_{C_\theta}}\epsilon_{\direalize{\;C_\theta\;}}\direalize{\theta}\\
		\dhomotopic{\direalize{\;\mathfrak{d}\;}}&\direalize{\eta_{C_\theta}}\epsilon_{\direalize{\;C_\theta\;}}\direalize{\theta}
	\end{align*}
        for the following reasons.
	The first and last $\direalize{\mathfrak{d}}_*$-homotopies exist by the existence of $\direalize{\mathfrak{d}}_*$-homotopies $r_{C_\theta}\dihomotopic{\id_{\direalize{\;C_\theta\;}}}$ natural in $\theta$ [Lemma \ref{lem:natural.approximations}].
	The first and last equalities exist by (\ref{eqn:idempotent.lift}).
	The middle $\direalize{\mathfrak{d}}_*$-homotopy exists by Lemma \ref{lem:restricted.homotopy.idempotent.comonad}.
	These $\direalize{\mathfrak{d}}_*$-homotopies, natural in $(\BOX/\catfont{S}C)$-objects $\theta$, induce the desired $\direalize{\mathfrak{d}}_*$-homotopies natural in cubical sets $C$.  
\end{proof}

\bibliography{gv}{}
\bibliographystyle{plain}

\end{document}